\documentclass[12pt]{amsart}
\usepackage{amsthm,amsmath,amssymb}
\usepackage{lmodern}
\usepackage[colorlinks=true,citecolor=blue,linkcolor=black,urlcolor=blue]{hyperref}

\usepackage{geometry}
 \usepackage{ytableau}
\ytableausetup{boxsize=1em}
\usepackage{graphicx,enumitem}
\usepackage{natbib}
\usepackage{color}
\usepackage{wasysym}

\numberwithin{equation}{section}
\numberwithin{figure}{section}

\theoremstyle{plain} 
\newtheorem{thm}{Theorem}
\newtheorem{lem}[thm]{Lemma}

\theoremstyle{remark}

\theoremstyle{plain}

\newcommand{\M}{\operatorname{M}}
\newcommand{\T}{\operatorname{T}}
\newcommand{\Hf}{\operatorname{H}}
\newcommand{\wt}{\operatorname{wt}}

\begin{document}

\title{Domino tilings of cruciform regions}

\author{Tri Lai\textsuperscript{1}}
\address{$^1$Department of Mathematics, University of Nebraska -- Lincoln, NE 68588, U.S.A.}
\email{tlai3@unl.edu}
\thanks{This research was supported in part by Simons Collaboration Grant (\# 585923).}
\author{Anh Thi Nguyen\textsuperscript{2,3}}
\address{$^2$ Faculty of Mathematics and Computer Science, University of Science, Ho Chi Minh City, Vietnam.}
\address{$^{3}$ Vietnam National University, Ho Chi Minh City, Vietnam}
\email{nathi@hcmus.edu.vn}
\thanks{This research is funded by Vietnam National University, Ho Chi Minh City (VNU-HCM) under grant number T.C2025-18-08.}



\date{\today}

\dedicatory{}

\begin{abstract}
P. Di Francesco first introduced the “Aztec triangle” in his study of the relationship between the twenty-vertex model and domino tilings. He conjectured an exact formula for the number of tilings of the Aztec triangle, and it has since been proved by several authors. In an attempt to prove the conjecture, M. Ciucu showed that the tiling number of the Aztec triangle divides the tiling number of a new region called the “cruciform region,” a superposition of two Aztec rectangles. Ciucu proved that the number of domino tilings of a cruciform region is given by a simple product formula. In this paper, we generalize Ciucu's tiling formula by providing a generating-function formula for the cruciform region.
\end{abstract}

\maketitle

\section{Introduction}\label{Sec:Intro}
The enumeration of tilings concerns the number of ways to cover lattice regions by tiles without gaps or overlaps. The study dates back to the early 1900s, when 
 MacMahon proved his plane partition formula, equivalently a tiling formula of a hexagon (see \cite{Mac}). Other classical results include Kasteleyn, Temperley, and Fisher’s formula for domino tilings of rectangles (see \cite{Kas}, \cite{Fish}) and Elkies, Kuperberg, Larsen, and Propp's \emph{Aztec diamond theorem} (see \cite{Elkies1, Elkies2}). In our paper, we study the \emph{cruciform regions} first introduced by Mihai Ciucu \cite{Ciucucrucify}.

As a natural generalization of the well-known Aztec diamond, the \emph{Aztec rectangle (region)} $\mathcal{AR}_{m,n}$ is a $45^{\circ}$-rotated $m\times n$ rectangle with a staircase boundary. See the left picture in Figure \ref{C}, where $\mathcal{AR}_{16,6}$ is shown in solid red and $\mathcal{AR}_{9,15}$ is outlined by the bold contour.  A \emph{cruciform region} is formed by overlapping the two Aztec rectangles $\mathcal{AR}_{m+b+d+1,n}$  and $\mathcal{AR}_{m,n+a+c+1}$, denoted $\mathcal{C}^{a,b,c,d}_{m,n}$  (see the right picture in Figure \ref{C}). A cruciform region $\mathcal{C}^{a,b,c,d}_{m,n}$ is \emph{balanced} if it admits a domino tiling. This occurs if and only if $a+b+c+d=m+n-1$,  $\max(a, c)\le m$, and $\max(b,d)\le n$. 

\begin{figure}[ht]\centering
		\includegraphics[width = 15cm]{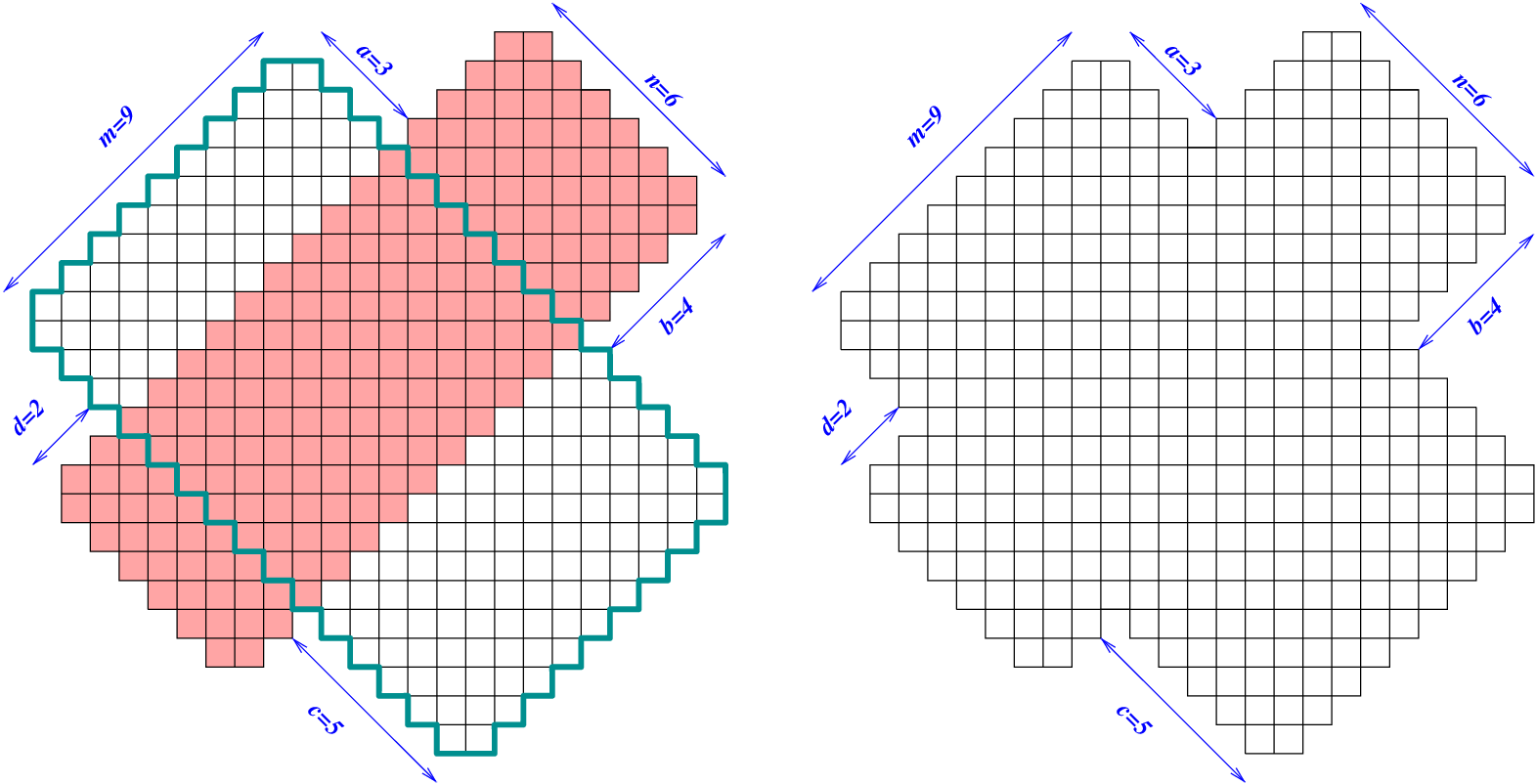}
		\caption{The Cruciform region $C^{a,b,c,d}_{m,n}=C_{9,6}^{3,4,5,2}.$}
		\label{C}
\end{figure}

Inspired by connections between the twenty-vertex model and domino tilings in \cite{20V},  Di Francesco conjectured \cite[Conjecture 8.1]{20V2}  that the number of domino tilings of the so-called \emph{Aztec triangle} is given by a product formula similar to that for the alternating sign matrices (see, e.g., \cite{Zie1, Zie2,Kup}). As partial progress toward this conjecture, Mihai Ciucu derived a formula for the number of tilings of cruciform regions, using a complementation theorem for perfect matchings of graphs with a cellular completion (see \cite{Ciucu97})\footnote{Together with B. Seok Hyun, Ciucu has recently provided a full combinatorial proof of Di Francesco's conjecture in \cite{Ciucucrucify2}.}.
Generalizing Ciucu's work, we investigate a weighted enumeration of the tilings as follows.

We give a chessboard coloring for the unit squares of the cruciform region and place a coordinate system so that the origin is at the center of a white unit square and that the centers of other unit squares are all at integer points as in Figure \ref{ColoredDomino}. We call a unit square an \emph{$(i,j)$-square} if the coordinates of its center are $(i,j)$, for some integers $i$ and $j$. 

There are four types of colored dominoes: \emph{even vertical, odd vertical, even horizontal, and odd horizontal}.   We give weights to these four types of dominoes as follows. Let $e,f,g,h$ be four nonzero real numbers, and let $q$ be an indeterminate. Each odd vertical domino is weighted by $e$, each even horizontal domino is weighted by $f$, an even vertical domino containing a white $(i,j)$-square is weighted by $g\cdot q^{i+j}$, and the odd horizontal domino containing a white $(i,j)$-square is weighted by $h\cdot q^{i+j}$. The weight $\wt(\tau)$ of a tiling $\tau$ is defined to be the product of the weights of its dominoes.

\begin{figure}[ht]\centering
			\includegraphics[width = 13cm]{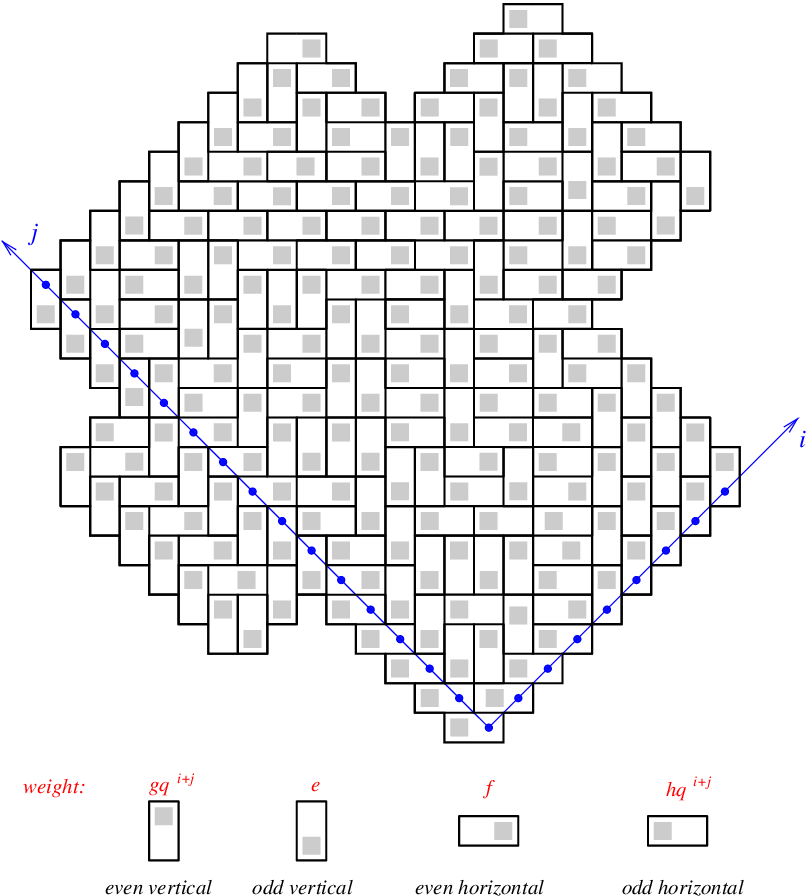}
		\caption{The weight assignment of domino tilings in a cruciform region.}
		\label{ColoredDomino}
\end{figure}

We define the \emph{shifted $q$-factorials} by \[(a;q)_k:=(1-a)(1-aq)\cdots(1-aq^{k-1}),\] for $k\geq 1$ and $(a;q)_0:=1$. Define the \emph{hyper $q$-shifted factor} by \[\Hf_q(n):=\prod_{k=1}^{n-1} (q;q)_k,\] and $\Hf_q (1):=1$.

The tiling generating function of a cruciform region is given in the following theorem. 

\begin{thm} \label{main}
Let $a,b,c,d,m,n$ be positive integers, and let $e,f,g,h$ be four  nonzero real numbers that satisfy the condition 
	\[\frac{g}{h}=q^{-c}\frac{f}{e},\] 
If  $|c-a|$,  say $|c-a|=2s$, then the tiling generating function of the balanced cruciform region $\mathcal{C}^{a,b,c,d}_{m,n}$ is given by
\begin{align}\label{maineq1}
\sum_{\tau} \wt (\tau) &= 2^{m-c}(1-q)^s q^{Q}\notag\\
&\times e^{(n-b)(m-a)}f^{(m+n-c-d)(m-a)}g^{(n-d)(m-c)}h^{(d+c+1)(m-c)}\notag\\
&\times \prod_{i=1}^{c} (egq^{i-1}+hf)^{c-i+1}   \prod_{i=1}^{n} (eg+hfq^{i})^{n-i+1}  \prod_{i=1}^{a} (egq^{i-1}+hf)^{a-i+1}\notag\\
 &\times \prod_{i=1}^{2m-a-c} (q^i;q)_{n-d}(q^{2m-a-c-i+1};q)_{n-b}\cdot  \prod_{j=1}^{n-b} (q^{2m-a-c+j};q)_{n-d}^2 \notag\\
&\times\frac{\Hf_q(m-c)\Hf_q(m-a)\Hf_q(n-d)^2\Hf_q(n-b)^2\Hf_{q^2}(m-c+s)^2}{\Hf_q(n+a+1)\Hf_q(n+c+1)\Hf_q(m-c)\Hf_q(m-c+2s)}\notag\\
&\times \prod_{i=1}^{s} \frac{(q^{3};q^2)_{i-1}(q;q^2)_{i-1}(q^{2(i+s)};q^2)_{m-c}(q^2;q^2)_{i-1}}{(q^2;q^2)_{m-c+s-i}},
\end{align}
 where the sum on the left-hand side ranges over all domino tilings $\tau$ of $\mathcal{C}^{a,b,c,d}_{m,n}$, and where the exponent of $q$ is given by
\begin{align}  
Q%%&\binom{c}{2}+\binom{c+1}{3}+\binom{a}{2}+\binom{a+1}{3}+(m+b+d)\binom{n+1}{2}+(n+c+1)\binom{a+1}{2}+(c-d)\binom{n+1}{2}\notag\\
%%&+\frac12(n+c+1)(m-c)(n+m-2d-3)-(n+1)(a+d+1)(m-a)+(n+1)(m-a)(m-c)\notag\\
%%&+(m-a)(m-c)(n-b)+(2m-a-c)\binom{n-b+1}{2}\notag\\
%%&+\frac{(m-c)(m-c+1)(2m-2c+6s+1)}{6}-\frac{s(s-1)(s-2)}{3}\\
&=\frac{5 a}{6} + 2 a^2 + \frac{a^3}{6} + \frac{a b}{2} - \frac{a b^2}{2} + \frac{2 c}{3} + 
 \frac{3 a c}{2} + \frac{a^2 c}{2} + \frac{b c}{2} - a b c\notag\notag\\
 &- \frac{b^2 c}{2} + \frac{5 c^2}{2} - \frac{c^3}{6} + a d + c d + c^2 d - \frac{7 m}{3} - 2 a m - b m + 
 a b m + b^2 m \notag\\
 & - 4 c m + b c m + \frac{c^2 m}{2} - 2 d m - c d m + 2 m^2 - 
 b m^2 - \frac{c m^2}{2} + \frac{m^3}{3} + a n \notag\\
 &+ \frac{3 a^2 n}{2} + \frac{b n}{2} + a b n + 
 c n + 2 a c n + b c n - \frac{c^2 n}{2} + a d n + c d n - \frac{m n}{2} \notag\\
 &- 3 a m n - 2 b m n - 2 c m n - 2 d m n + \frac{5 m^2 n}{2} - \frac{a n^2}{2} + \frac{b n^2}{2} - \frac{c n^2}{2} + 2 m n^2\notag\\
 & - \frac{2 s}{3} - c s + c^2 s + m s - 
 2 c m s + m^2 s + s^2 - \frac{s^3}{3}.
 \end{align}
	 If   $|c-a|=2s+1$, then we have
\begin{align}\label{maineq2}
\sum_{\tau} \wt (\tau)&=2^{m-c}(1-q)^{2s} q^{Q+\binom{m-c+1}{2}}\notag\\
&\times e^{(n-b)(m-a)}f^{(m+n-c-d)(m-a)}g^{(n-d)(m-c)}h^{(d+c+1)(m-c)}\notag\\
&\times \prod_{i=1}^{c} (egq^{i-1}+hf)^{c-i+1}   \prod_{i=1}^{n} (eg+hfq^{i})^{n-i+1}  \prod_{i=1}^{a} (eg q^{i-1}+hf)^{a-i+1}\notag\\
 &\times\prod_{i=1}^{2m-a-c} (q^i;q)_{n-d}(q^{2m-a-c-i+1};q)_{n-b}\cdot  \prod_{j=1}^{n-b} (q^{2m-a-c+j};q)_{n-d}^2 \notag\\
&\times\frac{\Hf_q(m-c)\Hf_q(m-a)\Hf_q(n-d)^2\Hf_q(n-b)^2\Hf_{q^2}(m-c+s)H_{q^2}(m-c+s+1)}{\Hf_q(n+a+1)\Hf_q(n+c+1)\Hf_q(m-c)\Hf_q(m-c+2s+1)}\notag\\
& \times \prod_{i=1}^{s} \frac{(q^{3};q^2)_{i-1}(q^3;q^2)_{i-1}(q^{2(i+s+1)};q^2)_{m-c}(q^2;q^2)_{i-1}}{(q^2;q^2)_{m-c+s-i}}.
\end{align}

\end{thm}

The remainder of the paper is organized as follows. Section \ref{Sec:Prelim} reviews local graph transformations and Schur function identities used in the proof. Section  \ref{Sec:ProofMatching} contains the proof of the main theorem. Section \ref{Sec:Complete} provides detailed proofs of the auxiliary results stated in Section \ref{Sec:Prelim}.

\section{Preliminaries}\label{Sec:Prelim}

\subsection{Local Graphical Transformations}

Consider an infinite lattice on the plane, such as the square lattice or the triangular lattice. The \emph{fundamental regions} in the square lattice and the triangular lattice are the unit squares and the equilateral triangles of side length 
$1$, respectively; we call the latter \emph{unit triangles}. The union of any two fundamental regions sharing an edge is called a \emph{tile}, for example, the domino in the square lattice and the unit rhombus (or unit lozenge) in the triangular lattice. A \emph{tiling} of a lattice region is a way to cover the region by tiles so that there are no gaps or overlaps. In the case of the square lattice, we call the tilings the \emph{domino tilings}, and the \emph{lozenge tilings} in the case of the triangular lattice. See Jim Propp's survey papers \cite{Propp, Propp2} for more aspects of tilings. Each tile can carry a \emph{weight}, and the weight of a tiling is defined to be the product of the weights of its tiles.
We denote by $\T(\mathcal{R})$ the sum of weights of all tilings of the region $\mathcal{R}$. If $\mathcal{R}$ does not accept any tiling, we set $\T(\mathcal{R})=0$. In the unweighted case, i.e., each tiling has weight $1$, $\T(\mathcal{R})$ is exactly the number of tilings of $\mathcal{R}$.

The \emph{(planar) dual graph} $G$ of a region $\mathcal{R}$ is the graph whose vertices are the fundamental regions in $R$ and whose edges connect exactly two fundamental regions sharing an edge. For example, the dual graph of the cruciform region $\mathcal{C}_{9,6}^{3,4,5,2}$ in Figure \ref{C} is shown in Figure \ref{DC}.  In this paper, we always consider the $45^{\circ}$-rotated version of the dual graph of a cruciform region. We denote by $C_{m,n}^{a,b,c,d}$ the dual graph of the cruciform region $\mathcal{C}_{m,n}^{a,b,c,d}$, and we call $C_{m,n}^{a,b,c,d}$ a \emph{cruciform graph}. A \emph{(perfect) matching} in a graph is a collection of vertex-disjoint edges that covers all vertices of the graph. If the edges of $G$ carry weight, we define the weight of a matching to be the product of the weights of all its edges. Then $\M(G)$ denotes the sum of weights of all perfect matchings in $G$. There is a natural weight-preserving bijection between the tilings of a region $\mathcal{R}$ and the perfect matchings of its dual graph $G$, i.e., $\M(G)=\T(\mathcal{R})$. In certain cases, it is convenient to enumerate the matchings of the dual graph rather than the tilings of the region directly. 
\begin{figure}[ht]\centering
		\includegraphics[width = 10cm]{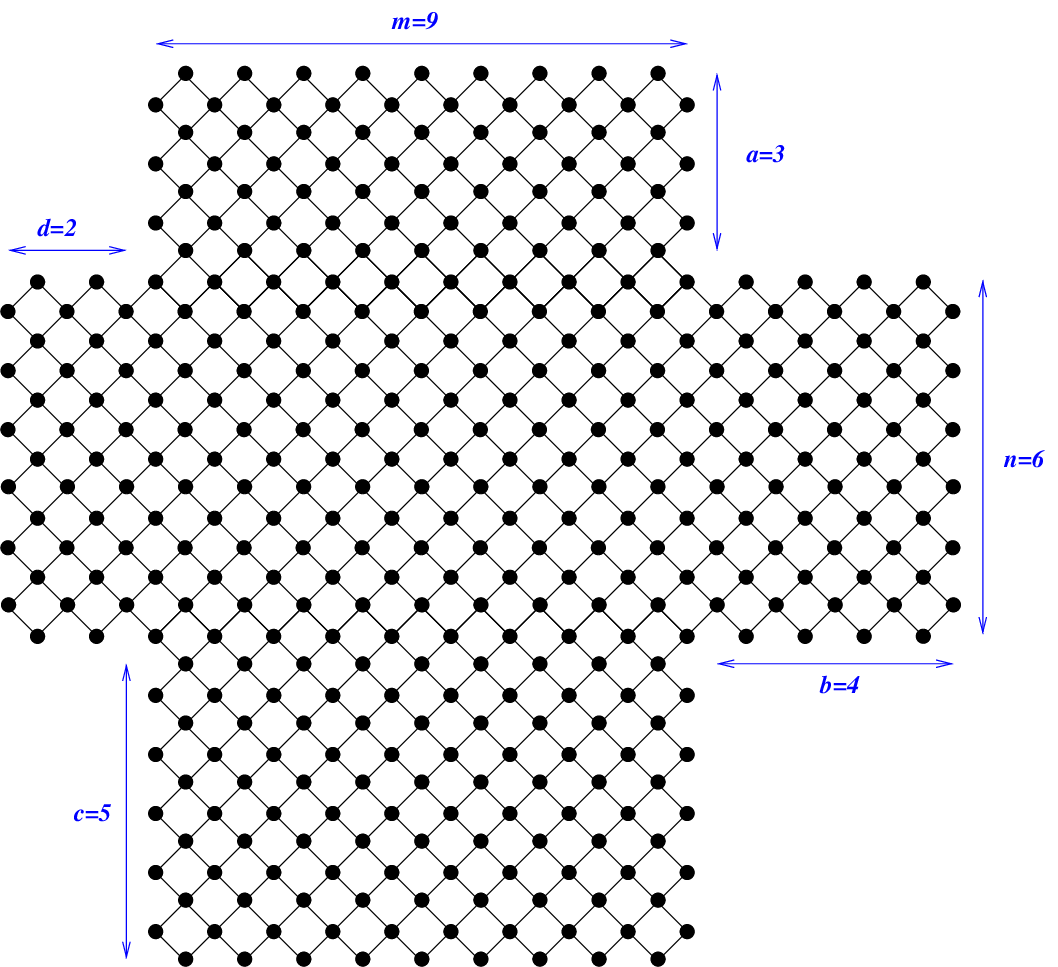}
		\caption{The dual graph of the cruciform region $\mathcal{C}_{m,n}^{a,b,c,d}$.}
		\label{DC}
\end{figure}

An edge of a graph $G$ is called a \emph{forced edge} if it belongs to every matching of $G$. Suppose that we remove forced edges $e_1,e_2,\dots, e_k$ (together with their vertices) from graph $G$\footnote{From now on, we always remove a forced edge together with its vertices.} and denote by $G'$ the resulting graph. Then
\[\M(G)=\left(\prod_{i=1}^{k}\wt(e_i)\right) \cdot \M(G'),\]
where $\wt(e_i)$ denotes the weight of the edge $e_i$. In particular, removal of forced edges of weight $1$ does not change the matching generating function of the graph.

The following three fundamental transformations will be used in our proof.

\begin{lem}[Vertex-Splitting Lemma]\label{VSlem}
Let $v$ be a vertex in $G$. Separate the neighbor set $N(v)$ of $v$ into two disjoint (possibly empty) sets $H$ and $K$. Replace $v$ by two new vertices $v'$ and $v"$, and connect $v'$ to every vertex in $H$ and $v"$ to every vertex in $K$, preserving the edge weights (i.e., $\wt((v',x))=\wt((v,x))$ for any $x\in H$, and $\wt((v",y))=\wt((v,y))$, for any $y\in K$). Add a new vertex $w$ of degree $2$ that connects to $v'$ and $v"$ by two edges of weight $1$. Denote by $G'$ the resulting graph. Then $\M(G)=\M(G')$.
\end{lem}

\begin{figure}[ht]\centering
		\includegraphics[width = 10cm]{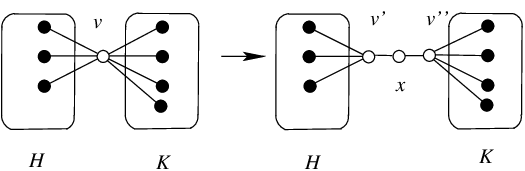}
		\caption{Vertex-splitting}
		\label{VS}
\end{figure}

\begin{lem}[Star Lemma]\label{starlem}
Let $G$ be a weighted graph, and let $v$ be a vertex of $G$. Let 
$G'$ be the graph obtained from $G$ by multiplying the weights of all edges incident to $v$ by 
 a constant $t>0$. Then $\M(G)=\dfrac{1}{t} \cdot \M(G')$.
\end{lem}

The following result, due to Propp, generalizes the ``urban renewal" trick first observed by Kuperberg (see \cite[Section 5]{ProppShuffling}).

\begin{lem}[Spider Lemma or Urban Renewal]\label{spiderlem}
Let $G$ be a weighted graph containing a subgraph $K$ as shown on the left in Figure \ref{Spider} (the labels indicate weights, unlabeled edges have weight $1$). Suppose that $x,y,z,t$ are four real numbers with
$xz+yt\not=0$. Suppose in addition that the four inner black vertices in the subgraph $K$, different from 
$A, B, C, D$, have no neighbors outside $K$. Let $G'$
 be the graph obtained from $G$ by replacing $K$ with the graph $\overline{K}$
 shown on the right in Figure \ref{Spider}, where the dashed lines indicate new edges with the weights shown. Then 
\[\M(G)=(xz+yt) \cdot \M(G').\]
\end{lem}

\begin{figure}[ht]\centering
		\includegraphics[width = 10cm]{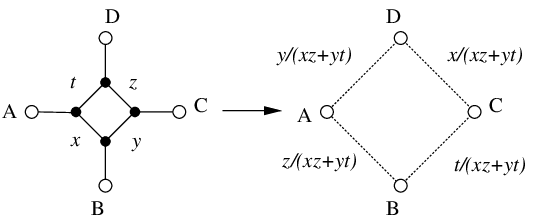}
		\caption{Spider Lemma}
		\label{Spider}
\end{figure}

The first author has used the fundamental transformations above to create more specialized transformations to solve problems in the enumeration of tilings, see, e.g., \cite{Tri1, Tri8, Trinewquarter, Tri10, Triholeyrectangle}. We will also need two additional specialized transformations.

Consider the \emph{Aztec rectangle graph} $AR_{m,n}$  obtained by gluing  $m$ rows of $n$ diamonds together (see the left picture in Figure \ref{ARweight1} for the Aztec rectangle graph $AR_{3,4}$)\footnote{The Aztec rectangle graph $AR_{m,n}$ is exactly the dual graph of the Aztec rectangle region $\mathcal{AR}_{m,n}$ (after being rotated by a $45^{\circ}$ angle).}.  Let $e,f,g,h$ be positive numbers. Define the weight assignment $\wt^{e,f}_{h,g}(q)$ on the edges of $AR_{m,n}$ as follows. For the diamond in the $i$-th row (counted from the bottom) and the $j$-th column (counted from the left), the edge weights are \[e,\quad f,\quad gq^{i+j-2},\quad hq^{i+j-2},\] listed in clockwise order starting from the northwest edge. Denote by $AR_{m,n}(\wt^{e,f}_{h,g}(q))$ the resulting weighted Aztec rectangle graph.  The graph $^|_|AR_{m,n-1}(\wt_{h,g}^{e,f}(q))$ is obtained from  $AR_{m+1,n-1}(wt_{h,g}^{e,f}(q))$ by deleting its top and bottom vertices and adjoining a vertical edge to each of the new top and bottom vertices (see the right picture in Figure \ref{ARweight1}; the dotted edges indicate the removed ones).

\begin{figure}[ht]\centering
\includegraphics[width = 15 cm]{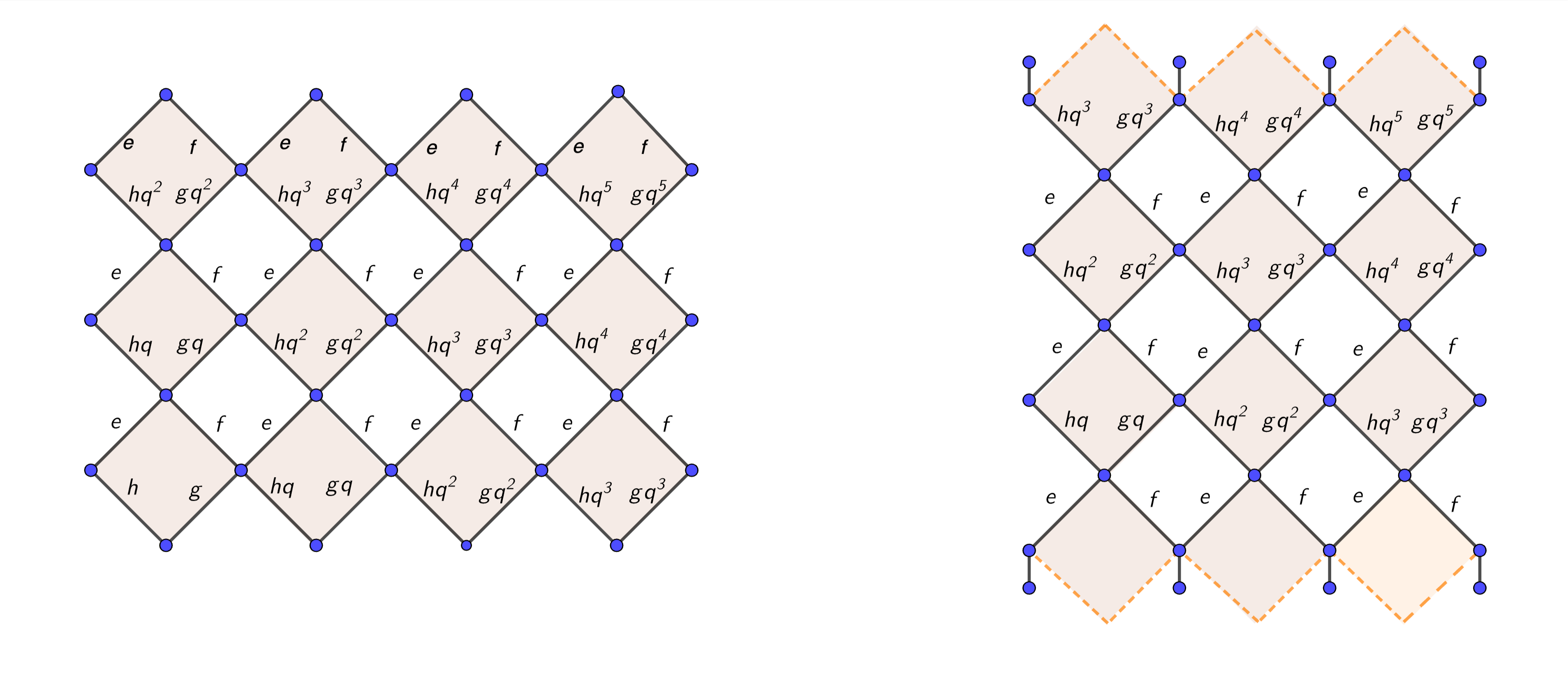}
		\caption{The weighted assignments in $AR_{3,4}(\wt^{e,f}_{h,g}(q))$ (left) and $_|^|AR_{3,3}(\wt^{e,f}_{h,g}(q))$ (right).}
		\label{ARweight1} 
\end{figure} 

The \emph{connected sum} $G\#G'$ of two disjoint graphs $G$ and $G'$ along the ordered sets of vertices $\{v_1,\dots,v_k\} \subseteq V(G)$ and $\{v'_1,\dots,v'_k\}\subseteq V(G')$ is the graph obtained by identifying each pair of corresponding vertices $v_i$ and $v'_i$, for $i=1,2,\dots,k.$\\ 

\begin{figure}[ht]\centering
\includegraphics[width = 15 cm]{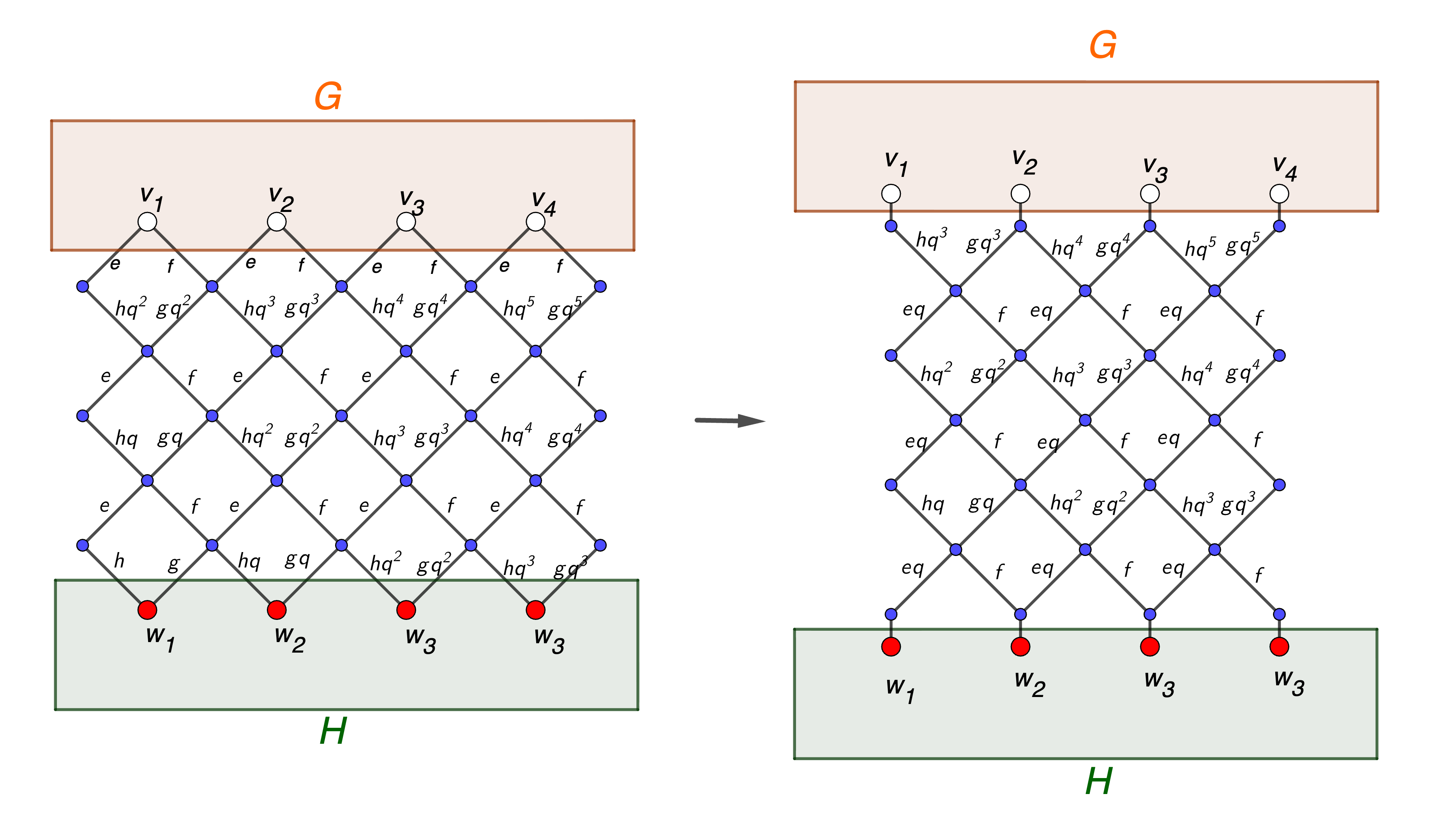}
		\caption{The transformation in Sandwich Lemma 1.}
		\label{AR1} 
\end{figure} 

Now we present some lemmas that will be applied in the paper.

\begin{lem}[Sandwich Lemma 1]\label{lm2}
	Let $G$ be a graph with an ordered vertex subset $\{v_1,$ $v_2,$  $...,$ $v_n\}$, and let $H$ be another graph with an ordered vertex set $\{w_1,w_2,\dots, w_n\}$. 
	Then \[\M(G\# AR_{m,n}(\wt_{h,g}^{e,f}(q))\#H)=(eg+hf)^mq^{\binom{m}{2}}\cdot \M\left(G\# ^|_|AR_{m,n-1}(wt_{h,g}^{eq,f}(q))\#H\right).\] The first connected sum on each side of the identity is taken along the ordered set  $\{v_1,$ $v_2,$ $...,$ $v_n\}$ of $G$ and along the $n$ topmost vertices of the Aztec rectangle graph, and the second connected sum is taken along the ordered set  $\{w_1,w_2,...,w_n\}$ of $H$ and along the $n$ bottommost vertices of the Aztec rectangle graph.	
\end{lem} 		
The proof of the lemma is similar to that of Lemma $2.4$ in \cite{Triholeyrectangle} and is deferred to Section \ref{Sec:Complete}.

In the subsequent lemmas, the same convention for the double connected sum will be used but omitted from the statements.

We also use an alternative weight assignment on $AR_{m,n}$  introduced in \cite{Tri10}, see the left picture in  Figure \ref{ARweight2}. In particular, for the diamond in the $i$-th row (counted from the bottom) and the $j$-th column (counted from the left), the edge weights are \[ gq^{i+j-2},\quad hq^{i+j-2},\quad  e, \quad  f,\] listed clockwise from the northwest edge. We denote by $AR_{m,n}(\overline{wt}^{e,f}_{h,g}(q))$ the resulting weighted Aztec rectangle graph. The graph  $^|_|AR_{m,n-1}(\overline{wt}_{h,g}^{e,f}(q))$ is obtained similarly from $AR_{m+1,n-1}(\overline{wt}_{h,g}^{e,f}(q))$ by removing the top and bottom vertices and appending vertical edges to the new top and bottom vertices (see the right picture in Figure \ref{ARweight2}).

\begin{figure}[ht]\centering
\includegraphics[width = 15 cm]{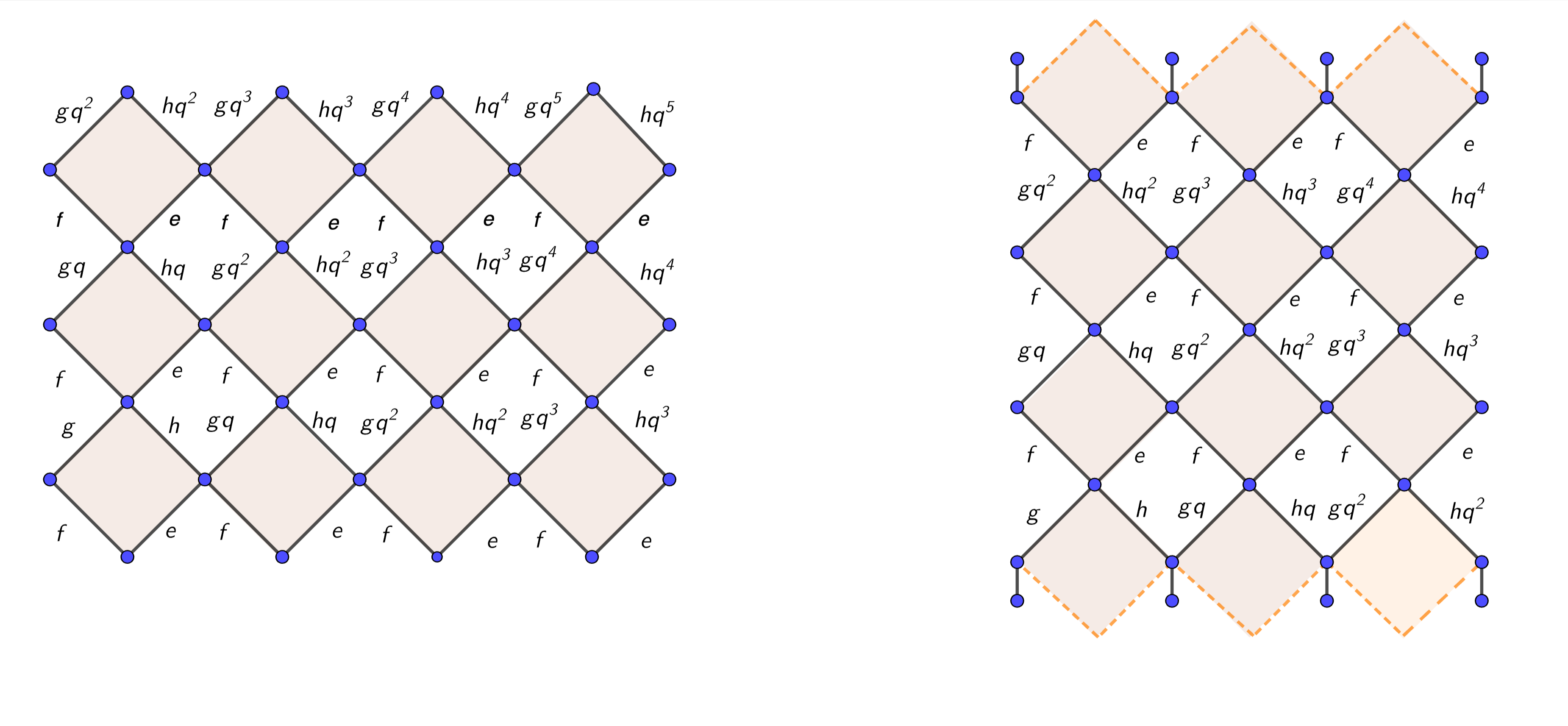}
		\caption{The weight assignments in $AR_{3,4}(\overline{wt}^{e,f}_{h,g}(q))$ (left) and $_|^|AR_{3,3}(\overline{wt}^{e,f}_{h,g}(q))$ (right).}
		\label{ARweight2} 
\end{figure} 

\begin{figure}[ht]\centering
\includegraphics[width = 15 cm]{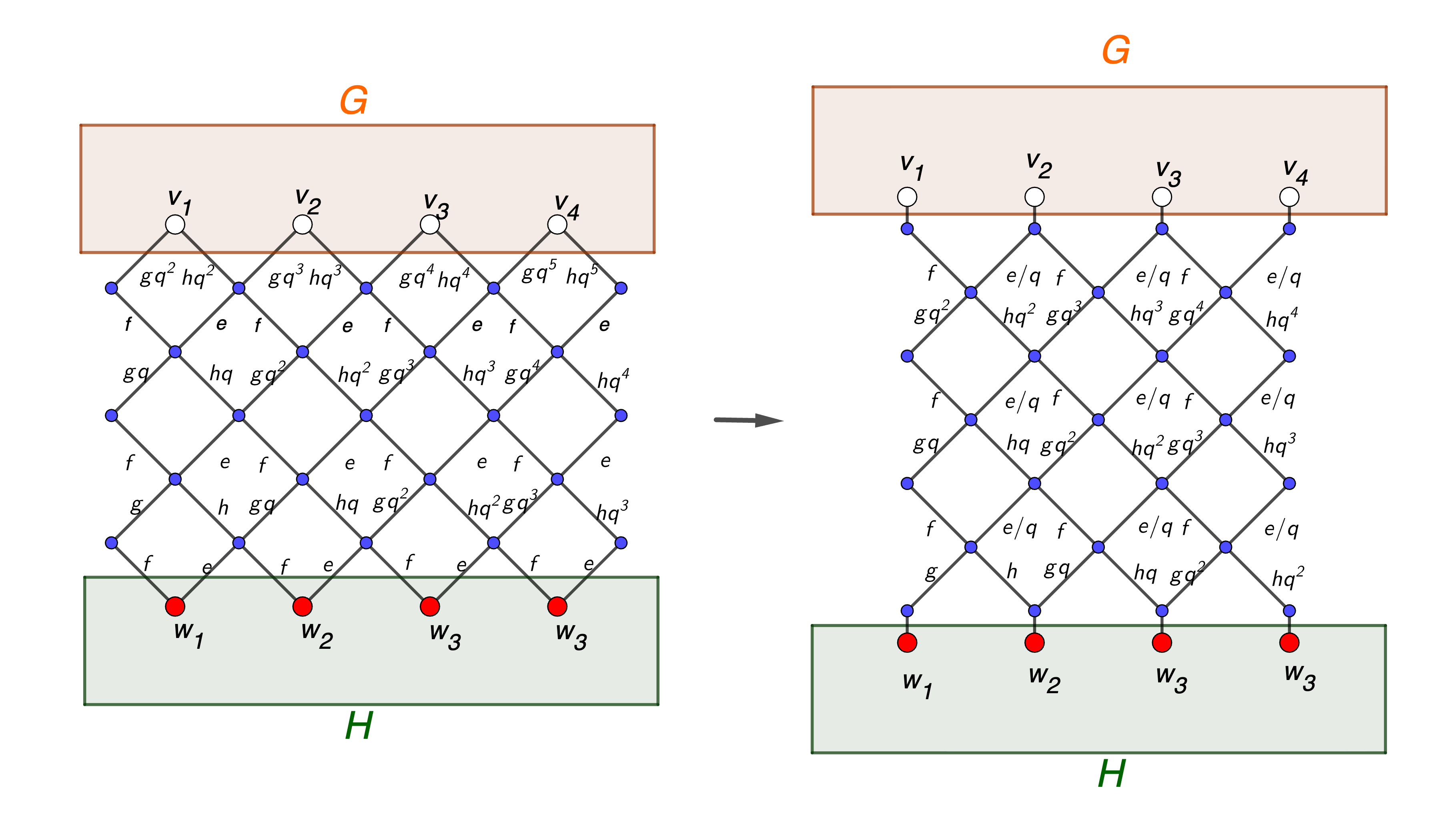}
		\caption{The transformation in Sandwich Lemma 2.}
		\label{AR2} 
\end{figure}

\begin{lem}[Sandwich Lemma 2]\label{lem3}
	Let $G$ be a graph with an ordered vertex subset $\{v_1,$ $v_2,$ $...,$ $v_n\}$, and let $H$ be another graph with an ordered vertex set $\{w_1,w_2,\dots, w_n\}$. 
	Then \[\M(G\# AR_{m,n}(\overline{\wt}_{h,g}^{e,f}(q))\#H)=(eg+hf)^mq^{m(n-1)+\binom{m}{2}} \cdot \M\left(G\# ^|_|AR_{m,n-1}(\overline{\wt}_{h,g}^{e/q,f}(q))\#H\right).\] %The first connected sum in each side of the identity is taken along the ordered set  $\{v_1,v_2,...,v_n\}$ of $G$ and along the $n$ topmost of the Aztec rectangle graph, and the second connected sum is taken along the ordered set  $\{w_1,w_2,...,w_n\}$ of $H$ and along the $n$ bottommost of the Aztec rectangle graph.	
\end{lem}

The proof of the lemma is also deferred to Section \ref{Sec:Complete}.
	
%	\begin{lem}\label{lem4}
%		Let $G$ be a graph with an ordered subset of vertices $\{v_1,v_2,...,v_n\}.$ Then, $$M(G\#AR_{m,n}(e,f,h,g,q))=(eg+hf)^mq^{m(n-1)+C^2_m}.M(G\#^|AR_{m-\frac{1}{2},n-1}(e/q,f,hq,gq,q)),$$ where $^|AR_{m-\frac{1}{2},n-1}(e/q,f,hq,gq,q)$ is obtained from $AR_{m,n-1}(e/q,f,hq,gq,q)$ by removing the topmost vertices and attaching vertical edges to the resulting topmost vertices. The connected sum is taken along $\{v_1,v_2,...,v_n\}$ in $G$ and along the topmost vertices (ordered from left to right) of the other summands.
%	\end{lem}
%	\begin{proof} Similar to the proof of Lemma $3.4$ in \cite{Tri10}.
%	\end{proof}

By applying Sandwich Lemmas \ref{lm2} and~\ref{lem3} repeatedly, we obtain further transformations.

\begin{figure}[ht]\centering
\includegraphics[width = 6 cm]{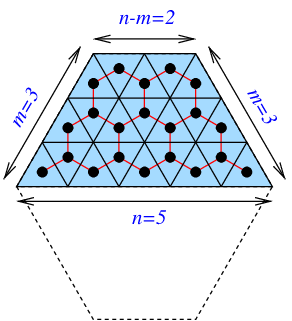}
		\caption{The semi-hexagon $\mathcal{SH}_{3,2}$ and its dual graph.}
		\label{Semihexa} 
\end{figure} 

We denote by $\mathcal{SH}_{m,n-m}$ the \emph{semihexagon} of side-lengths $m,n-m,m,n$. Its dual graph, the \emph{semi-honeycomb graph}, is denoted by $SH_{m,n-m}$ (see Figure \ref{Semihexa} for an example).
Define graph $K_{m,n}$ by attaching the top vertices of the two copies of the semi-honeycomb  $SH_{m,n-m}$ (see Figure \ref{Honey} (left); we disregard the labels for a moment). We now assign weights for the edges of $K_{m,n}$ as follows. View $K_{m,n}$ as a union of $m$ rows of triangles (in the bottom part) and $m$ rows of upside-down triangles (in the top part). In the bottom part, the triangles in the $i$-th row (counted from the base) have the left side weighted by $eq^i$ and the right side weighted by $f$. In the top part, the $j$-th upside-down triangle (counted from the left)  in each row has the left side weighted by $hq^{m+j-1}$ and the right side weighted by $gq^{m+j-1}$. All vertical edges are weighted by $1$. Denote by $K_{m,n}(\wt_{h,g}^{e,f}(q))$ the resulting weighted graph (see the left graph with labels in Figure \ref{Honey}).

\begin{figure}[ht]\centering
\resizebox{!}{10cm}{\begin{picture}(0,0)%
\includegraphics{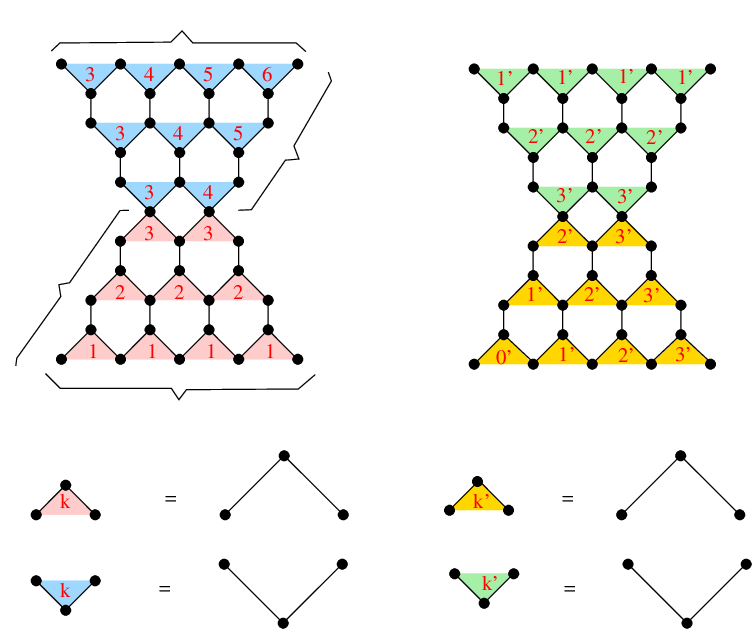}%
\end{picture}%
%
%  Created by WinFIG version 2024.2 
%  METADATA <version>1.0</version> 
%
\setlength{\unitlength}{3947sp}%
\begin{picture}(6017,5032)(1167,-7673)
%  METADATA <id>313</id> 
\put(3015,-6475){\makebox(0,0)[lb]{\smash{\fontsize{10}{12}\normalfont\itshape {\color[rgb]{0,0,1}$eq^k$}%
}}}
%  METADATA <id>314</id> 
\put(3771,-6522){\makebox(0,0)[lb]{\smash{\fontsize{10}{12}\normalfont\itshape {\color[rgb]{0,0,1}$f$}%
}}}
%  METADATA <id>328</id> 
\put(2912,-7483){\makebox(0,0)[lb]{\smash{\fontsize{10}{12}\normalfont\itshape {\color[rgb]{0,0,1}$hq^k$}%
}}}
%  METADATA <id>329</id> 
\put(3716,-7484){\makebox(0,0)[lb]{\smash{\fontsize{10}{12}\normalfont\itshape {\color[rgb]{0,0,1}$gq^k$}%
}}}
%  METADATA <id>386</id> 
\put(2439,-2823){\makebox(0,0)[lb]{\smash{\fontsize{10}{12}\normalfont\itshape {\color[rgb]{0,0,0}$n=5$}%
}}}
%  METADATA <id>388</id> 
\put(2414,-6023){\makebox(0,0)[lb]{\smash{\fontsize{10}{12}\normalfont\itshape {\color[rgb]{0,0,0}$n=5$}%
}}}
%  METADATA <id>389</id> 
\put(3899,-3886){\makebox(0,0)[lb]{\smash{\fontsize{10}{12}\normalfont\itshape {\color[rgb]{0,0,0}$m=3$}%
}}}
%  METADATA <id>390</id> 
\put(1182,-4830){\makebox(0,0)[lb]{\smash{\fontsize{10}{12}\normalfont\itshape {\color[rgb]{0,0,0}$m=3$}%
}}}
%  METADATA <id>514</id> 
\put(6187,-6475){\makebox(0,0)[lb]{\smash{\fontsize{10}{12}\normalfont\itshape {\color[rgb]{0,0,1}$gq^k$}%
}}}
%  METADATA <id>515</id> 
\put(6943,-6522){\makebox(0,0)[lb]{\smash{\fontsize{10}{12}\normalfont\itshape {\color[rgb]{0,0,1}$hq^k$}%
}}}
%  METADATA <id>524</id> 
\put(6143,-7483){\makebox(0,0)[lb]{\smash{\fontsize{10}{12}\normalfont\itshape {\color[rgb]{0,0,1}$f$}%
}}}
%  METADATA <id>525</id> 
\put(6947,-7484){\makebox(0,0)[lb]{\smash{\fontsize{10}{12}\normalfont\itshape {\color[rgb]{0,0,1}$e/q^k$}%
}}}
\end{picture}%
}
		\caption{The graphs $K_{m,n}(\wt_{h,g}^{e,f}(q))=K_{3,5}(\wt_{h,g}^{e,f})$ (left) and $K_{m,n}(\overline{\wt}_{h,g}^{e,f}(q))=K_{3,5}(\overline{\wt}_{h,g}^{e,f})$.}
		\label{Honey} 
\end{figure}
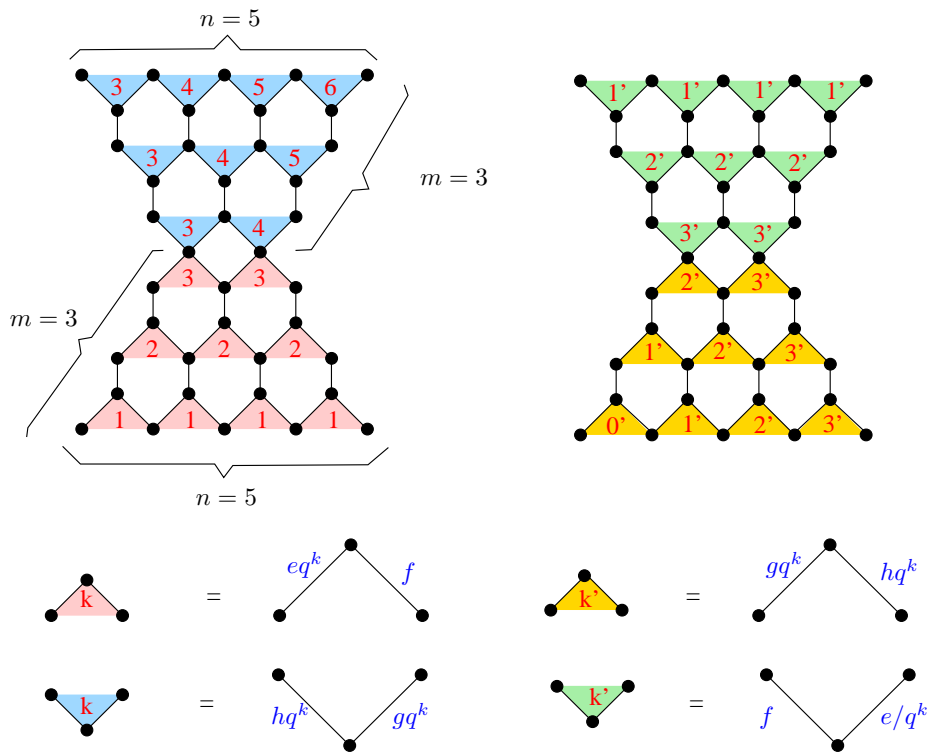

\begin{figure}[ht]\centering
\resizebox{!}{10cm}{\begin{picture}(0,0)%
\includegraphics{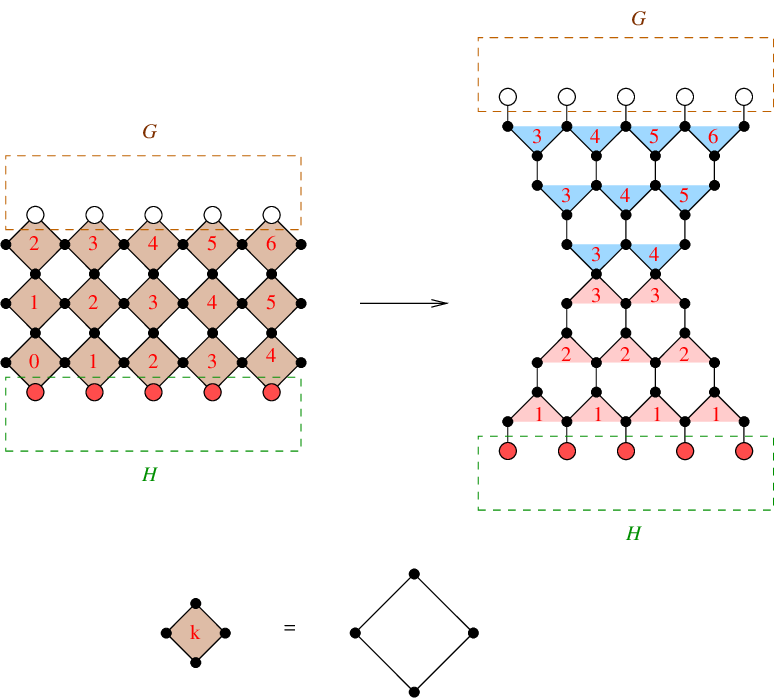}%
\end{picture}%
%
%  Created by WinFIG version 2024.2 
%  METADATA <version>1.0</version> 
%
\setlength{\unitlength}{3947sp}%
\begin{picture}(6200,5568)(2081,-7751)
%  METADATA <id>334</id> 
\put(2316,-3696){\makebox(0,0)[lb]{\smash{\fontsize{10}{12}\usefont{T1}{ptm}{m}{n}{\color[rgb]{0,0,0}$v_1$}%
}}}
%  METADATA <id>335</id> 
\put(2741,-3696){\makebox(0,0)[lb]{\smash{\fontsize{10}{12}\usefont{T1}{ptm}{m}{n}{\color[rgb]{0,0,0}$v_2$}%
}}}
%  METADATA <id>336</id> 
\put(3261,-3696){\makebox(0,0)[lb]{\smash{\fontsize{10}{12}\usefont{T1}{ptm}{m}{n}{\color[rgb]{0,0,0}$v_3$}%
}}}
%  METADATA <id>337</id> 
\put(3733,-3696){\makebox(0,0)[lb]{\smash{\fontsize{10}{12}\usefont{T1}{ptm}{m}{n}{\color[rgb]{0,0,0}$v_4$}%
}}}
%  METADATA <id>338</id> 
\put(4206,-3696){\makebox(0,0)[lb]{\smash{\fontsize{10}{12}\usefont{T1}{ptm}{m}{n}{\color[rgb]{0,0,0}$v_5$}%
}}}
%  METADATA <id>339</id> 
\put(2316,-5586){\makebox(0,0)[lb]{\smash{\fontsize{10}{12}\usefont{T1}{ptm}{m}{n}{\color[rgb]{0,0,0}$w_1$}%
}}}
%  METADATA <id>340</id> 
\put(2788,-5586){\makebox(0,0)[lb]{\smash{\fontsize{10}{12}\usefont{T1}{ptm}{m}{n}{\color[rgb]{0,0,0}$w_2$}%
}}}
%  METADATA <id>341</id> 
\put(3261,-5586){\makebox(0,0)[lb]{\smash{\fontsize{10}{12}\usefont{T1}{ptm}{m}{n}{\color[rgb]{0,0,0}$w_3$}%
}}}
%  METADATA <id>342</id> 
\put(3686,-5586){\makebox(0,0)[lb]{\smash{\fontsize{10}{12}\usefont{T1}{ptm}{m}{n}{\color[rgb]{0,0,0}$w_4$}%
}}}
%  METADATA <id>343</id> 
\put(4158,-5586){\makebox(0,0)[lb]{\smash{\fontsize{10}{12}\usefont{T1}{ptm}{m}{n}{\color[rgb]{0,0,0}$w_5$}%
}}}
%  METADATA <id>349</id> 
\put(6059,-6050){\makebox(0,0)[lb]{\smash{\fontsize{10}{12}\usefont{T1}{ptm}{m}{n}{\color[rgb]{0,0,0}$w_1$}%
}}}
%  METADATA <id>350</id> 
\put(6531,-6050){\makebox(0,0)[lb]{\smash{\fontsize{10}{12}\usefont{T1}{ptm}{m}{n}{\color[rgb]{0,0,0}$w_2$}%
}}}
%  METADATA <id>351</id> 
\put(7004,-6050){\makebox(0,0)[lb]{\smash{\fontsize{10}{12}\usefont{T1}{ptm}{m}{n}{\color[rgb]{0,0,0}$w_3$}%
}}}
%  METADATA <id>352</id> 
\put(7429,-6050){\makebox(0,0)[lb]{\smash{\fontsize{10}{12}\usefont{T1}{ptm}{m}{n}{\color[rgb]{0,0,0}$w_4$}%
}}}
%  METADATA <id>353</id> 
\put(7901,-6050){\makebox(0,0)[lb]{\smash{\fontsize{10}{12}\usefont{T1}{ptm}{m}{n}{\color[rgb]{0,0,0}$w_5$}%
}}}
%  METADATA <id>359</id> 
\put(6097,-2791){\makebox(0,0)[lb]{\smash{\fontsize{10}{12}\usefont{T1}{ptm}{m}{n}{\color[rgb]{0,0,0}$v_1$}%
}}}
%  METADATA <id>360</id> 
\put(6522,-2791){\makebox(0,0)[lb]{\smash{\fontsize{10}{12}\usefont{T1}{ptm}{m}{n}{\color[rgb]{0,0,0}$v_2$}%
}}}
%  METADATA <id>361</id> 
\put(7042,-2791){\makebox(0,0)[lb]{\smash{\fontsize{10}{12}\usefont{T1}{ptm}{m}{n}{\color[rgb]{0,0,0}$v_3$}%
}}}
%  METADATA <id>362</id> 
\put(7514,-2791){\makebox(0,0)[lb]{\smash{\fontsize{10}{12}\usefont{T1}{ptm}{m}{n}{\color[rgb]{0,0,0}$v_4$}%
}}}
%  METADATA <id>363</id> 
\put(7987,-2791){\makebox(0,0)[lb]{\smash{\fontsize{10}{12}\usefont{T1}{ptm}{m}{n}{\color[rgb]{0,0,0}$v_5$}%
}}}chang%  METADATA <id>297</id> 
\put(4250,-7561){\makebox(0,0)[lb]{\smash{\fontsize{10}{12}\normalfont\itshape {\color[rgb]{0,0,1}$hq^k$}%
}}}
%  METADATA <id>294</id> 
\put(4350,-6994){\makebox(0,0)[lb]{\smash{\fontsize{10}{12}\normalfont\itshape {\color[rgb]{0,0,1}$e$}%
}}}
%  METADATA <id>295</id> 
\put(5200,-6994){\makebox(0,0)[lb]{\smash{\fontsize{10}{12}\normalfont\itshape {\color[rgb]{0,0,1}$f$}%
}}}
%  METADATA <id>296</id> 
\put(5200,-7561){\makebox(0,0)[lb]{\smash{\fontsize{10}{12}\normalfont\itshape {\color[rgb]{0,0,1}$gq^k$}%
}}}
\end{picture}%
}
\caption{The transformation in Mega-Sandwich Lemma 1.}\label{Megasandwich1}
\end{figure}

\begin{lem}[Mega-Sandwich Lemma 1]\label{megalem1}
	Let $G$ be a graph with an ordered vertex subset $\{v_1,v_2,...,v_n\}$, and let $H$ be another graph with an ordered vertex set $\{w_1,w_2,\dots, w_n\}$. 
	Then
 \begin{align}
	\M(G\# AR_{m,n}(\wt_{h,g}^{e,f}(q))\#H)=&\prod_{i=1}^{m} (egq^{i-1}+hf)^{m-i+1} \cdot q^{\binom{m}{2}+\binom{m+1}{3}}\notag \\
	&\times \M\left(G\# _| ^|K_{m,n}(\wt_{h,g}^{e,f}(q)) \#H\right),
\end{align}
	where $_| ^|K_{m,n}(\wt_{h,g}^{e,f}(q))$ is the graph obtained by appending a vertical edge (of weight $1$) to each vertex on the top and bottom of $K_{m,n}(\wt_{h,g}^{e,f}(q))$. See Figure \ref{Megasandwich1}.  %The first connected sum in each side of the identity is taken along the ordered set  $\{v_1,v_2,...,v_n\}$ of $G$ and along the $n$ topmost of the Aztec rectangle graph, and the second connected sum is taken along the ordered set  $\{w_1,w_2,...,w_n\}$ of $H$ and along the $n$ bottommost of the Aztec rectangle graph.	
\end{lem}

\begin{proof}
The proof is illustrated in Figure \ref{Megasandwich1c}. We apply Sandwich Lemma 1 (Lemma \ref{lm2}) to replace the Aztec rectangle graph $AR_{m,n}(\wt_{h,g}^{e,f}(q))$ in the connected sum by the graph $^|_|AR_{m,n-1}(\wt_{h,g}^{eq,f}(q))$ (see picture (a)). Here, the diamond of label $i$ has edge-weights $e,f,gq^i,hq^i$, in clockwise order starting from the northwest edge. The latter graph can be partitioned into three parts: the top part is a row of $n-1$ upside-down triangles together with $n$ vertical edges, the middle part is the Aztec rectangle graph $AR_{m-1,n-1}(\wt_{hq,gq}^{eq,f}(q))$, and the bottom part is a row of $n-1$ triangles together with $n$ vertical edges (see right graph in picture (a); the diamond of label $i$' has edge-weights $eq,f,gq \cdot q^{i},hq \cdot q^{i}$. The triangles are weighted the same as in Figure \ref{Megasandwich1}.  We now apply Sandwich Lemma 1 again to the middle part, we replace it by $^|_|AR_{m-1,n-2}(\wt_{hq,gq}^{eq^2,f}(q))$ (see picture (b)). This new graph can again be partitioned into three parts. The middle part is exactly the weighted Aztec rectangle graph $AR_{m-2,n-2}(\wt_{hq^2,gq^2}^{eq^2,f}(q))$ (the diamond of label $i''$ has edge weights $eq^2,f,gq^2 \cdot q^{i},hq^2 \cdot q^{i}$). Repeating this procedure, we apply Sandwich Lemma 1 to the middle part at each step (see picture (c)). After $m$ applications, the original Aztec rectangle graph $AR_{m,n}(\wt_{h,g}^{e,f}(q))$ is replaced by the graph $^|_|K_{m,n}(\wt_{h,g}^{e,f}(q))$. 

Gathering the multiplicative factors of the $m$ transformations, we have
\begin{align}
	\M(G\# AR_{m,n}(\wt_{h,g}^{e,f}(q))\#H)&=\prod_{i=1}^{m} \left(eq^{i-1}\cdot gq^{i-1}+hq^{i-1}\cdot f\right)^{m-i+1} \cdot q^{\sum_{i=1}^{m}\binom{m-i+1}{2}}\notag \\
	&\quad\quad\times \M\left(G\# _| ^|K_{m,n}(\wt_{h,g}^{e,f}(q)) \#H\right)\\
	%%&=\prod_{i=1}^{m} \left(egq^{i-1}+h f\right)^{m-i+1} \cdot q^{\sum_{i=1}^m(i-1)}q^{\sum_{i=1}^{m}\binom{m-i+1}{2}}\notag \\
	%%&\quad\quad\times \M\left(G\# _| ^|K_{m,n}(\wt_{h,g}^{e,f}(q)) \#H\right)\\
	&=\prod_{i=1}^{m} (egq^{i-1}+hf)^{m-i+1} \cdot q^{\binom{m}{2}+\binom{m+1}{3}}\notag \\
	&\quad\quad\times \M\left(G\# _| ^|K_{m,n}(\wt_{h,g}^{e,f}(q)) \#H\right).
\end{align}
\end{proof}

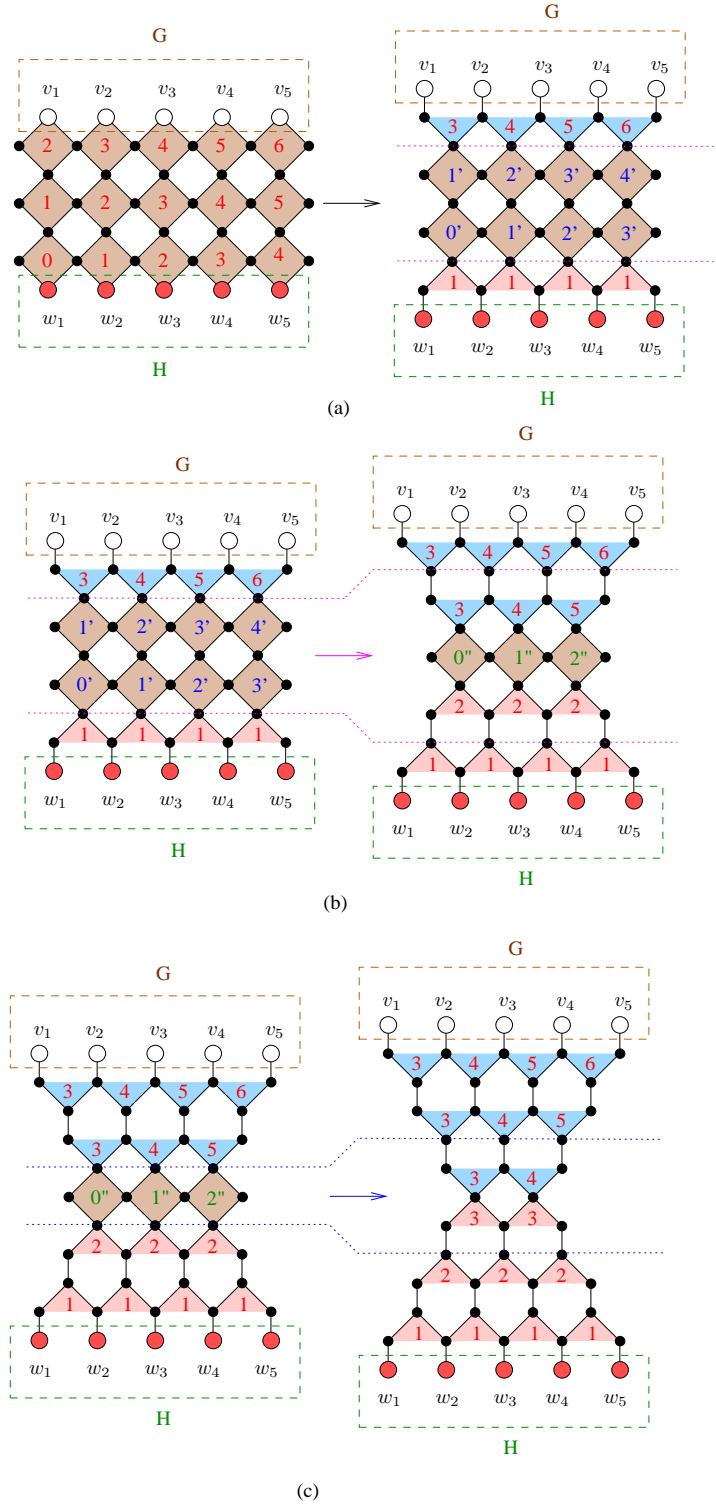
\begin{figure}[ht]\centering
\resizebox{!}{20cm}{
\begin{picture}(0,0)%
\includegraphics{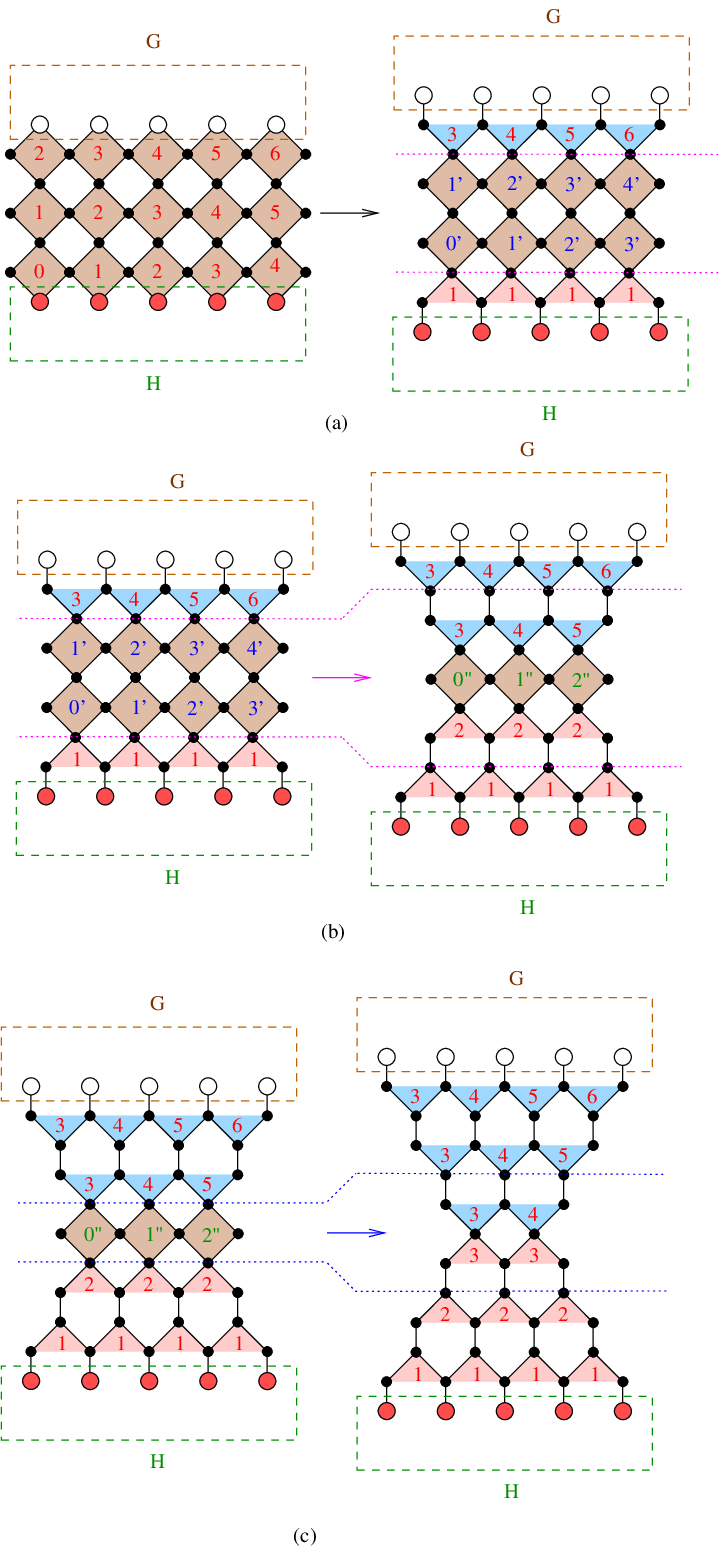}%
\end{picture}%
%
%  Created by WinFIG version 2024.2 
%  METADATA <version>1.0</version> 
%
\setlength{\unitlength}{3947sp}%
\begin{picture}(5758,12394)(2044,-15290)
%  METADATA <id>1451</id> 
\put(6622,-6994){\makebox(0,0)[lb]{\smash{\fontsize{10}{12}\usefont{T1}{ptm}{m}{n}{\color[rgb]{0,0,0}$v_4$}%
}}}
%  METADATA <id>1452</id> 
\put(7095,-6994){\makebox(0,0)[lb]{\smash{\fontsize{10}{12}\usefont{T1}{ptm}{m}{n}{\color[rgb]{0,0,0}$v_5$}%
}}}
%  METADATA <id>849</id> 
\put(5053,-14453){\makebox(0,0)[lb]{\smash{\fontsize{10}{12}\usefont{T1}{ptm}{m}{n}{\color[rgb]{0,0,0}$w_1$}%
}}}
%  METADATA <id>850</id> 
\put(5525,-14453){\makebox(0,0)[lb]{\smash{\fontsize{10}{12}\usefont{T1}{ptm}{m}{n}{\color[rgb]{0,0,0}$w_2$}%
}}}
%  METADATA <id>851</id> 
\put(5998,-14453){\makebox(0,0)[lb]{\smash{\fontsize{10}{12}\usefont{T1}{ptm}{m}{n}{\color[rgb]{0,0,0}$w_3$}%
}}}
%  METADATA <id>852</id> 
\put(6423,-14453){\makebox(0,0)[lb]{\smash{\fontsize{10}{12}\usefont{T1}{ptm}{m}{n}{\color[rgb]{0,0,0}$w_4$}%
}}}
%  METADATA <id>853</id> 
\put(6895,-14453){\makebox(0,0)[lb]{\smash{\fontsize{10}{12}\usefont{T1}{ptm}{m}{n}{\color[rgb]{0,0,0}$w_5$}%
}}}
%  METADATA <id>854</id> 
\put(5091,-11194){\makebox(0,0)[lb]{\smash{\fontsize{10}{12}\usefont{T1}{ptm}{m}{n}{\color[rgb]{0,0,0}$v_1$}%
}}}
%  METADATA <id>855</id> 
\put(5516,-11194){\makebox(0,0)[lb]{\smash{\fontsize{10}{12}\usefont{T1}{ptm}{m}{n}{\color[rgb]{0,0,0}$v_2$}%
}}}
%  METADATA <id>856</id> 
\put(6036,-11194){\makebox(0,0)[lb]{\smash{\fontsize{10}{12}\usefont{T1}{ptm}{m}{n}{\color[rgb]{0,0,0}$v_3$}%
}}}
%  METADATA <id>857</id> 
\put(6508,-11194){\makebox(0,0)[lb]{\smash{\fontsize{10}{12}\usefont{T1}{ptm}{m}{n}{\color[rgb]{0,0,0}$v_4$}%
}}}
%  METADATA <id>858</id> 
\put(6981,-11194){\makebox(0,0)[lb]{\smash{\fontsize{10}{12}\usefont{T1}{ptm}{m}{n}{\color[rgb]{0,0,0}$v_5$}%
}}}
%  METADATA <id>1443</id> 
\put(5167,-9777){\makebox(0,0)[lb]{\smash{\fontsize{10}{12}\usefont{T1}{ptm}{m}{n}{\color[rgb]{0,0,0}$w_1$}%
}}}
%  METADATA <id>1444</id> 
\put(5639,-9777){\makebox(0,0)[lb]{\smash{\fontsize{10}{12}\usefont{T1}{ptm}{m}{n}{\color[rgb]{0,0,0}$w_2$}%
}}}
%  METADATA <id>1445</id> 
\put(6112,-9777){\makebox(0,0)[lb]{\smash{\fontsize{10}{12}\usefont{T1}{ptm}{m}{n}{\color[rgb]{0,0,0}$w_3$}%
}}}
%  METADATA <id>1446</id> 
\put(6537,-9777){\makebox(0,0)[lb]{\smash{\fontsize{10}{12}\usefont{T1}{ptm}{m}{n}{\color[rgb]{0,0,0}$w_4$}%
}}}
%  METADATA <id>334</id> 
\put(2316,-3696){\makebox(0,0)[lb]{\smash{\fontsize{10}{12}\usefont{T1}{ptm}{m}{n}{\color[rgb]{0,0,0}$v_1$}%
}}}
%  METADATA <id>335</id> 
\put(2741,-3696){\makebox(0,0)[lb]{\smash{\fontsize{10}{12}\usefont{T1}{ptm}{m}{n}{\color[rgb]{0,0,0}$v_2$}%
}}}
%  METADATA <id>336</id> 
\put(3261,-3696){\makebox(0,0)[lb]{\smash{\fontsize{10}{12}\usefont{T1}{ptm}{m}{n}{\color[rgb]{0,0,0}$v_3$}%
}}}
%  METADATA <id>337</id> 
\put(3733,-3696){\makebox(0,0)[lb]{\smash{\fontsize{10}{12}\usefont{T1}{ptm}{m}{n}{\color[rgb]{0,0,0}$v_4$}%
}}}
%  METADATA <id>338</id> 
\put(4206,-3696){\makebox(0,0)[lb]{\smash{\fontsize{10}{12}\usefont{T1}{ptm}{m}{n}{\color[rgb]{0,0,0}$v_5$}%
}}}
%  METADATA <id>339</id> 
\put(2316,-5586){\makebox(0,0)[lb]{\smash{\fontsize{10}{12}\usefont{T1}{ptm}{m}{n}{\color[rgb]{0,0,0}$w_1$}%
}}}
%  METADATA <id>340</id> 
\put(2788,-5586){\makebox(0,0)[lb]{\smash{\fontsize{10}{12}\usefont{T1}{ptm}{m}{n}{\color[rgb]{0,0,0}$w_2$}%
}}}
%  METADATA <id>341</id> 
\put(3261,-5586){\makebox(0,0)[lb]{\smash{\fontsize{10}{12}\usefont{T1}{ptm}{m}{n}{\color[rgb]{0,0,0}$w_3$}%
}}}
%  METADATA <id>342</id> 
\put(3686,-5586){\makebox(0,0)[lb]{\smash{\fontsize{10}{12}\usefont{T1}{ptm}{m}{n}{\color[rgb]{0,0,0}$w_4$}%
}}}
%  METADATA <id>343</id> 
\put(4158,-5586){\makebox(0,0)[lb]{\smash{\fontsize{10}{12}\usefont{T1}{ptm}{m}{n}{\color[rgb]{0,0,0}$w_5$}%
}}}
%  METADATA <id>349</id> 
\put(5339,-5820){\makebox(0,0)[lb]{\smash{\fontsize{10}{12}\usefont{T1}{ptm}{m}{n}{\color[rgb]{0,0,0}$w_1$}%
}}}
%  METADATA <id>350</id> 
\put(5811,-5820){\makebox(0,0)[lb]{\smash{\fontsize{10}{12}\usefont{T1}{ptm}{m}{n}{\color[rgb]{0,0,0}$w_2$}%
}}}
%  METADATA <id>351</id> 
\put(6284,-5820){\makebox(0,0)[lb]{\smash{\fontsize{10}{12}\usefont{T1}{ptm}{m}{n}{\color[rgb]{0,0,0}$w_3$}%
}}}
%  METADATA <id>1341</id> 
\put(2328,-9535){\makebox(0,0)[lb]{\smash{\fontsize{10}{12}\usefont{T1}{ptm}{m}{n}{\color[rgb]{0,0,0}$w_1$}%
}}}
%  METADATA <id>1342</id> 
\put(2800,-9535){\makebox(0,0)[lb]{\smash{\fontsize{10}{12}\usefont{T1}{ptm}{m}{n}{\color[rgb]{0,0,0}$w_2$}%
}}}
%  METADATA <id>1343</id> 
\put(3273,-9535){\makebox(0,0)[lb]{\smash{\fontsize{10}{12}\usefont{T1}{ptm}{m}{n}{\color[rgb]{0,0,0}$w_3$}%
}}}
%  METADATA <id>1344</id> 
\put(3698,-9535){\makebox(0,0)[lb]{\smash{\fontsize{10}{12}\usefont{T1}{ptm}{m}{n}{\color[rgb]{0,0,0}$w_4$}%
}}}
%  METADATA <id>1345</id> 
\put(4170,-9535){\makebox(0,0)[lb]{\smash{\fontsize{10}{12}\usefont{T1}{ptm}{m}{n}{\color[rgb]{0,0,0}$w_5$}%
}}}
%  METADATA <id>1346</id> 
\put(2376,-7216){\makebox(0,0)[lb]{\smash{\fontsize{10}{12}\usefont{T1}{ptm}{m}{n}{\color[rgb]{0,0,0}$v_1$}%
}}}
%  METADATA <id>1347</id> 
\put(2801,-7216){\makebox(0,0)[lb]{\smash{\fontsize{10}{12}\usefont{T1}{ptm}{m}{n}{\color[rgb]{0,0,0}$v_2$}%
}}}
%  METADATA <id>1348</id> 
\put(3321,-7216){\makebox(0,0)[lb]{\smash{\fontsize{10}{12}\usefont{T1}{ptm}{m}{n}{\color[rgb]{0,0,0}$v_3$}%
}}}
%  METADATA <id>1349</id> 
\put(3793,-7216){\makebox(0,0)[lb]{\smash{\fontsize{10}{12}\usefont{T1}{ptm}{m}{n}{\color[rgb]{0,0,0}$v_4$}%
}}}
%  METADATA <id>1350</id> 
\put(4266,-7216){\makebox(0,0)[lb]{\smash{\fontsize{10}{12}\usefont{T1}{ptm}{m}{n}{\color[rgb]{0,0,0}$v_5$}%
}}}
%  METADATA <id>352</id> 
\put(6709,-5820){\makebox(0,0)[lb]{\smash{\fontsize{10}{12}\usefont{T1}{ptm}{m}{n}{\color[rgb]{0,0,0}$w_4$}%
}}}
%  METADATA <id>353</id> 
\put(7181,-5820){\makebox(0,0)[lb]{\smash{\fontsize{10}{12}\usefont{T1}{ptm}{m}{n}{\color[rgb]{0,0,0}$w_5$}%
}}}
%  METADATA <id>359</id> 
\put(5387,-3501){\makebox(0,0)[lb]{\smash{\fontsize{10}{12}\usefont{T1}{ptm}{m}{n}{\color[rgb]{0,0,0}$v_1$}%
}}}
%  METADATA <id>360</id> 
\put(5812,-3501){\makebox(0,0)[lb]{\smash{\fontsize{10}{12}\usefont{T1}{ptm}{m}{n}{\color[rgb]{0,0,0}$v_2$}%
}}}
%  METADATA <id>361</id> 
\put(6332,-3501){\makebox(0,0)[lb]{\smash{\fontsize{10}{12}\usefont{T1}{ptm}{m}{n}{\color[rgb]{0,0,0}$v_3$}%
}}}
%  METADATA <id>362</id> 
\put(6804,-3501){\makebox(0,0)[lb]{\smash{\fontsize{10}{12}\usefont{T1}{ptm}{m}{n}{\color[rgb]{0,0,0}$v_4$}%
}}}
%  METADATA <id>363</id> 
\put(7277,-3501){\makebox(0,0)[lb]{\smash{\fontsize{10}{12}\usefont{T1}{ptm}{m}{n}{\color[rgb]{0,0,0}$v_5$}%
}}}
%  METADATA <id>1447</id> 
\put(7009,-9777){\makebox(0,0)[lb]{\smash{\fontsize{10}{12}\usefont{T1}{ptm}{m}{n}{\color[rgb]{0,0,0}$w_5$}%
}}}
%  METADATA <id>1448</id> 
\put(5205,-6994){\makebox(0,0)[lb]{\smash{\fontsize{10}{12}\usefont{T1}{ptm}{m}{n}{\color[rgb]{0,0,0}$v_1$}%
}}}
%  METADATA <id>1449</id> 
\put(5630,-6994){\makebox(0,0)[lb]{\smash{\fontsize{10}{12}\usefont{T1}{ptm}{m}{n}{\color[rgb]{0,0,0}$v_2$}%
}}}
%  METADATA <id>1450</id> 
\put(6150,-6994){\makebox(0,0)[lb]{\smash{\fontsize{10}{12}\usefont{T1}{ptm}{m}{n}{\color[rgb]{0,0,0}$v_3$}%
}}}
%  METADATA <id>603</id> 
\put(2208,-14212){\makebox(0,0)[lb]{\smash{\fontsize{10}{12}\usefont{T1}{ptm}{m}{n}{\color[rgb]{0,0,0}$w_1$}%
}}}
%  METADATA <id>604</id> 
\put(2680,-14212){\makebox(0,0)[lb]{\smash{\fontsize{10}{12}\usefont{T1}{ptm}{m}{n}{\color[rgb]{0,0,0}$w_2$}%
}}}
%  METADATA <id>605</id> 
\put(3153,-14212){\makebox(0,0)[lb]{\smash{\fontsize{10}{12}\usefont{T1}{ptm}{m}{n}{\color[rgb]{0,0,0}$w_3$}%
}}}
%  METADATA <id>606</id> 
\put(3578,-14212){\makebox(0,0)[lb]{\smash{\fontsize{10}{12}\usefont{T1}{ptm}{m}{n}{\color[rgb]{0,0,0}$w_4$}%
}}}
%  METADATA <id>607</id> 
\put(4050,-14212){\makebox(0,0)[lb]{\smash{\fontsize{10}{12}\usefont{T1}{ptm}{m}{n}{\color[rgb]{0,0,0}$w_5$}%
}}}
%  METADATA <id>608</id> 
\put(2246,-11429){\makebox(0,0)[lb]{\smash{\fontsize{10}{12}\usefont{T1}{ptm}{m}{n}{\color[rgb]{0,0,0}$v_1$}%
}}}
%  METADATA <id>609</id> 
\put(2671,-11429){\makebox(0,0)[lb]{\smash{\fontsize{10}{12}\usefont{T1}{ptm}{m}{n}{\color[rgb]{0,0,0}$v_2$}%
}}}
%  METADATA <id>610</id> 
\put(3191,-11429){\makebox(0,0)[lb]{\smash{\fontsize{10}{12}\usefont{T1}{ptm}{m}{n}{\color[rgb]{0,0,0}$v_3$}%
}}}
%  METADATA <id>611</id> 
\put(3663,-11429){\makebox(0,0)[lb]{\smash{\fontsize{10}{12}\usefont{T1}{ptm}{m}{n}{\color[rgb]{0,0,0}$v_4$}%
}}}
%  METADATA <id>612</id> 
\put(4136,-11429){\makebox(0,0)[lb]{\smash{\fontsize{10}{12}\usefont{T1}{ptm}{m}{n}{\color[rgb]{0,0,0}$v_5$}%
}}}
\end{picture}%
}
\caption{Illustration of the proof of Mega-Sandwich Lemma 1.}\label{Megasandwich1c}
\end{figure}

Next, we assign weights to the edges of $K_{m,n}$ differently (see the right graph in Figure \ref{Honey}). In particular, we also view $K_{m,n}$ as a union of $m$ rows of bottomless triangles in the bottom part, and $m$ rows of upside-down triangles in the top part. For the top part, the triangles in the $i$th row from the top (the one labeled by $i'$) have the right edge weighted by $f$ and the left edge weighted by $\dfrac{e}{q^i}$.  For the bottom part, the $j$-th triangle from left in the $i$-th row from the bottom (labeled by $(i+j-2)$') has the left edge weighted by $gq^{i+j-2}$ and the right edge weighted by $hq^{i+j-2}$. All vertical edges have weight $1$. Denote by $K_{m,n}(\overline{\wt}_{h,g}^{e,f}(q))$ this weighted version of $K_{m,n}$.

\begin{figure}[ht]\centering
\resizebox{!}{10cm}{
\begin{picture}(0,0)%
\includegraphics{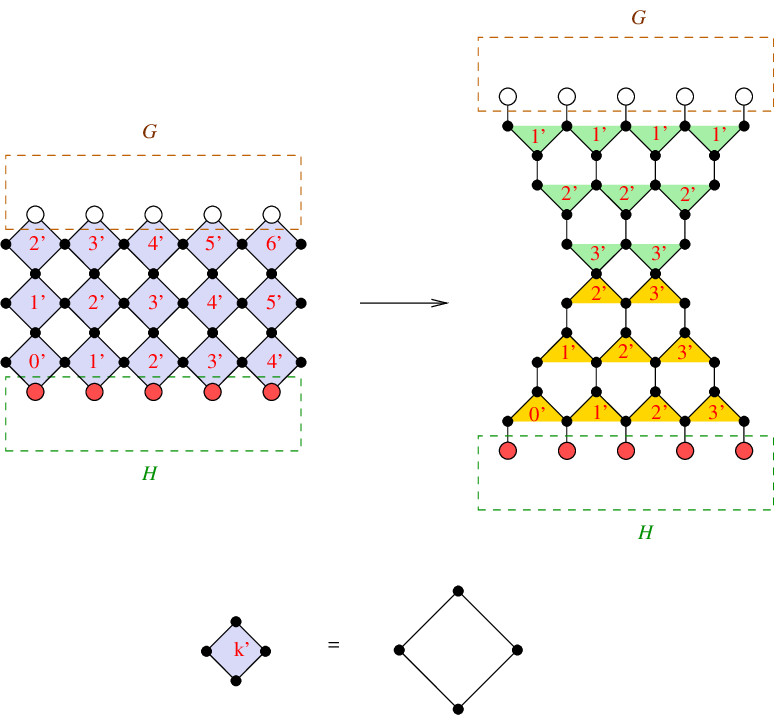}%
\end{picture}%
%
%  Created by WinFIG version 2024.2 
%  METADATA <version>1.0</version> 
%
\setlength{\unitlength}{3947sp}%
\begin{picture}(6200,5707)(2081,-7890)
%  METADATA <id>334</id> 
\put(2316,-3696){\makebox(0,0)[lb]{\smash{\fontsize{10}{12}\usefont{T1}{ptm}{m}{n}{\color[rgb]{0,0,0}$v_1$}%
}}}
%  METADATA <id>335</id> 
\put(2741,-3696){\makebox(0,0)[lb]{\smash{\fontsize{10}{12}\usefont{T1}{ptm}{m}{n}{\color[rgb]{0,0,0}$v_2$}%
}}}
%  METADATA <id>336</id> 
\put(3261,-3696){\makebox(0,0)[lb]{\smash{\fontsize{10}{12}\usefont{T1}{ptm}{m}{n}{\color[rgb]{0,0,0}$v_3$}%
}}}
%  METADATA <id>337</id> 
\put(3733,-3696){\makebox(0,0)[lb]{\smash{\fontsize{10}{12}\usefont{T1}{ptm}{m}{n}{\color[rgb]{0,0,0}$v_4$}%
}}}
%  METADATA <id>338</id> 
\put(4206,-3696){\makebox(0,0)[lb]{\smash{\fontsize{10}{12}\usefont{T1}{ptm}{m}{n}{\color[rgb]{0,0,0}$v_5$}%
}}}
%  METADATA <id>339</id> 
\put(2316,-5586){\makebox(0,0)[lb]{\smash{\fontsize{10}{12}\usefont{T1}{ptm}{m}{n}{\color[rgb]{0,0,0}$w_1$}%
}}}
%  METADATA <id>340</id> 
\put(2788,-5586){\makebox(0,0)[lb]{\smash{\fontsize{10}{12}\usefont{T1}{ptm}{m}{n}{\color[rgb]{0,0,0}$w_2$}%
}}}
%  METADATA <id>341</id> 
\put(3261,-5586){\makebox(0,0)[lb]{\smash{\fontsize{10}{12}\usefont{T1}{ptm}{m}{n}{\color[rgb]{0,0,0}$w_3$}%
}}}
%  METADATA <id>342</id> 
\put(3686,-5586){\makebox(0,0)[lb]{\smash{\fontsize{10}{12}\usefont{T1}{ptm}{m}{n}{\color[rgb]{0,0,0}$w_4$}%
}}}
%  METADATA <id>343</id> 
\put(4158,-5586){\makebox(0,0)[lb]{\smash{\fontsize{10}{12}\usefont{T1}{ptm}{m}{n}{\color[rgb]{0,0,0}$w_5$}%
}}}
%  METADATA <id>349</id> 
\put(6059,-6050){\makebox(0,0)[lb]{\smash{\fontsize{10}{12}\usefont{T1}{ptm}{m}{n}{\color[rgb]{0,0,0}$w_1$}%
}}}
%  METADATA <id>350</id> 
\put(6531,-6050){\makebox(0,0)[lb]{\smash{\fontsize{10}{12}\usefont{T1}{ptm}{m}{n}{\color[rgb]{0,0,0}$w_2$}%
}}}
%  METADATA <id>351</id> 
\put(7004,-6050){\makebox(0,0)[lb]{\smash{\fontsize{10}{12}\usefont{T1}{ptm}{m}{n}{\color[rgb]{0,0,0}$w_3$}%
}}}
%  METADATA <id>352</id> 
\put(7429,-6050){\makebox(0,0)[lb]{\smash{\fontsize{10}{12}\usefont{T1}{ptm}{m}{n}{\color[rgb]{0,0,0}$w_4$}%
}}}
%  METADATA <id>353</id> 
\put(7901,-6050){\makebox(0,0)[lb]{\smash{\fontsize{10}{12}\usefont{T1}{ptm}{m}{n}{\color[rgb]{0,0,0}$w_5$}%
}}}
%  METADATA <id>359</id> 
\put(6097,-2791){\makebox(0,0)[lb]{\smash{\fontsize{10}{12}\usefont{T1}{ptm}{m}{n}{\color[rgb]{0,0,0}$v_1$}%
}}}
%  METADATA <id>360</id> 
\put(6522,-2791){\makebox(0,0)[lb]{\smash{\fontsize{10}{12}\usefont{T1}{ptm}{m}{n}{\color[rgb]{0,0,0}$v_2$}%
}}}
%  METADATA <id>361</id> 
\put(7042,-2791){\makebox(0,0)[lb]{\smash{\fontsize{10}{12}\usefont{T1}{ptm}{m}{n}{\color[rgb]{0,0,0}$v_3$}%
}}}
%  METADATA <id>362</id> 
\put(7514,-2791){\makebox(0,0)[lb]{\smash{\fontsize{10}{12}\usefont{T1}{ptm}{m}{n}{\color[rgb]{0,0,0}$v_4$}%
}}}
%  METADATA <id>363</id> 
\put(7987,-2791){\makebox(0,0)[lb]{\smash{\fontsize{10}{12}\usefont{T1}{ptm}{m}{n}{\color[rgb]{0,0,0}$v_5$}%
}}}
%  METADATA <id>296</id> 
\put(5243,-7048){\makebox(0,0)[lb]{\smash{\fontsize{10}{12}\normalfont\itshape {\color[rgb]{0,0,1}$gq^k$}%
}}}
%  METADATA <id>297</id> 
\put(6094,-7095){\makebox(0,0)[lb]{\smash{\fontsize{10}{12}\normalfont\itshape {\color[rgb]{0,0,1}$hq^k$}%
}}}
%  METADATA <id>294</id> 
\put(5999,-7709){\makebox(0,0)[lb]{\smash{\fontsize{10}{12}\normalfont\itshape {\color[rgb]{0,0,1}$e$}%
}}}
%  METADATA <id>295</id> 
\put(5338,-7756){\makebox(0,0)[lb]{\smash{\fontsize{10}{12}\normalfont\itshape {\color[rgb]{0,0,1}$f$}%
}}}
\end{picture}%
}
\caption{The transformation in Mega-Sandwich Lemma 2.}\label{Megasandwich2}
\end{figure}

\begin{lem}[Mega-Sandwich Lemma 2]\label{megalem2}
	Let $G$ be a graph with an ordered vertex subset $\{v_1,v_2,...,v_n\}$, and let $H$ be another graph with an ordered vertex set $\{w_1,w_2,\dots, w_n\}$. 
	Then
 \begin{align}
	\M(G\# AR_{m,n}(\overline{\wt}_{h,g}^{e,f}(q))\#H)=&\prod_{i=1}^{m} (eg+hfq^{i-1})^{m-i+1} \cdot q^{(n-1)m(m+1)/2}\notag \\
	&\times \M\left(G\# _| ^|K_{m,n}(\overline{\wt}_{h,g}^{e,f}(q)) \#H\right),
\end{align}
	where $_| ^|K_{m,n}(\overline{\wt}_{h,g}^{e,f}(q))$ is the graph obtained by appending a vertical edges of weight $1$ to each vertex on the top and bottom of $K_{m,n}(\overline{\wt}_{h,g}^{e,f}(q))$. See Figure \ref{Megasandwich2}.  %The first connected sum in each side of the identity is taken along the ordered set  $\{v_1,v_2,...,v_n\}$ of $G$ and along the $n$ topmost of the Aztec rectangle graph, and the second connected sum is taken along the ordered set  $\{w_1,w_2,...,w_n\}$ of $H$ and along the $n$ bottommost of the Aztec rectangle graph.	
\end{lem}

\begin{proof}
Similarly to Lemma \ref{megalem1}, we  apply $m$ times Lemma \ref{lem3} to transform $AR_{m,n}(\overline{\wt}_{h,g}^{e,f}(q))$ into $_| ^|K_{m,n}(\overline{\wt}_{h,g}^{e,f}(q))$. The multiplicative factor arising in the $i$-th application is \[\left(\frac{e}{q^{i-1}}\cdot gq^{i-1}+hq^{i-1}\cdot f\right)^{m-i+1} \cdot q^{(m-i+1)(n-i)+\binom{m-i+1}{2}}.\] Collecting all the factors, we obtain
\begin{align}
	\M(G\# AR_{m,n}(\overline{\wt}_{h,g}^{e,f}(q))\#H)&=\prod_{i=1}^{m} (eg+hfq^{i-1})^{m-i+1} \cdot q^{\sum_{i=1}^m(m-i+1)(n-i)+\sum_{i=1}^{m}\binom{m-i+1}{2}}\notag \\
	&\quad\quad\times \M\left(G\# _| ^|K_{m,n}(\overline{\wt}_{h,g}^{e,f}(q)) \#H\right)\\
	&=\prod_{i=1}^{m} (eg+hfq^{i-1})^{m-i+1} \cdot q^{(n-1)m(m+1)/2}\notag \\
	&\quad\quad\times \M\left(G\# _| ^|K_{m,n}(\overline{\wt}_{h,g}^{e,f}(q)) \#H\right).
\end{align}
\end{proof}

\subsection{Schur function identities}

Given $k$ integers $\lambda_1\geq \lambda_2\geq\cdots \geq\lambda_k\geq 0$, a \emph{plane partition} of shape $(\lambda_1,\lambda_2,\dotsc,\lambda_k)$ is an array of nonnegative integers of the form 
\begin{center}
\begin{tabular}{rccccccccc}
$n_{1,1}$   &$n_{1,2}$                 &$n_{1,3}$               & $\dotsc$               &  $\dotsc$                        & $\dotsc$                            &   $n_{1,\lambda_1}$ \\\noalign{\smallskip\smallskip}
$n_{2,1}$   &  $n_{2,2}$              & $n_{2,3}$             &  $\dotsc$               & $\dotsc$                        &         $n_{2,\lambda_2}$&          \\\noalign{\smallskip\smallskip}
$\vdots$    &       $\vdots$            & $\vdots$                &        $\vdots$         &     $\vdots$                &    &              \\\noalign{\smallskip\smallskip}
 $n_{k,1}$  &  $n_{k,2}$               & $n_{k,3}$              &     $\dotsc$             &   $n_{k,\lambda_k}$ &                                          &           \\\noalign{\smallskip\smallskip}
\end{tabular},
\end{center}
where the entries are weakly decreasing across the rows and down the columns. A \emph{column-strict plane partition} is a plane partition whose entries in each column are strictly decreasing. We refer the reader to \cite{Stanley} for properties of column-strict plane partitions.

\begin{figure}[ht]\centering
		\includegraphics[width=7cm]{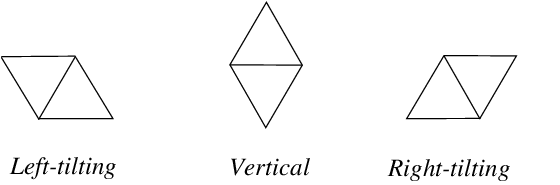}
		\caption{The orientations of the lozenges.}
		\label{rhombustype}
\end{figure}

We are interested in the lozenge tilings of the semi-hexagon $\mathcal{SH}_{a,b}$ with $a$ up-pointing unit triangles removed from the base; these removed triangles are called \emph{dents}. Assume that the positions of the dents are $s_1,s_2,\dots,s_a$, where $1\leq s_1<s_2<\dotsc <s_a\leq a+b$. We denote by $\mathcal{SH}=\mathcal{SH}_{a,b}(s_1,\dotsc,s_a)$ the semi-hexagon with dents (see Figure \ref{semihex}(a) for an example of a lozenge tiling of $\mathcal{SH}_{6,4}(1,3,6,7,8,10)$).
We now encode each lozenge tiling $L$ of $\mathcal{SH}$ as an $a$-tuple of  disjoint lozenge paths $(Q_1,\dotsc,Q_{a})$. Each lozenge path $Q_i$ consists of \emph{vertical} and \emph{right-tilting lozenges} connecting the $i$-th node (counted from the top) on the northwest side and the $i$-th node (counted from the right) along the base (illustrated in Figure \ref{semihex}(c); see also Figure \ref{rhombustype} for the three different orientations of the lozenges). The dotted lines indicate these paths.

For any set of positive integers $X=\{x_1,x_2,\dots,x_n\}$, with $x_1<x_2<\cdots<x_n$,  we denote by $\lambda(X)$ the integer partition $(x_n-n, x_{n-1}-(n-1), \dots ,x_1-1)$. (Here we allow zero parts in an integer partition.) Next, we give each right-tilting lozenge a weight $q^{k}$, where $k\cdot \frac{\sqrt{3}}{2}$ is the distance from its top to the base of the semi-hexagon, as in Figure \ref{semihex}(c) (all other lozenges have weight $1$). The exponents of $q$ along the path $Q_i$ give the entries of the $i$-th row of a \emph{column-strict plane partition} of shape $\lambda(s_1,s_2,\dots,s_a):=(s_a-a, s_{a-1}-a+1, \dotsc , s_1-1)$ with positive entries at most $a$ (see Figure \ref{semihex}(d)). We note that the path $Q_6$ in Figure \ref{semihex}(c) is degenerate (it has no lozenge) and that the vertical interval at the bottom of the plane partition in Figure \ref{semihex}(d) represents a row of length $0$. It is easy to verify that the above correspondence yields a bijection. Moreover, the weight of the lozenge tiling $L$ of $\mathcal{SH}_{a,b}(s_1,\dotsc,s_a)$ is exactly $q^{|\pi_L|}$, where $\pi_L$ is the column-strict plane partition corresponding  to $L$ and where $|\pi_L|$ is the sum of all entries of $\pi_L$.  Summing over all lozenge tilings $L$ of $\mathcal{SH}_{a,b}(s_1,\dotsc,s_a)$, we obtain the following Schur function identity:
\begin{align}\label{Gp}
\T(\mathcal{SH}_{a,b}(s_1,\dotsc,s_a))&=\sum_{\pi}q^{|\pi|}=s_{\lambda(s_1,s_2,\dots,s_a)}(q,q^2,\dots,q^{a})\notag\\
&=q^{\sum_{i=1}^{a}(s_i-i)}\prod_{1\leq i<j\leq a}\frac{q^{s_j}-q^{s_i}}{q^{j}-q^{i}},
\end{align}
where the sum after the first equal sign is taken over all column-strict plane partitions $\pi$ of shape $\lambda(s_1,s_2,\dots,s_a)=(s_a-a, s_{a-1}-a+1, \dotsc, s_1-1)$ with positive entries at most $a$ (see, e.g., \cite[pp.375]{Stanley}). We note that a lozenge tiling $L$ of the semi-hexagon can also be encoded as a family of $b$ lozenge paths $(P_1,P_2,\dots,P_b)$ connecting the top and the bottom of the region, as shown in Figure \ref{semihex}(b). In this case, the exponents of $q$ along the path $P_i$ correspond to the entries in column $i$ of the column-strict plane partition $\pi_L$.

\begin{figure}[ht]\centering
		\includegraphics[width=15 cm]{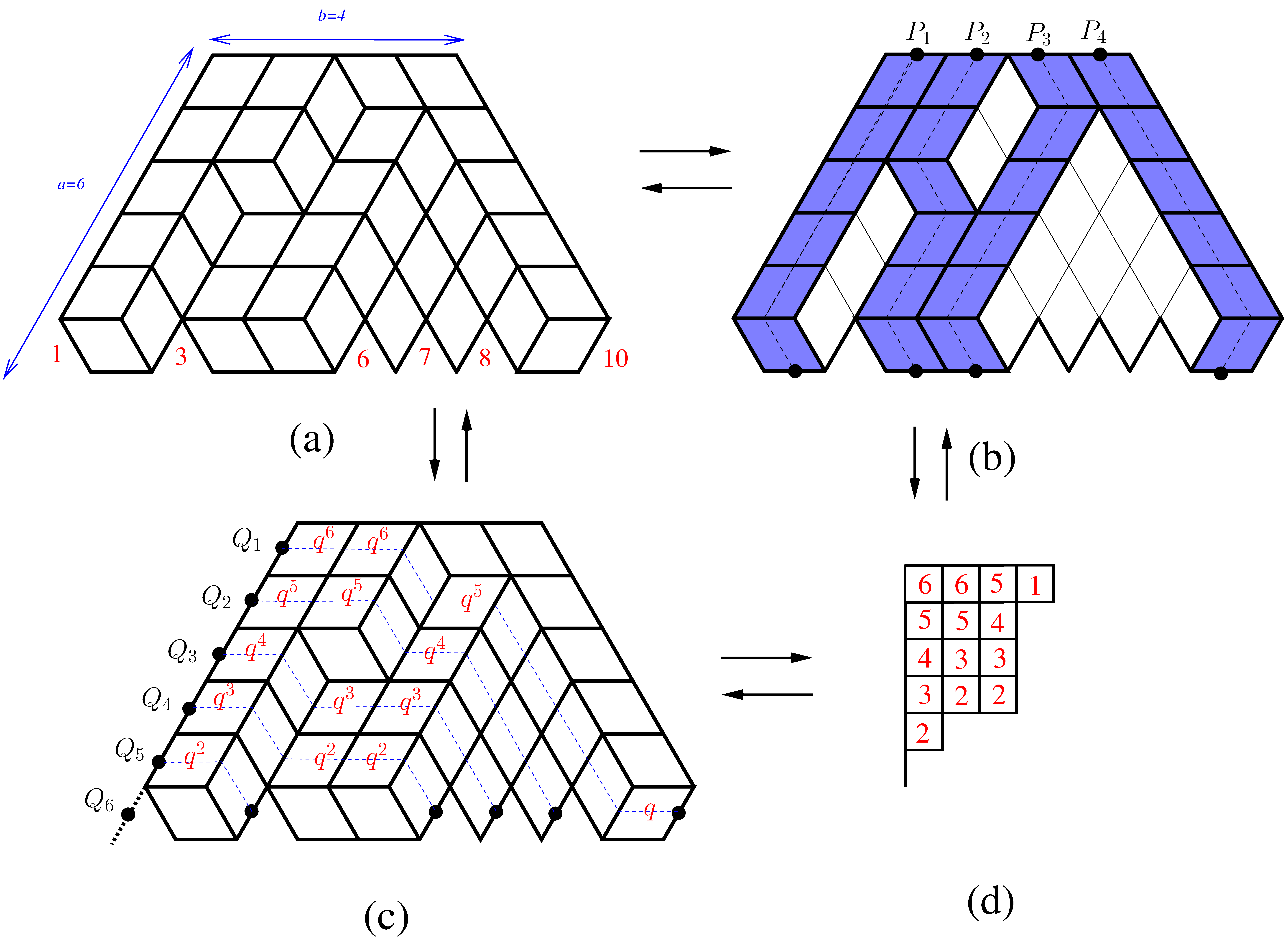}
		\caption{The correspondence between lozenge tilings of a semihexagon $\mathcal{SH}_{a,b}$ with dents and column strict plane partition of shape .}
		\label{semihex}
\end{figure}

We consider the following two variants of the weighted lozenge tilings. Let $ x$ and $ y$ be two positive real numbers. We assign weights to lozenges of the semi-hexagon $\mathcal{SH}_{a,b}(s_1,s_2,\dots,s_a)$ as follows. Assign weight $xq^{i}$ to each right-tilting lozenge in the $i$-th row counted from the bottom, weight $y$ to each left-tilting lozenge, and weight $1$ to each vertical lozenge (see Figure \ref{semihex2}(a)). Denote the resulting weighted semi-hexagon by $\mathcal{SH}_{a,b}(\wt_{x,y}(q))$. The weight assignment in Figure \ref{semihex} is exactly the special case $x=y=1$.

\begin{figure}[ht]\centering
\resizebox{15cm}{!}{
\begin{picture}(0,0)%
\includegraphics{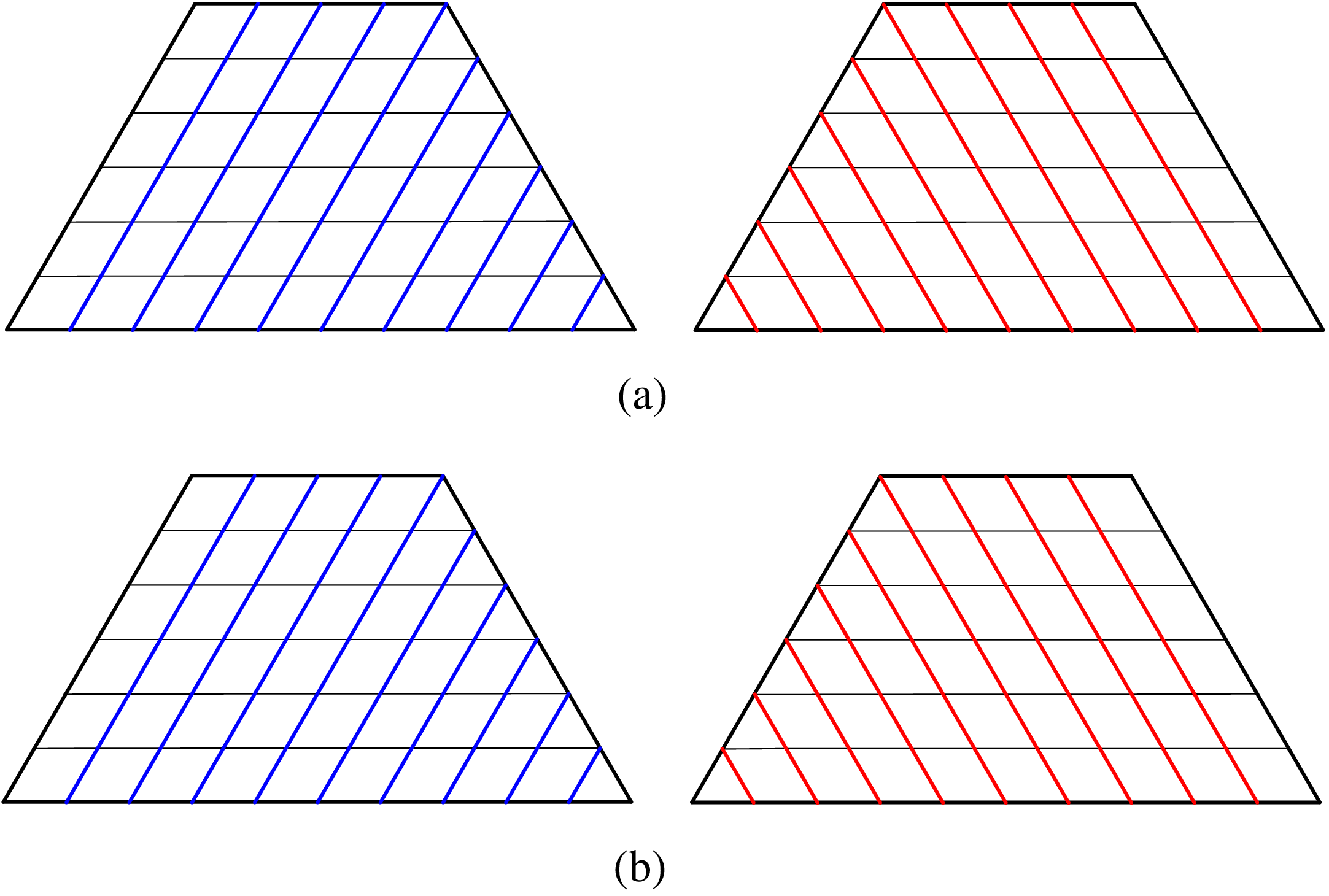}%
\end{picture}%
%
%  Created by WinFIG version 2024.2 
%  METADATA <version>1.0</version> 
%
\setlength{\unitlength}{3947sp}%
\begin{picture}(17310,11677)(333,-11664)
%  METADATA <id>271</id> 
\put(14712,-6620){\makebox(0,0)[lb]{\smash{\fontsize{22}{26.4}\normalfont\itshape {\color[rgb]{0,0,0}$yq^4$}%
}}}
%  METADATA <id>235</id> 
\put(11854,-10200){\makebox(0,0)[lb]{\smash{\fontsize{22}{26.4}\normalfont\itshape {\color[rgb]{0,0,0}$yq^3$}%
}}}
%  METADATA <id>236</id> 
\put(12672,-10185){\makebox(0,0)[lb]{\smash{\fontsize{22}{26.4}\normalfont\itshape {\color[rgb]{0,0,0}$yq^4$}%
}}}
%  METADATA <id>237</id> 
\put(13467,-10170){\makebox(0,0)[lb]{\smash{\fontsize{22}{26.4}\normalfont\itshape {\color[rgb]{0,0,0}$yq^5$}%
}}}
%  METADATA <id>238</id> 
\put(14292,-10177){\makebox(0,0)[lb]{\smash{\fontsize{22}{26.4}\normalfont\itshape {\color[rgb]{0,0,0}$yq^6$}%
}}}
%  METADATA <id>239</id> 
\put(15102,-10192){\makebox(0,0)[lb]{\smash{\fontsize{22}{26.4}\normalfont\itshape {\color[rgb]{0,0,0}$yq^7$}%
}}}
%  METADATA <id>240</id> 
\put(15942,-10192){\makebox(0,0)[lb]{\smash{\fontsize{22}{26.4}\normalfont\itshape {\color[rgb]{0,0,0}$yq^8$}%
}}}
%  METADATA <id>241</id> 
\put(16767,-10185){\makebox(0,0)[lb]{\smash{\fontsize{22}{26.4}\normalfont\itshape {\color[rgb]{0,0,0}$yq^9$}%
}}}
%  METADATA <id>242</id> 
\put(10617,-9466){\makebox(0,0)[lb]{\smash{\fontsize{22}{26.4}\normalfont\itshape {\color[rgb]{0,0,0}$yq$}%
}}}
%  METADATA <id>243</id> 
\put(11434,-9458){\makebox(0,0)[lb]{\smash{\fontsize{22}{26.4}\normalfont\itshape {\color[rgb]{0,0,0}$yq^2$}%
}}}
%  METADATA <id>244</id> 
\put(12259,-9463){\makebox(0,0)[lb]{\smash{\fontsize{22}{26.4}\normalfont\itshape {\color[rgb]{0,0,0}$yq^3$}%
}}}
%  METADATA <id>245</id> 
\put(13077,-9448){\makebox(0,0)[lb]{\smash{\fontsize{22}{26.4}\normalfont\itshape {\color[rgb]{0,0,0}$yq^4$}%
}}}
%  METADATA <id>246</id> 
\put(13872,-9433){\makebox(0,0)[lb]{\smash{\fontsize{22}{26.4}\normalfont\itshape {\color[rgb]{0,0,0}$yq^5$}%
}}}
%  METADATA <id>247</id> 
\put(14697,-9440){\makebox(0,0)[lb]{\smash{\fontsize{22}{26.4}\normalfont\itshape {\color[rgb]{0,0,0}$yq^6$}%
}}}
%  METADATA <id>248</id> 
\put(15507,-9455){\makebox(0,0)[lb]{\smash{\fontsize{22}{26.4}\normalfont\itshape {\color[rgb]{0,0,0}$yq^7$}%
}}}
%  METADATA <id>249</id> 
\put(16347,-9455){\makebox(0,0)[lb]{\smash{\fontsize{22}{26.4}\normalfont\itshape {\color[rgb]{0,0,0}$yq^8$}%
}}}
%  METADATA <id>250</id> 
\put(10992,-8761){\makebox(0,0)[lb]{\smash{\fontsize{22}{26.4}\normalfont\itshape {\color[rgb]{0,0,0}$yq$}%
}}}
%  METADATA <id>251</id> 
\put(11809,-8753){\makebox(0,0)[lb]{\smash{\fontsize{22}{26.4}\normalfont\itshape {\color[rgb]{0,0,0}$yq^2$}%
}}}
%  METADATA <id>252</id> 
\put(12634,-8758){\makebox(0,0)[lb]{\smash{\fontsize{22}{26.4}\normalfont\itshape {\color[rgb]{0,0,0}$yq^3$}%
}}}
%  METADATA <id>253</id> 
\put(13452,-8743){\makebox(0,0)[lb]{\smash{\fontsize{22}{26.4}\normalfont\itshape {\color[rgb]{0,0,0}$yq^4$}%
}}}
%  METADATA <id>254</id> 
\put(14247,-8728){\makebox(0,0)[lb]{\smash{\fontsize{22}{26.4}\normalfont\itshape {\color[rgb]{0,0,0}$yq^5$}%
}}}
%  METADATA <id>255</id> 
\put(15072,-8735){\makebox(0,0)[lb]{\smash{\fontsize{22}{26.4}\normalfont\itshape {\color[rgb]{0,0,0}$yq^6$}%
}}}
%  METADATA <id>256</id> 
\put(15882,-8750){\makebox(0,0)[lb]{\smash{\fontsize{22}{26.4}\normalfont\itshape {\color[rgb]{0,0,0}$yq^7$}%
}}}
%  METADATA <id>257</id> 
\put(11412,-8033){\makebox(0,0)[lb]{\smash{\fontsize{22}{26.4}\normalfont\itshape {\color[rgb]{0,0,0}$yq$}%
}}}
%  METADATA <id>258</id> 
\put(12229,-8025){\makebox(0,0)[lb]{\smash{\fontsize{22}{26.4}\normalfont\itshape {\color[rgb]{0,0,0}$yq^2$}%
}}}
%  METADATA <id>259</id> 
\put(13054,-8030){\makebox(0,0)[lb]{\smash{\fontsize{22}{26.4}\normalfont\itshape {\color[rgb]{0,0,0}$yq^3$}%
}}}
%  METADATA <id>260</id> 
\put(13872,-8015){\makebox(0,0)[lb]{\smash{\fontsize{22}{26.4}\normalfont\itshape {\color[rgb]{0,0,0}$yq^4$}%
}}}
%  METADATA <id>261</id> 
\put(14667,-8000){\makebox(0,0)[lb]{\smash{\fontsize{22}{26.4}\normalfont\itshape {\color[rgb]{0,0,0}$yq^5$}%
}}}
%  METADATA <id>262</id> 
\put(15492,-8007){\makebox(0,0)[lb]{\smash{\fontsize{22}{26.4}\normalfont\itshape {\color[rgb]{0,0,0}$yq^6$}%
}}}
%  METADATA <id>263</id> 
\put(11832,-7313){\makebox(0,0)[lb]{\smash{\fontsize{22}{26.4}\normalfont\itshape {\color[rgb]{0,0,0}$yq$}%
}}}
%  METADATA <id>264</id> 
\put(12649,-7305){\makebox(0,0)[lb]{\smash{\fontsize{22}{26.4}\normalfont\itshape {\color[rgb]{0,0,0}$yq^2$}%
}}}
%  METADATA <id>265</id> 
\put(13474,-7310){\makebox(0,0)[lb]{\smash{\fontsize{22}{26.4}\normalfont\itshape {\color[rgb]{0,0,0}$yq^3$}%
}}}
%  METADATA <id>266</id> 
\put(14292,-7295){\makebox(0,0)[lb]{\smash{\fontsize{22}{26.4}\normalfont\itshape {\color[rgb]{0,0,0}$yq^4$}%
}}}
%  METADATA <id>267</id> 
\put(15087,-7280){\makebox(0,0)[lb]{\smash{\fontsize{22}{26.4}\normalfont\itshape {\color[rgb]{0,0,0}$yq^5$}%
}}}
%  METADATA <id>268</id> 
\put(12252,-6638){\makebox(0,0)[lb]{\smash{\fontsize{22}{26.4}\normalfont\itshape {\color[rgb]{0,0,0}$yq$}%
}}}
%  METADATA <id>269</id> 
\put(13069,-6630){\makebox(0,0)[lb]{\smash{\fontsize{22}{26.4}\normalfont\itshape {\color[rgb]{0,0,0}$yq^2$}%
}}}
%  METADATA <id>270</id> 
\put(13894,-6635){\makebox(0,0)[lb]{\smash{\fontsize{22}{26.4}\normalfont\itshape {\color[rgb]{0,0,0}$yq^3$}%
}}}
%  METADATA <id>35</id> 
\put(802,-4059){\makebox(0,0)[lb]{\smash{\fontsize{22}{26.4}\normalfont\itshape {\color[rgb]{0,0,0}$xq$}%
}}}
%  METADATA <id>36</id> 
\put(1702,-4059){\makebox(0,0)[lb]{\smash{\fontsize{22}{26.4}\normalfont\itshape {\color[rgb]{0,0,0}$xq$}%
}}}
%  METADATA <id>37</id> 
\put(2493,-4102){\makebox(0,0)[lb]{\smash{\fontsize{22}{26.4}\normalfont\itshape {\color[rgb]{0,0,0}$xq$}%
}}}
%  METADATA <id>38</id> 
\put(3312,-4102){\makebox(0,0)[lb]{\smash{\fontsize{22}{26.4}\normalfont\itshape {\color[rgb]{0,0,0}$xq$}%
}}}
%  METADATA <id>39</id> 
\put(4130,-4102){\makebox(0,0)[lb]{\smash{\fontsize{22}{26.4}\normalfont\itshape {\color[rgb]{0,0,0}$xq$}%
}}}
%  METADATA <id>40</id> 
\put(4948,-4102){\makebox(0,0)[lb]{\smash{\fontsize{22}{26.4}\normalfont\itshape {\color[rgb]{0,0,0}$xq$}%
}}}
%  METADATA <id>41</id> 
\put(5766,-4102){\makebox(0,0)[lb]{\smash{\fontsize{22}{26.4}\normalfont\itshape {\color[rgb]{0,0,0}$xq$}%
}}}
%  METADATA <id>42</id> 
\put(6585,-4102){\makebox(0,0)[lb]{\smash{\fontsize{22}{26.4}\normalfont\itshape {\color[rgb]{0,0,0}$xq$}%
}}}
%  METADATA <id>43</id> 
\put(7403,-4102){\makebox(0,0)[lb]{\smash{\fontsize{22}{26.4}\normalfont\itshape {\color[rgb]{0,0,0}$xq$}%
}}}
%  METADATA <id>44</id> 
\put(1211,-3309){\makebox(0,0)[lb]{\smash{\fontsize{22}{26.4}\normalfont\itshape {\color[rgb]{0,0,0}$xq^2$}%
}}}
%  METADATA <id>45</id> 
\put(2050,-3326){\makebox(0,0)[lb]{\smash{\fontsize{22}{26.4}\normalfont\itshape {\color[rgb]{0,0,0}$xq^2$}%
}}}
%  METADATA <id>46</id> 
\put(2868,-3326){\makebox(0,0)[lb]{\smash{\fontsize{22}{26.4}\normalfont\itshape {\color[rgb]{0,0,0}$xq^2$}%
}}}
%  METADATA <id>47</id> 
\put(3687,-3326){\makebox(0,0)[lb]{\smash{\fontsize{22}{26.4}\normalfont\itshape {\color[rgb]{0,0,0}$xq^2$}%
}}}
%  METADATA <id>48</id> 
\put(4505,-3326){\makebox(0,0)[lb]{\smash{\fontsize{22}{26.4}\normalfont\itshape {\color[rgb]{0,0,0}$xq^2$}%
}}}
%  METADATA <id>49</id> 
\put(5323,-3326){\makebox(0,0)[lb]{\smash{\fontsize{22}{26.4}\normalfont\itshape {\color[rgb]{0,0,0}$xq^2$}%
}}}
%  METADATA <id>50</id> 
\put(6142,-3373){\makebox(0,0)[lb]{\smash{\fontsize{22}{26.4}\normalfont\itshape {\color[rgb]{0,0,0}$xq^2$}%
}}}
%  METADATA <id>51</id> 
\put(6960,-3326){\makebox(0,0)[lb]{\smash{\fontsize{22}{26.4}\normalfont\itshape {\color[rgb]{0,0,0}$xq^2$}%
}}}
%  METADATA <id>52</id> 
\put(1641,-2617){\makebox(0,0)[lb]{\smash{\fontsize{22}{26.4}\normalfont\itshape {\color[rgb]{0,0,0}$xq^3$}%
}}}
%  METADATA <id>53</id> 
\put(2459,-2617){\makebox(0,0)[lb]{\smash{\fontsize{22}{26.4}\normalfont\itshape {\color[rgb]{0,0,0}$xq^3$}%
}}}
%  METADATA <id>54</id> 
\put(3278,-2617){\makebox(0,0)[lb]{\smash{\fontsize{22}{26.4}\normalfont\itshape {\color[rgb]{0,0,0}$xq^3$}%
}}}
%  METADATA <id>55</id> 
\put(4096,-2617){\makebox(0,0)[lb]{\smash{\fontsize{22}{26.4}\normalfont\itshape {\color[rgb]{0,0,0}$xq^3$}%
}}}
%  METADATA <id>56</id> 
\put(4914,-2617){\makebox(0,0)[lb]{\smash{\fontsize{22}{26.4}\normalfont\itshape {\color[rgb]{0,0,0}$xq^3$}%
}}}
%  METADATA <id>57</id> 
\put(5732,-2617){\makebox(0,0)[lb]{\smash{\fontsize{22}{26.4}\normalfont\itshape {\color[rgb]{0,0,0}$xq^3$}%
}}}
%  METADATA <id>58</id> 
\put(6551,-2617){\makebox(0,0)[lb]{\smash{\fontsize{22}{26.4}\normalfont\itshape {\color[rgb]{0,0,0}$xq^3$}%
}}}
%  METADATA <id>59</id> 
\put(2050,-1908){\makebox(0,0)[lb]{\smash{\fontsize{22}{26.4}\normalfont\itshape {\color[rgb]{0,0,0}$xq^4$}%
}}}
%  METADATA <id>60</id> 
\put(2868,-1908){\makebox(0,0)[lb]{\smash{\fontsize{22}{26.4}\normalfont\itshape {\color[rgb]{0,0,0}$xq^4$}%
}}}
%  METADATA <id>61</id> 
\put(3687,-1908){\makebox(0,0)[lb]{\smash{\fontsize{22}{26.4}\normalfont\itshape {\color[rgb]{0,0,0}$xq^4$}%
}}}
%  METADATA <id>62</id> 
\put(4505,-1908){\makebox(0,0)[lb]{\smash{\fontsize{22}{26.4}\normalfont\itshape {\color[rgb]{0,0,0}$xq^4$}%
}}}
%  METADATA <id>63</id> 
\put(5323,-1908){\makebox(0,0)[lb]{\smash{\fontsize{22}{26.4}\normalfont\itshape {\color[rgb]{0,0,0}$xq^4$}%
}}}
%  METADATA <id>64</id> 
\put(6142,-1908){\makebox(0,0)[lb]{\smash{\fontsize{22}{26.4}\normalfont\itshape {\color[rgb]{0,0,0}$xq^4$}%
}}}
%  METADATA <id>65</id> 
\put(2459,-1200){\makebox(0,0)[lb]{\smash{\fontsize{22}{26.4}\normalfont\itshape {\color[rgb]{0,0,0}$xq^5$}%
}}}
%  METADATA <id>66</id> 
\put(3278,-1200){\makebox(0,0)[lb]{\smash{\fontsize{22}{26.4}\normalfont\itshape {\color[rgb]{0,0,0}$xq^5$}%
}}}
%  METADATA <id>67</id> 
\put(4096,-1200){\makebox(0,0)[lb]{\smash{\fontsize{22}{26.4}\normalfont\itshape {\color[rgb]{0,0,0}$xq^5$}%
}}}
%  METADATA <id>68</id> 
\put(4914,-1200){\makebox(0,0)[lb]{\smash{\fontsize{22}{26.4}\normalfont\itshape {\color[rgb]{0,0,0}$xq^5$}%
}}}
%  METADATA <id>69</id> 
\put(5732,-1200){\makebox(0,0)[lb]{\smash{\fontsize{22}{26.4}\normalfont\itshape {\color[rgb]{0,0,0}$xq^5$}%
}}}
%  METADATA <id>70</id> 
\put(2868,-491){\makebox(0,0)[lb]{\smash{\fontsize{22}{26.4}\normalfont\itshape {\color[rgb]{0,0,0}$xq^6$}%
}}}
%  METADATA <id>71</id> 
\put(3687,-491){\makebox(0,0)[lb]{\smash{\fontsize{22}{26.4}\normalfont\itshape {\color[rgb]{0,0,0}$xq^6$}%
}}}
%  METADATA <id>72</id> 
\put(4505,-491){\makebox(0,0)[lb]{\smash{\fontsize{22}{26.4}\normalfont\itshape {\color[rgb]{0,0,0}$xq^6$}%
}}}
%  METADATA <id>73</id> 
\put(5323,-491){\makebox(0,0)[lb]{\smash{\fontsize{22}{26.4}\normalfont\itshape {\color[rgb]{0,0,0}$xq^6$}%
}}}
%  METADATA <id>97</id> 
\put(10226,-4031){\makebox(0,0)[lb]{\smash{\fontsize{22}{26.4}\normalfont\itshape {\color[rgb]{0,0,0}$y$}%
}}}
%  METADATA <id>98</id> 
\put(11111,-4031){\makebox(0,0)[lb]{\smash{\fontsize{22}{26.4}\normalfont\itshape {\color[rgb]{0,0,0}$y$}%
}}}
%  METADATA <id>99</id> 
\put(11930,-4031){\makebox(0,0)[lb]{\smash{\fontsize{22}{26.4}\normalfont\itshape {\color[rgb]{0,0,0}$y$}%
}}}
%  METADATA <id>100</id> 
\put(12748,-4031){\makebox(0,0)[lb]{\smash{\fontsize{22}{26.4}\normalfont\itshape {\color[rgb]{0,0,0}$y$}%
}}}
%  METADATA <id>101</id> 
\put(13566,-4031){\makebox(0,0)[lb]{\smash{\fontsize{22}{26.4}\normalfont\itshape {\color[rgb]{0,0,0}$y$}%
}}}
%  METADATA <id>102</id> 
\put(14385,-4031){\makebox(0,0)[lb]{\smash{\fontsize{22}{26.4}\normalfont\itshape {\color[rgb]{0,0,0}$y$}%
}}}
%  METADATA <id>103</id> 
\put(15203,-4031){\makebox(0,0)[lb]{\smash{\fontsize{22}{26.4}\normalfont\itshape {\color[rgb]{0,0,0}$y$}%
}}}
%  METADATA <id>104</id> 
\put(16021,-4031){\makebox(0,0)[lb]{\smash{\fontsize{22}{26.4}\normalfont\itshape {\color[rgb]{0,0,0}$y$}%
}}}
%  METADATA <id>105</id> 
\put(16839,-4031){\makebox(0,0)[lb]{\smash{\fontsize{22}{26.4}\normalfont\itshape {\color[rgb]{0,0,0}$y$}%
}}}
%  METADATA <id>106</id> 
\put(10669,-3322){\makebox(0,0)[lb]{\smash{\fontsize{22}{26.4}\normalfont\itshape {\color[rgb]{0,0,0}$y$}%
}}}
%  METADATA <id>107</id> 
\put(11554,-3322){\makebox(0,0)[lb]{\smash{\fontsize{22}{26.4}\normalfont\itshape {\color[rgb]{0,0,0}$y$}%
}}}
%  METADATA <id>108</id> 
\put(12373,-3322){\makebox(0,0)[lb]{\smash{\fontsize{22}{26.4}\normalfont\itshape {\color[rgb]{0,0,0}$y$}%
}}}
%  METADATA <id>109</id> 
\put(13191,-3322){\makebox(0,0)[lb]{\smash{\fontsize{22}{26.4}\normalfont\itshape {\color[rgb]{0,0,0}$y$}%
}}}
%  METADATA <id>110</id> 
\put(14009,-3322){\makebox(0,0)[lb]{\smash{\fontsize{22}{26.4}\normalfont\itshape {\color[rgb]{0,0,0}$y$}%
}}}
%  METADATA <id>111</id> 
\put(14828,-3322){\makebox(0,0)[lb]{\smash{\fontsize{22}{26.4}\normalfont\itshape {\color[rgb]{0,0,0}$y$}%
}}}
%  METADATA <id>112</id> 
\put(15646,-3322){\makebox(0,0)[lb]{\smash{\fontsize{22}{26.4}\normalfont\itshape {\color[rgb]{0,0,0}$y$}%
}}}
%  METADATA <id>113</id> 
\put(16464,-3322){\makebox(0,0)[lb]{\smash{\fontsize{22}{26.4}\normalfont\itshape {\color[rgb]{0,0,0}$y$}%
}}}
%  METADATA <id>114</id> 
\put(11079,-2614){\makebox(0,0)[lb]{\smash{\fontsize{22}{26.4}\normalfont\itshape {\color[rgb]{0,0,0}$y$}%
}}}
%  METADATA <id>115</id> 
\put(11964,-2614){\makebox(0,0)[lb]{\smash{\fontsize{22}{26.4}\normalfont\itshape {\color[rgb]{0,0,0}$y$}%
}}}
%  METADATA <id>116</id> 
\put(12783,-2614){\makebox(0,0)[lb]{\smash{\fontsize{22}{26.4}\normalfont\itshape {\color[rgb]{0,0,0}$y$}%
}}}
%  METADATA <id>117</id> 
\put(13601,-2614){\makebox(0,0)[lb]{\smash{\fontsize{22}{26.4}\normalfont\itshape {\color[rgb]{0,0,0}$y$}%
}}}
%  METADATA <id>118</id> 
\put(14419,-2614){\makebox(0,0)[lb]{\smash{\fontsize{22}{26.4}\normalfont\itshape {\color[rgb]{0,0,0}$y$}%
}}}
%  METADATA <id>119</id> 
\put(15238,-2614){\makebox(0,0)[lb]{\smash{\fontsize{22}{26.4}\normalfont\itshape {\color[rgb]{0,0,0}$y$}%
}}}
%  METADATA <id>120</id> 
\put(16056,-2614){\makebox(0,0)[lb]{\smash{\fontsize{22}{26.4}\normalfont\itshape {\color[rgb]{0,0,0}$y$}%
}}}
%  METADATA <id>121</id> 
\put(11488,-1905){\makebox(0,0)[lb]{\smash{\fontsize{22}{26.4}\normalfont\itshape {\color[rgb]{0,0,0}$y$}%
}}}
%  METADATA <id>122</id> 
\put(12373,-1905){\makebox(0,0)[lb]{\smash{\fontsize{22}{26.4}\normalfont\itshape {\color[rgb]{0,0,0}$y$}%
}}}
%  METADATA <id>123</id> 
\put(13192,-1905){\makebox(0,0)[lb]{\smash{\fontsize{22}{26.4}\normalfont\itshape {\color[rgb]{0,0,0}$y$}%
}}}
%  METADATA <id>124</id> 
\put(14010,-1905){\makebox(0,0)[lb]{\smash{\fontsize{22}{26.4}\normalfont\itshape {\color[rgb]{0,0,0}$y$}%
}}}
%  METADATA <id>125</id> 
\put(14828,-1905){\makebox(0,0)[lb]{\smash{\fontsize{22}{26.4}\normalfont\itshape {\color[rgb]{0,0,0}$y$}%
}}}
%  METADATA <id>126</id> 
\put(15647,-1905){\makebox(0,0)[lb]{\smash{\fontsize{22}{26.4}\normalfont\itshape {\color[rgb]{0,0,0}$y$}%
}}}
%  METADATA <id>127</id> 
\put(11897,-1196){\makebox(0,0)[lb]{\smash{\fontsize{22}{26.4}\normalfont\itshape {\color[rgb]{0,0,0}$y$}%
}}}
%  METADATA <id>128</id> 
\put(12782,-1196){\makebox(0,0)[lb]{\smash{\fontsize{22}{26.4}\normalfont\itshape {\color[rgb]{0,0,0}$y$}%
}}}
%  METADATA <id>129</id> 
\put(13601,-1196){\makebox(0,0)[lb]{\smash{\fontsize{22}{26.4}\normalfont\itshape {\color[rgb]{0,0,0}$y$}%
}}}
%  METADATA <id>130</id> 
\put(14419,-1196){\makebox(0,0)[lb]{\smash{\fontsize{22}{26.4}\normalfont\itshape {\color[rgb]{0,0,0}$y$}%
}}}
%  METADATA <id>131</id> 
\put(15237,-1196){\makebox(0,0)[lb]{\smash{\fontsize{22}{26.4}\normalfont\itshape {\color[rgb]{0,0,0}$y$}%
}}}
%  METADATA <id>132</id> 
\put(12306,-488){\makebox(0,0)[lb]{\smash{\fontsize{22}{26.4}\normalfont\itshape {\color[rgb]{0,0,0}$y$}%
}}}
%  METADATA <id>133</id> 
\put(13191,-488){\makebox(0,0)[lb]{\smash{\fontsize{22}{26.4}\normalfont\itshape {\color[rgb]{0,0,0}$y$}%
}}}
%  METADATA <id>134</id> 
\put(14010,-488){\makebox(0,0)[lb]{\smash{\fontsize{22}{26.4}\normalfont\itshape {\color[rgb]{0,0,0}$y$}%
}}}
%  METADATA <id>135</id> 
\put(14828,-488){\makebox(0,0)[lb]{\smash{\fontsize{22}{26.4}\normalfont\itshape {\color[rgb]{0,0,0}$y$}%
}}}
%  METADATA <id>170</id> 
\put(758,-10215){\makebox(0,0)[lb]{\smash{\fontsize{22}{26.4}\normalfont\itshape {\color[rgb]{0,0,0}$xq$}%
}}}
%  METADATA <id>171</id> 
\put(5719,-10218){\makebox(0,0)[lb]{\smash{\fontsize{22}{26.4}\normalfont\itshape {\color[rgb]{0,0,0}$xq^7$}%
}}}
%  METADATA <id>172</id> 
\put(6537,-10218){\makebox(0,0)[lb]{\smash{\fontsize{22}{26.4}\normalfont\itshape {\color[rgb]{0,0,0}$xq^8$}%
}}}
%  METADATA <id>173</id> 
\put(7359,-10258){\makebox(0,0)[lb]{\smash{\fontsize{22}{26.4}\normalfont\itshape {\color[rgb]{0,0,0}$xq^9$}%
}}}
%  METADATA <id>174</id> 
\put(1594,-10180){\makebox(0,0)[lb]{\smash{\fontsize{22}{26.4}\normalfont\itshape {\color[rgb]{0,0,0}$xq^2$}%
}}}
%  METADATA <id>175</id> 
\put(2457,-10195){\makebox(0,0)[lb]{\smash{\fontsize{22}{26.4}\normalfont\itshape {\color[rgb]{0,0,0}$xq^3$}%
}}}
%  METADATA <id>176</id> 
\put(3259,-10210){\makebox(0,0)[lb]{\smash{\fontsize{22}{26.4}\normalfont\itshape {\color[rgb]{0,0,0}$xq^4$}%
}}}
%  METADATA <id>177</id> 
\put(4092,-10195){\makebox(0,0)[lb]{\smash{\fontsize{22}{26.4}\normalfont\itshape {\color[rgb]{0,0,0}$xq^5$}%
}}}
%  METADATA <id>178</id> 
\put(4894,-10195){\makebox(0,0)[lb]{\smash{\fontsize{22}{26.4}\normalfont\itshape {\color[rgb]{0,0,0}$xq^6$}%
}}}
%  METADATA <id>202</id> 
\put(1234,-9491){\makebox(0,0)[lb]{\smash{\fontsize{22}{26.4}\normalfont\itshape {\color[rgb]{0,0,0}$xq$}%
}}}
%  METADATA <id>203</id> 
\put(6195,-9494){\makebox(0,0)[lb]{\smash{\fontsize{22}{26.4}\normalfont\itshape {\color[rgb]{0,0,0}$xq^7$}%
}}}
%  METADATA <id>204</id> 
\put(7013,-9494){\makebox(0,0)[lb]{\smash{\fontsize{22}{26.4}\normalfont\itshape {\color[rgb]{0,0,0}$xq^8$}%
}}}
%  METADATA <id>205</id> 
\put(2070,-9456){\makebox(0,0)[lb]{\smash{\fontsize{22}{26.4}\normalfont\itshape {\color[rgb]{0,0,0}$xq^2$}%
}}}
%  METADATA <id>206</id> 
\put(2933,-9471){\makebox(0,0)[lb]{\smash{\fontsize{22}{26.4}\normalfont\itshape {\color[rgb]{0,0,0}$xq^3$}%
}}}
%  METADATA <id>207</id> 
\put(3735,-9486){\makebox(0,0)[lb]{\smash{\fontsize{22}{26.4}\normalfont\itshape {\color[rgb]{0,0,0}$xq^4$}%
}}}
%  METADATA <id>208</id> 
\put(4568,-9471){\makebox(0,0)[lb]{\smash{\fontsize{22}{26.4}\normalfont\itshape {\color[rgb]{0,0,0}$xq^5$}%
}}}
%  METADATA <id>209</id> 
\put(5370,-9471){\makebox(0,0)[lb]{\smash{\fontsize{22}{26.4}\normalfont\itshape {\color[rgb]{0,0,0}$xq^6$}%
}}}
%  METADATA <id>210</id> 
\put(1638,-8798){\makebox(0,0)[lb]{\smash{\fontsize{22}{26.4}\normalfont\itshape {\color[rgb]{0,0,0}$xq$}%
}}}
%  METADATA <id>211</id> 
\put(6599,-8801){\makebox(0,0)[lb]{\smash{\fontsize{22}{26.4}\normalfont\itshape {\color[rgb]{0,0,0}$xq^7$}%
}}}
%  METADATA <id>213</id> 
\put(2474,-8763){\makebox(0,0)[lb]{\smash{\fontsize{22}{26.4}\normalfont\itshape {\color[rgb]{0,0,0}$xq^2$}%
}}}
%  METADATA <id>214</id> 
\put(3337,-8778){\makebox(0,0)[lb]{\smash{\fontsize{22}{26.4}\normalfont\itshape {\color[rgb]{0,0,0}$xq^3$}%
}}}
%  METADATA <id>215</id> 
\put(4139,-8793){\makebox(0,0)[lb]{\smash{\fontsize{22}{26.4}\normalfont\itshape {\color[rgb]{0,0,0}$xq^4$}%
}}}
%  METADATA <id>216</id> 
\put(4972,-8778){\makebox(0,0)[lb]{\smash{\fontsize{22}{26.4}\normalfont\itshape {\color[rgb]{0,0,0}$xq^5$}%
}}}
%  METADATA <id>217</id> 
\put(5774,-8778){\makebox(0,0)[lb]{\smash{\fontsize{22}{26.4}\normalfont\itshape {\color[rgb]{0,0,0}$xq^6$}%
}}}
%  METADATA <id>218</id> 
\put(2028,-8086){\makebox(0,0)[lb]{\smash{\fontsize{22}{26.4}\normalfont\itshape {\color[rgb]{0,0,0}$xq$}%
}}}
%  METADATA <id>219</id> 
\put(2864,-8051){\makebox(0,0)[lb]{\smash{\fontsize{22}{26.4}\normalfont\itshape {\color[rgb]{0,0,0}$xq^2$}%
}}}
%  METADATA <id>220</id> 
\put(3727,-8066){\makebox(0,0)[lb]{\smash{\fontsize{22}{26.4}\normalfont\itshape {\color[rgb]{0,0,0}$xq^3$}%
}}}
%  METADATA <id>221</id> 
\put(4529,-8081){\makebox(0,0)[lb]{\smash{\fontsize{22}{26.4}\normalfont\itshape {\color[rgb]{0,0,0}$xq^4$}%
}}}
%  METADATA <id>222</id> 
\put(5362,-8066){\makebox(0,0)[lb]{\smash{\fontsize{22}{26.4}\normalfont\itshape {\color[rgb]{0,0,0}$xq^5$}%
}}}
%  METADATA <id>223</id> 
\put(6164,-8066){\makebox(0,0)[lb]{\smash{\fontsize{22}{26.4}\normalfont\itshape {\color[rgb]{0,0,0}$xq^6$}%
}}}
%  METADATA <id>224</id> 
\put(2448,-7366){\makebox(0,0)[lb]{\smash{\fontsize{22}{26.4}\normalfont\itshape {\color[rgb]{0,0,0}$xq$}%
}}}
%  METADATA <id>225</id> 
\put(3284,-7331){\makebox(0,0)[lb]{\smash{\fontsize{22}{26.4}\normalfont\itshape {\color[rgb]{0,0,0}$xq^2$}%
}}}
%  METADATA <id>226</id> 
\put(4147,-7346){\makebox(0,0)[lb]{\smash{\fontsize{22}{26.4}\normalfont\itshape {\color[rgb]{0,0,0}$xq^3$}%
}}}
%  METADATA <id>227</id> 
\put(4949,-7361){\makebox(0,0)[lb]{\smash{\fontsize{22}{26.4}\normalfont\itshape {\color[rgb]{0,0,0}$xq^4$}%
}}}
%  METADATA <id>228</id> 
\put(5782,-7346){\makebox(0,0)[lb]{\smash{\fontsize{22}{26.4}\normalfont\itshape {\color[rgb]{0,0,0}$xq^5$}%
}}}
%  METADATA <id>229</id> 
\put(2860,-6668){\makebox(0,0)[lb]{\smash{\fontsize{22}{26.4}\normalfont\itshape {\color[rgb]{0,0,0}$xq$}%
}}}
%  METADATA <id>230</id> 
\put(3696,-6633){\makebox(0,0)[lb]{\smash{\fontsize{22}{26.4}\normalfont\itshape {\color[rgb]{0,0,0}$xq^2$}%
}}}
%  METADATA <id>231</id> 
\put(4559,-6648){\makebox(0,0)[lb]{\smash{\fontsize{22}{26.4}\normalfont\itshape {\color[rgb]{0,0,0}$xq^3$}%
}}}
%  METADATA <id>232</id> 
\put(5361,-6663){\makebox(0,0)[lb]{\smash{\fontsize{22}{26.4}\normalfont\itshape {\color[rgb]{0,0,0}$xq^4$}%
}}}
%  METADATA <id>233</id> 
\put(10212,-10203){\makebox(0,0)[lb]{\smash{\fontsize{22}{26.4}\normalfont\itshape {\color[rgb]{0,0,0}$yq$}%
}}}
%  METADATA <id>234</id> 
\put(11029,-10195){\makebox(0,0)[lb]{\smash{\fontsize{22}{26.4}\normalfont\itshape {\color[rgb]{0,0,0}$yq^2$}%
}}}
\end{picture}}
		\caption{Two weight assignments on a semihexagon.}
		\label{semihex2}
\end{figure}
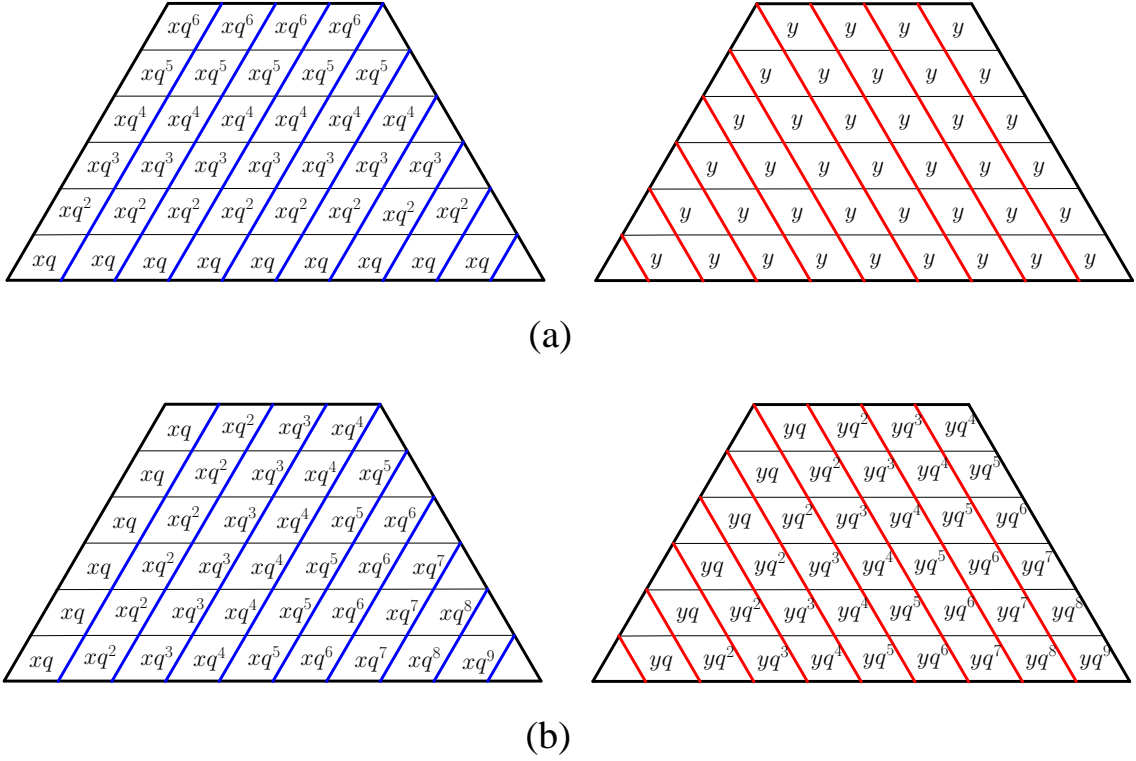

\begin{lem}\label{semilem1} We have
\[\T(\mathcal{SH}_{a,b}(\wt_{x,y}(q)))=x^{ab} \cdot \left(\frac{y}{x}\right)^{\sum_{i=1}^{b}(r_i-i)} \cdot s_{\lambda(s_1,s_2,\dots,s_a)}(q,q^2,\dots,q^{a}),\]
where $\{r_1,r_2,\dots,r_b\}=[a+b]\setminus\{s_1,s_2,\dots,s_a\}$.
\end{lem}
\begin{proof}
 Encode each lozenge tiling $L$ of the region as a family of $b$ disjoint lozenge paths $P_1,P_2,\dots P_b$ connecting the top and bottom of the region as in Figure \ref{semihex}(b). Each path $P_i$ has exactly $r_i-i$ left-tilting lozenges and $a-r_i+i$ right-tilting lozenges. The lozenges outside the $b$ paths $P_i$ are all vertical. We now change the weights of the left- and right-tilting lozenges as follows.  Divide the weight of each left-tilting lozenge by $y$ and the weight of each right-tilting lozenge by $x$.  It means, we divided the weight of each tiling $L$ by a constant $x^{\sum_{i=1}^{b}(a-r_i+i)}y^{\sum_{i}^b(r_i-i)}$. This modification reduces our weight assignment to the weight assignment $\wt_{1,1}(q)$ considered in Figure \ref{semihex}.    Thus \[\T(\mathcal{SH}_{a,b}(\wt_{x,y}(q)))=x^{ab} \cdot (y/x)^{\sum_{i=1}^{b}(r_i-i)} \cdot \T(\mathcal{SH}_{a,b}(\wt_{1,1}(q))).\] Then the lemma now follows from identity (\ref{Gp}).
\end{proof}

The second weight assignment on $\mathcal{SH}_{a,b}(s_1,\dotsc,s_a)$ is shown in Figure \ref{semihex2}(b). In particular,  in each row of the region, the $j$-th right-tilting lozenge counted from the left is weighted by $xq^{j}$, and the $j$-th  left-tilting lozenge is weighted by $yq^{j}$. Denote the resulting weighted region by $\mathcal{SH}_{a,b}(\overline{\wt}_{x,y}(q))$.

\begin{lem}\label{semilem2} We have
\begin{align}
\T(\mathcal{SH}_{a,b}(\overline{\wt}_{x,y}(q)))%%&=x^{ab} \cdot \left(\frac{y}{x}\right)^{\sum_{i=1}^{b}(r_i-i)}\cdot q^{ab(a+b+1)-\sum_{i=1}^{b}(r_i-i)}\notag\\
%%&\quad \quad \quad \quad\times s_{\lambda(s_1,s_2,\dots,s_a)}(q^{-1},q^{-2},\dots,q^{-a})\\
&=x^{ab} \cdot \left(\frac{y}{x}\right)^{\sum_{i=1}^{b}(r_i-i)}\cdot q^{ab(b-a)/2+a\sum_{i=1}^{b}(r_i-i)}\notag \\
&\quad \quad  \quad\quad \times s_{\lambda(s_1,s_2,\dots,s_a)}(q,q^{2},\dots,q^{a}).
\end{align}

\end{lem}
The proof of Lemma \ref{semilem2} is deferred to Section \ref{Sec:Complete}.

\medskip
In \cite{Ful2}, Fulmek provides an elegant combinatorial proof for Ciucu’s conjecture on a Schur function identity. Krattenthaler then gives several extensions of the identity and uses them to enumerate perfect matchings of a ``holey Aztec rectangle" graph \cite{Krat1}.

 \begin{thm}[Theorem 5 in \cite{Krat1}]\label{thr5}Let $T=\{t_1,t_2,\dots,t_{2m+d}\}$ be a set of positive integers with $t_1<t_2<\cdots<t_{2m+d}.$ Let $\mathcal{E}=\{t_2,t_4,t_6,\dots\}$ and $\mathcal{O}=\{t_1,t_3,t_5,\dots\}$ denote the sets of even-indexed and odd-indexed elements of $T$, respectively. Then \begin{align}&\sum_{(A,B)}s_{\lambda(A)}(X_{m})s_{\lambda(B)}(X_{m+d})=2^{m}\notag\\
&\times\displaystyle\sum _{1\leq k_1<...<k_{\lfloor d/2\rfloor}\leq m+\lfloor d/2\rfloor }s_{\lambda\left(\mathcal{E}\setminus \left\{t_{2k_1},...,t_{2k_{\lfloor d/2\rfloor}}\right\}\right)}(X_{n}).s_{\lambda\left(\mathcal{O}\cup \left\{t_{2k_1},...,t_{2k_{\lfloor d/2\rfloor}}\right\}\right)}(X_{n+d}),\end{align}
where the sum on the left-hand side is taken over all pairs of disjoint sets $(A,B)$ whose union is $T$ and whose cardinalities are given by $|A|=m$ and $|B|=m+d.$	
\end{thm}

\begin{thm}[Theorem $6$ in \cite{Krat1}]\label{thr6} We have
\begin{align}\sum_{0\leq k_1<k_2<\cdots<k_s\leq m}&y^{\sum_{i=1}^sk_i} \prod_{1\leq i<j\leq s}(q^{k_j}-q^{k_i})^2 \prod_{i=1}^s\frac{(x;q)_{k_i}(y;q)_{m-k_i}}{(q;q)_{k_i}(q;q)_{m-k_i}}\notag \\
&=q^{\binom{s}{3}}y^{\binom{s}{2}}\prod_{i=1}^s\frac{(x;q)_{i-1}(y;q)_{i-1}(xyq^{i+s-2};q)_{m-s+1}(q;q)_{i-1}}{(q;q)_{m-i+1}}.\end{align}
\end{thm}

Denote by $[n]:=\{1,2,\dots,n\}$ the set of the first $n$ positive integers. For any nonempty set of integers $A$, we denote by $n+A:=\{n+a \mid a \in A\}$ the shifted version of $A$.  We will need the following consequences of the preceding theorems of Krattenthaler.

\begin{lem}\label{Schurlem2} Let $m$ be a positive integer and $d$ a nonnegative integer. 
If $d$ is even, say $d=2s$, then %({\color{red}simplify})
\begin{align}
&\sum_{(A,B)}s_{\lambda(A)}(q,q^2,\dots,q^{m})s_{\lambda(B)}(q,q^2,\dots,q^{m+d})\notag\\
%%&= 2^m \frac{(1-q)^s\cdot q^{2s(s-1)+s}\cdot q^{m(m+d)} q^{2\binom{s}{3}} q^{\binom{s}{2}} q^{4\binom{m+s+1}{3}}}{q^{\binom{m+1}{3}+\binom{m+d+1}{3}} q^{\binom{m+s}{2}}} \frac{\Hf_{q^2}(m+s)^2}{ \Hf_q(m)\Hf_q(m+d)}\notag\\
%%&\quad \quad \times \prod_{i=1}^{s} \frac{(q^3;q^2)_{i-1}(q;q^2)_{i-1}(q^{2(i+s)};q^2)_{m}(q^2;q^2)_{i-1}}{(q^2;q^2)_{m+s-i}},\\
%%&\quad=2^{m}(1-q)^sq^{m/6 + m^2/2 + m^3/3 - (2 s)/3 + m s + m^2 s + s^2 - s^3/3} \frac{\Hf_{q^2}(m+s)^2}{ \Hf_q(m)\Hf_q(m+2s)}\notag\\
%%&\quad \quad \times \prod_{i=1}^{s} \frac{(q^{3};q^2)_{i-1}(q;q^2)_{i-1}(q^{2(i+s)};q^2)_{m}(q^2;q^2)_{i-1}}{(q^2;q^2)_{m+s-i}}\\
&\quad=2^{m}(1-q)^sq^{\frac{m(m+1)(2m+6s+1)}{6}-\frac{s(s-1)(s-2)}{3}}\times \frac{\Hf_{q^2}(m+s)^2}{ \Hf_q(m)\Hf_q(m+2s)}\notag\\
&\quad \quad \times \prod_{i=1}^{s} \frac{(q^{3};q^2)_{i-1}(q;q^2)_{i-1}(q^{2(i+s)};q^2)_{m}(q^2;q^2)_{i-1}}{(q^2;q^2)_{m+s-i}},
\end{align}
%(-1)^{-\binom{m}{2}-\binom{m+2s}{2}}=(-1)^s
where the sum on the left-hand side is taken over all  pairs of disjoint sets $(A,B)$ whose union is $[2m+d]$ and whose cardinalities are given by $|A|=m$ and $|B|=m+d.$	

If $d$ is odd, say $d=2s+1$, then %({\color{red}simplify})
\begin{align}
&\sum_{A,B}s_{\lambda(A)}(q,q^2,\dots,q^{m})s_{\lambda(B)}(q,q^2,\dots,q^{m+d})\notag\\
%%&=2^{m} \frac{(1-q)^{2s}\cdot q^{2s(s-1)+3s}\cdot q^{m(m+d)}q^{2\binom{s}{3}}q^{3\binom{s}{2}}q^{\binom{m+s+1}{2}+4\binom{m+s+1}{3}}}{q^{\binom{m+1}{3}+\binom{m+d+1}{3}}}\\
%%&\frac{\Hf_{q^2}(m+s)H_{q^2}(m+s+1)}{ \Hf_q(m)\Hf_q(m+d)}\notag\\
%%&\quad \quad \times \prod_{i=1}^{s} \frac{(q^3;q^2)_{i-1}(q^3;q^2)_{i-1}(q^{2(i+s+1)};q^2)_{m}(q^2;q^2)_{i-1}}{(q^2;q^2)_{m+s-i}}.\\
%%&=2^{m}(1-q)^{2s}q^{(2 m)/3 + m^2 + m^3/3 - (2 s)/3 + m s + m^2 s + s^2 - s^3/3} \\
%%&\frac{\Hf_{q^2}(m+s)H_{q^2}(m+s+1)}{ \Hf_q(m)\Hf_q(m+2s+1)}\notag\\
%%&\quad \quad \times \prod_{i=1}^{s} \frac{(q^{3};q^2)_{i-1}(q^3;q^2)_{i-1}(q^{2(i+s+1)};q^2)_{m}(q^2;q^2)_{i-1}}{(q^2;q^2)_{m+s-i}}\\
&=2^{m}(1-q)^{2s}q^{\frac{m(m+1)(m+3s+2)}{3}-\frac{s(s-1)(s-2)}{3}} \times\frac{\Hf_{q^2}(m+s)H_{q^2}(m+s+1)}{ \Hf_q(m)\Hf_q(m+2s+1)}\notag\\
&\quad \quad \times \prod_{i=1}^{s} \frac{(q^{3};q^2)_{i-1}(q^3;q^2)_{i-1}(q^{2(i+s+1)};q^2)_{m}(q^2;q^2)_{i-1}}{(q^2;q^2)_{m+s-i}}
\end{align}
 \end{lem}
 %(-1)^{m+s-\binom{m}{2}-\binom{m+2s+1}{2}}=1
 We also need the following lemma to simplify our formulas.
\begin{lem}\label{Schurlem} For any positive integers $m,n,a,b$
\begin{align}
s_{\lambda(A')}&(q,q^2,\dots, q^{m+n+a})s_{\lambda(B')}(q,q^2,\dots, q^{m+n+b}) \notag \\
&=q^{abn+(a+b)\binom{n+1}{2}} \prod_{i=1}^{a+b}(q^i;q)_m (q^{a+b-i+1};q)_n \cdot  \prod_{j=1}^{n} (q^{a+b+j};q)_m^2.\notag \\
& \times \frac{\Hf_q(a)\Hf_q(b)\Hf_q(m)^2\Hf_q(n)^2}{\Hf_q(m+n+a)\Hf_q(m+n+b)} \cdot  s_{\lambda(A)}(q,q^2,\dots, q^{a})s_{\lambda(B)}(q,q^2,\dots, q^{b}),
\end{align}
%where 
%\begin{align*}
%R&(m,n,a,b)=abn+(a+b)\binom{n+1}{2}
%&=n(a+b)+\binom{a+1}{3}+\binom{b+1}{3}-\binom{m+n+a+1}{3}\\
%&-\binom{m+n+b+1}{3}+m\binom{a}{2}+m\binom{b}{2}+2(m+a+b)\binom{n}{2}+2\binom{m+1}{3}\\
%&+2\binom{n+1}{3}+(a+b)\binom{m+1}{2}+mn(a+b)+n\binom{a+b+1}{2}+m(m+1)n\\
%&=abn+(a+b)\binom{n+1}{2}.
%\end{align*}
%and 
%\begin{align*}
%K=\frac{H(a)H(b)H(n)^2H(m+n)H(m+n+d)^2 }{H(m+n+a)H(m+n+b)H(n+d)H(m+d)^2}.
%&\frac{H(a)H(b)H(m)^2H(n)^2 }{H(m+n+a)H(m+n+b)}\\
%&\times \frac{H(m+n)}{H(m)^2} \frac{H(n+d)}{H(d)^2} \left(\frac{H(m+n+d)/H(n+d)}{H(m+d)/H(d)}\right)^2.
%\end{align*}
where  $A'=[m]\cup(m+A)\cup(m+a+b+[n])$ and $B'=[m]\cup (m+B)\cup(m+a+b+[n])$ with $A$ and $B$ are two disjoint sets whose union is $[a+b]$ and whose cardinalities are  $|A|=a$ and $|B|=b$.
\end{lem}

The proofs of Lemmas \ref{Schurlem2} and \ref{Schurlem}   are provided in Section \ref{Sec:Complete}.

\section{Proof of the main theorem (Theorem \ref{main})} \label{Sec:ProofMatching}

\begin{proof}[Proof of Theorem \ref{main}.] 
To find the tiling generating function of the cruciform region $\mathcal{C}_{m,n}^{a,b,c,d}$ in Theorem \ref{main}, we will compute the matching generating function of its dual graph. We now assign weights to the edges of the cruciform graph $C=C_{m,n}^{a,b,c,d}$ induced by the weight assignment on the cruciform region. The reader should compare the two depictions of the same graph in Figures \ref{DC} and \ref{DCW}. We label the shaded diamonds in the graph in the same way as in the weighted Aztec rectangle graph $AR_{m,n}(\wt^{e,f}_{h,g}(q))$, as follows. The bottom-left diamond is labeled $0$, and the label increases by $1$ whenever we move one step to the right or one step upward. Although some ``partial" diamonds on the boundary of the middle part contribute only one or two edges to the graph, we label them in the same way. The diamond with label $k$ has edge weights $e,f,gq^k,hq^k$ in clockwise order, starting from the northwest edge. For the partial diamonds in the middle part, we consider only the weights of the edges that belong to the graph $C$. For example, the partial diamond of label $c-d-1=2$ on the leftmost side of the graph only contributes one edge of weight $f$ (its northeast edge) to the graph $C$. We denote this weighted cruciform graph by $C^{a,b,c,d}_{m,n}(\wt^{e,f}_{h,g}(q))$. The remainder of the proof consists of computing the matching generating function $\M(C^{a,b,c,d}_{m,n}(\wt^{e,f}_{h,g}(q)))$in the following steps:
\begin{itemize}
\item We first use our transformation in the Mega-Sanwich Lemmas (Lemmas \ref{megalem1} and \ref{megalem2}) to transform our cruciform graph into a connected sum of two semi-honeycomb graphs.
\item Next, we express the matching generating function of the resulting graph as a Schur-function sum, using Lemma \ref{semilem2}
\item Finally, we apply Lemma \ref{semilem2} to obtain the desired product formula for the matching generating function.
\end{itemize}

\begin{figure}[ht]\centering
		\includegraphics[width = 14cm]{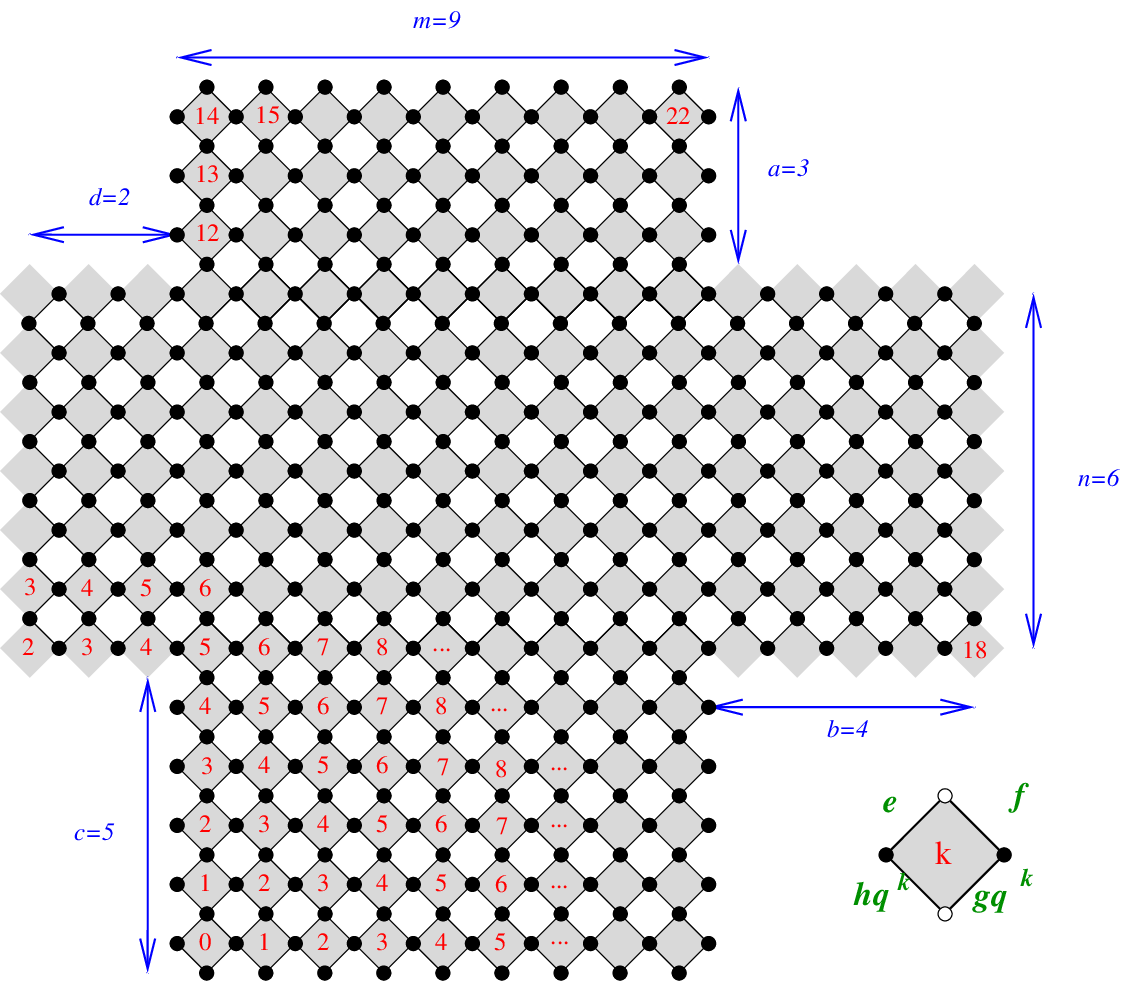}
		\caption{Assigning weights to the edges of the cruciform graph $C_{9,6}^{3,4,5,2}$.}
		\label{DCW}
\end{figure}

Let us proceed with the above steps.

\medskip

\textbf{Step 1.}   \emph{Transform the cruciform graph $C^{a,b,c,d}_{m,n}(\wt^{e,f}_{h,g}(q))$.}

\medskip

We divide the cruciform graph  $C= C^{a,b,c,d}_{m,n}(\wt^{e,f}_{h,g}(q))$ into three Aztec rectangle graphs, together with two red zigzags, using four dotted horizontal lines, as in Figure \ref{threeapp}. The top part is an Aztec rectangle graph $AR_{a,m}$,  the middle part is an Aztec rectangle $AR_{n,m+b+d+1}$, and the bottom part is an Aztec rectangle $AR_{c,m}$.

 Each  Aztec rectangle graph carries a weight assignment of the type appearing in Lemma \ref{lm2} or Lemma \ref{lem3}. 
 By investigating the edge weights of the diamond in the lower-left corner of each Aztec rectangle graph, we can determine the particular weight assignments. This is immediate for the top and the bottom Aztec rectangle graphs based on the labels of the shaded diamonds, as in Figure \ref{DCW}: the lower-left diamond of the bottom Aztec rectangle graph has label $0$, the lower-left diamond of the top Aztec rectangle graph has label $c+n+1$. Thus the bottom graph has weight assignment $\wt_{h,g}^{e,f}(q)$, and the top graph has similar weight assignment $\wt_{hq^{c+n+1},gq^{c+n+1}}^{e,f}(q)$.
 
 However, the lower left diamond of the middle Aztec rectangle graph is \emph{not} a shaded diamond in Figure \ref{DCW}. It is the white diamond adjacent to four shaded diamonds with labels $c-d$, $c-d+1$, $c-d$, $c-d-1$ (in cyclic order, starting from the northwest one), and its edge weights are determined by these four adjacent shaded diamonds. We then obtain the weight assignment of the middle graph as $\overline{\wt}_{hq^{c-d+1},gq^{c-d}}^{e,f}$. 

\begin{figure}[ht]\centering
		\includegraphics[width = 12cm]{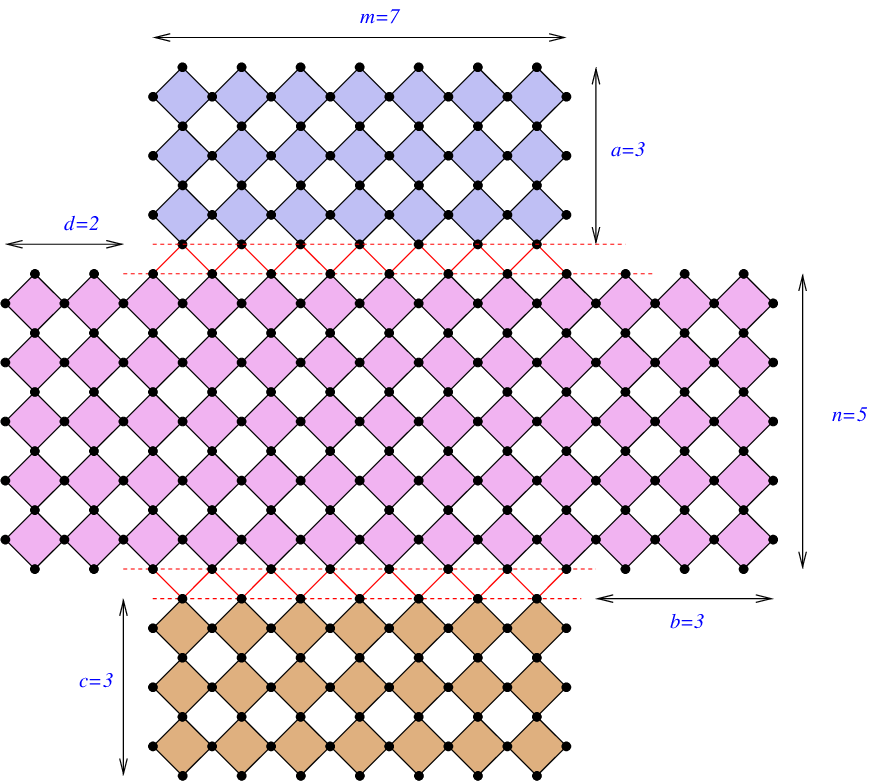}
		\caption{Dividing the cruciform graph $C_{7,5}^{3,3,3,2}$ into three parts.}
		\label{threeapp}
\end{figure}

Next, we apply the Mega-Sandwich Lemma 1 (Lemma \ref{megalem1}) to the top and the bottom Aztec rectangle graphs, and the Mega-Sandwich Lemma 2 (Lemma \ref{megalem2}) to the middle part. We obtain the connected sum $D$ of three $K$-typed graphs as in Figure \ref{threeapp2}. In particular, the top Aztec rectangle graph $AR_{a,m}(\wt_{hq^{c+n+1},gq^{c+n+1}}^{e,f}(q))$ is replaced by the graph $^|_|K_{a,m}(\wt_{hq^{c+n+1},gq^{c+n+1}}^{e,f}(q))$ (see the part above the horizontal dotted lines); the middle Aztec rectangle graph $AR_{n,m+b+d+1}(\overline{\wt}_{hq^{c-d+1},gq^{c-d}}^{e,f})$ is replaced by the graph $^|_|K_{n,m+b+d+1}(\overline{\wt}_{hq^{c-d+1},gq^{c-d}}^{e,f})$ (see the part between the upper two dotted lines and the lower dotted lines); the bottom Aztec rectangle graph $AR_{a,m}(\wt_{h,g}^{e,f}(q))$ is replaced by $^|_|K_{a,m}(\wt_{h,g}^{e,f}(q))$ (see the part below the four horizontal dotted lines). Collecting the multiplicative factors arising from the three transformations, we obtain 
\begin{align}\label{maineqb}
\M(C)%%=&\prod_{i=1}^{c} (egq^{i-1}+hf)^{c-i+1} \cdot q^{\binom{c}{2}+\binom{c+1}{3}}\notag\\
%%&\times \prod_{i=1}^{n} (egq^{c-d}+hq^{c-d+1}fq^{i-1})^{n-i+1} \cdot q^{(m+b+d)n(n+1)/2}\notag\\
%%&\times \prod_{i=1}^{a} (e\cdot gq^{n+c+1} \cdot q^{i-1}+hq^{n+c+1}f)^{a-i+1} \cdot q^{\binom{a}{2}+\binom{a+1}{3}} \cdot \M(D)\notag\\
=&\prod_{i=1}^{c} (egq^{i-1}+hf)^{c-i+1}   \prod_{i=1}^{n} (eg+hfq^{i})^{n-i+1}  \prod_{i=1}^{a} (eg q^{i-1}+hf)^{a-i+1}\notag\\
&\times q^{\binom{c}{2}+\binom{c+1}{3}+\binom{a}{2}+\binom{a+1}{3}+(m+b+d)n(n+1)/2+(n+c+1)\binom{a+1}{2}+(c-d)\binom{n+1}{2}} \cdot \M(D).
\end{align}

\begin{figure}[ht]\centering
		\includegraphics[width = 10cm]{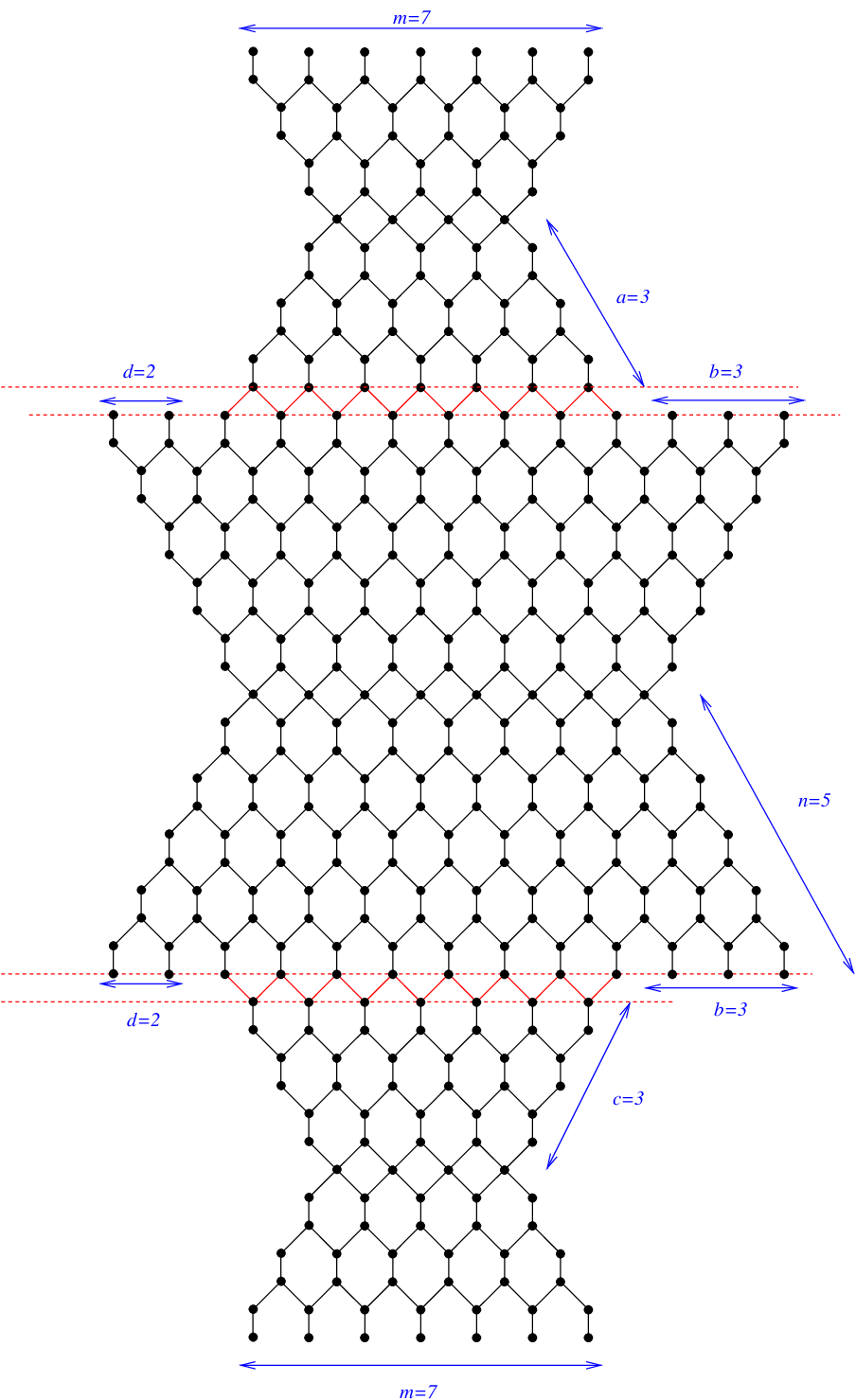}
		\caption{Graph $D$ is obtained by applying Mega-sandwich Lemmas to the three parts of the cruciform graph in Figure \ref{threeapp}.}
		\label{threeapp2}
\end{figure}

Next, we remove vertical forced edges of weight $1$ from the graph $D$ (indicated by circled edges in the graph in Figure \ref{threeapp3}) to obtain graph $E$ (see the graph bounded by the bold contour in Figure \ref{threeapp3}). Graph $E$ is the connected sum of two honeycomb graphs: the honeycomb graph $H_1$ of side-lengths $m-a,a+b+1,n-b,m+b+d-n+1,n-d,a+d+1$ and the honeycomb graph $H_2$ of side-lengths $m+b+d-n+1,n-b,b+c+1,m-c,c+d+1,n-d$; the connected sum is taken along the sides of length $m+b+d-n+1$ of the two honeycomb graphs (the sides running along the horizontal dotted line).

\begin{figure}[ht]\centering
		\includegraphics[width = 10cm]{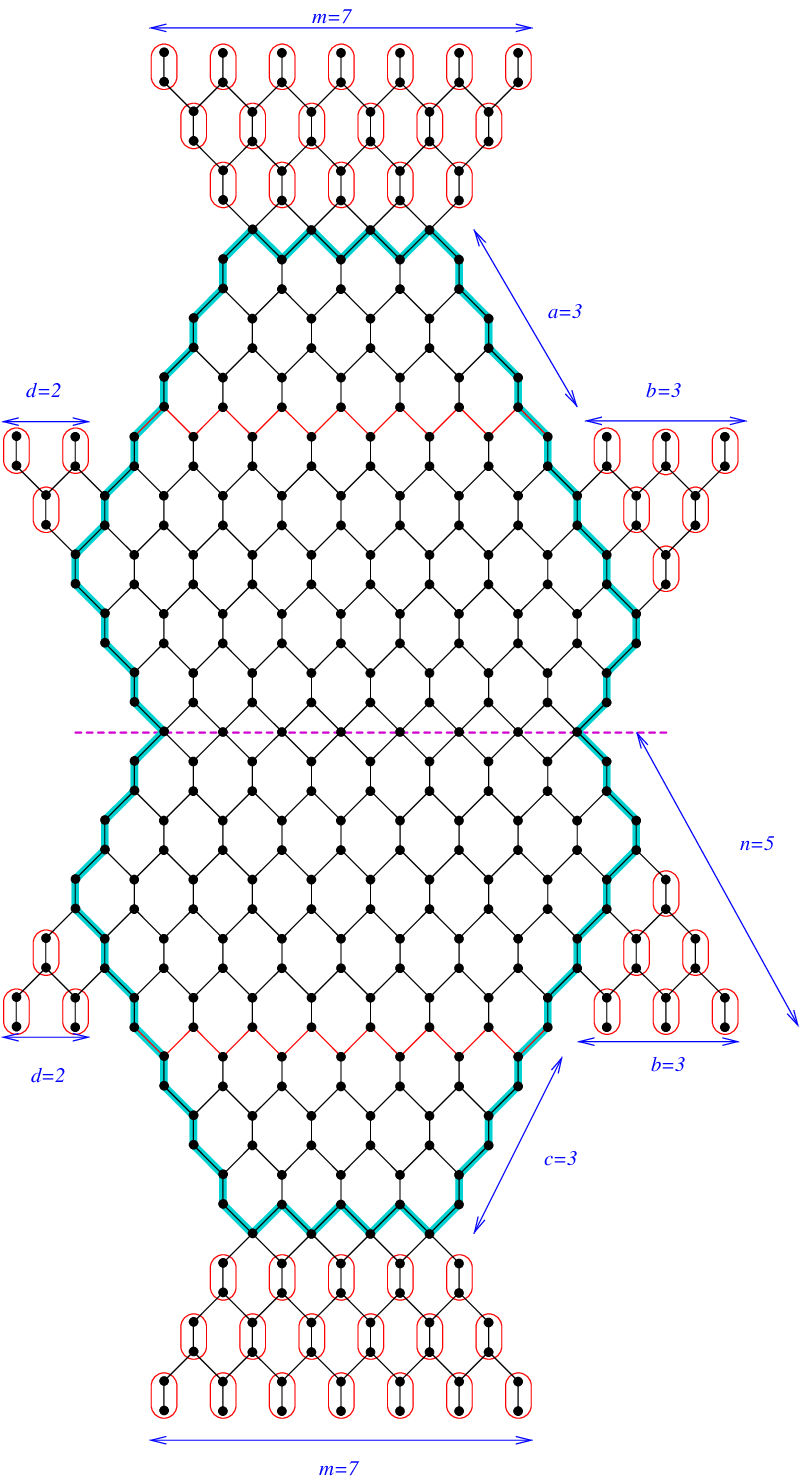}
		\caption{Graph $E$ is obtained by removing forced edges from graph $D$ in Figure \ref{threeapp2}.}
		\label{threeapp3}
\end{figure}

The graph $E$ can also be obtained by removing the forced edges from a different graph $F$, as shown in Figure \ref{threeapp4} (the forced edges are also indicated by circled ones).  The graph $F$ is the connected sum of two weighted semi-honeycomb graphs with defects. More precisely, the top part is the semi-honeycomb graph $SH_1=SH_{n+a+1,m-a}$ with the bottom vertices in the positions $1,2,\dots,n-d, \,m+b+2, m+b+3,\dots, m+n+1$ (listed from left to right) removed, and the bottom part is the upside-down version of the semi-honeycomb graph $SH_2=SH_{n+c+1,m-c}$ with bottom vertices in the same positions removed. The connected sum is taken over the remaining bottom vertices of the two semi-honeycombs.

\begin{figure}[ht]\centering
		\includegraphics[width = 14cm]{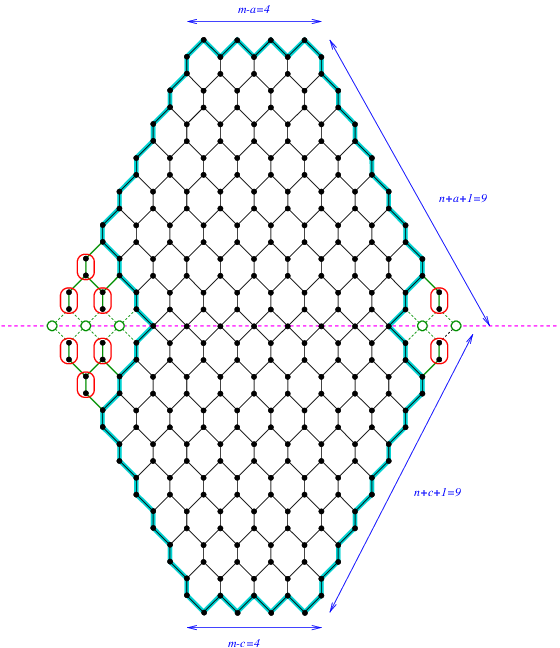}
		\caption{Graph $E$ (in Figure \ref{threeapp3} is also obtained by removing forced edges from graph $F$.}
		\label{threeapp4}
\end{figure}

Transformations from $D$ to $E$ and from $ F$ to $E$ introduce no additional factor, as we removed only forced edges of weight $1$. Thus $\M(D)=\M(E)=\M(F)$. Therefore, by (\ref{maineqb}), we have 
\begin{align}\label{maineqc}
\M(C)=&\prod_{i=1}^{c} (egq^{i-1}+hf)^{c-i+1}   \prod_{i=1}^{n} (eg+hfq^{i})^{n-i+1}  \prod_{i=1}^{a} (eg  q^{i-1}+hf)^{a-i+1}\notag\\
&\times q^{\binom{c}{2}+\binom{c+1}{3}+\binom{a}{2}+\binom{a+1}{3}+(m+b+d)\binom{n+1}{2}+(n+c+1)\binom{a+1}{2}+(c-d)\binom{n+1}{2}} \cdot \M(F).
\end{align}

\medskip

\textbf{Step 2.} \emph{Write $\M(F)$ as a Schur-function formula.}

\medskip
We now consider the weight assignment on the edges of $F$. As shown in Figure \ref{threeapp4}, the graph $F$ is divided by the dotted line into two parts. The top part is the semi-honeycomb $SH_1$, weighted as in Figure \ref{threeapp5a}. In particular, we view $SH_1$ as $n+a+1$ rows of baseless triangles. The triangles on the $i$-th row from the top have their left edges weighted by $eq^i$ and the right edges weighted by $f$ (these triangles are labeled by $i$ in the top picture). %It is exactly the weight assignment $\wt_{eq^{-n-1},f}(q)$ in Lemma \ref{semilem1}.

 The lower part is the (upside-down) semi-honeycomb $SH_2$, with edge weights as shown in Figure \ref{threeapp5b}. In particular, the $j$-th upside-down triangle (counted from the left) in each row has the left and right edges weighted by $hq^{c+j-1}$ and $gq^{c+j-1}$, respectively (these triangles are labeled by $c+j-1$). %It is the weight assignment $\overline{\wt}_{hq^{c-1},gq^{c-1}}(q)$ in Lemma \ref{semilem2}.

\begin{figure}[ht]\centering
\resizebox{14cm}{!}{\begin{picture}(0,0)%
\includegraphics{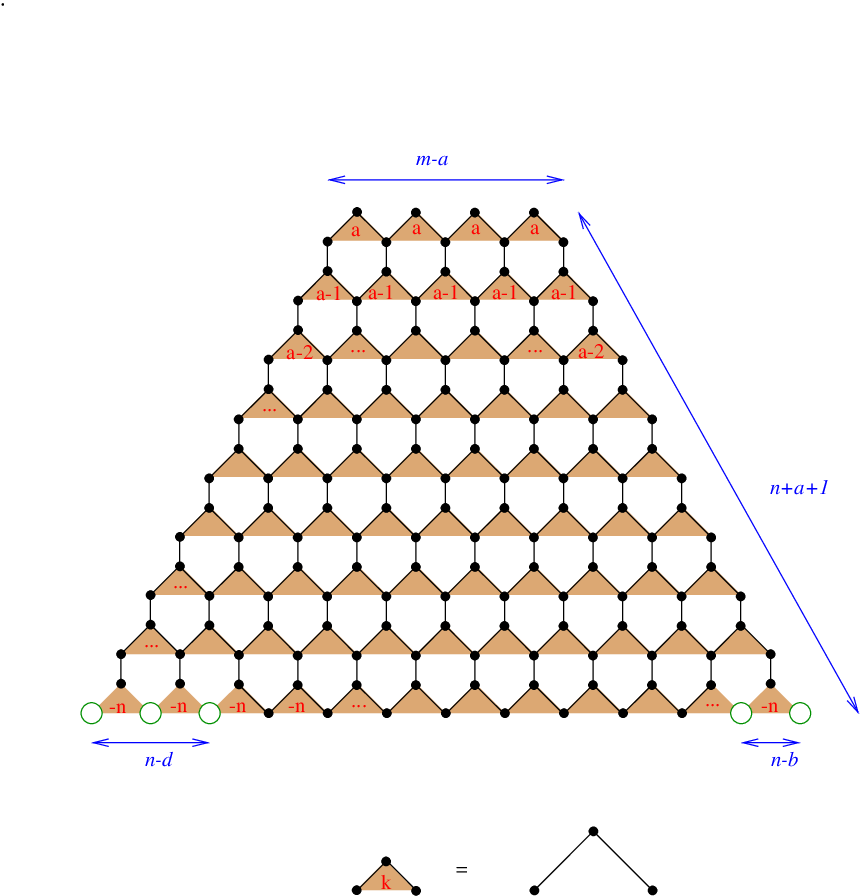}%
\end{picture}%
%
%  Created by WinFIG version 2024.2 
%  METADATA <version>1.0</version> 
%
\setlength{\unitlength}{3947sp}%
\begin{picture}(6882,7154)(40,-6354)
%  METADATA <id>4877</id> 
\put(4337,-6027){\makebox(0,0)[lb]{\smash{\fontsize{10}{12}\normalfont {\color[rgb]{0,0,1}$eq^k$}%
}}}
%  METADATA <id>4878</id> 
\put(5117,-6107){\makebox(0,0)[lb]{\smash{\fontsize{10}{12}\normalfont {\color[rgb]{0,0,1}$f$}%
}}}
\end{picture}}
		\caption{Weight assignment for the top part of $F$.}
		\label{threeapp5a}
\end{figure}

\begin{figure}[ht]\centering
\resizebox{13cm}{!}{\begin{picture}(0,0)%
\includegraphics{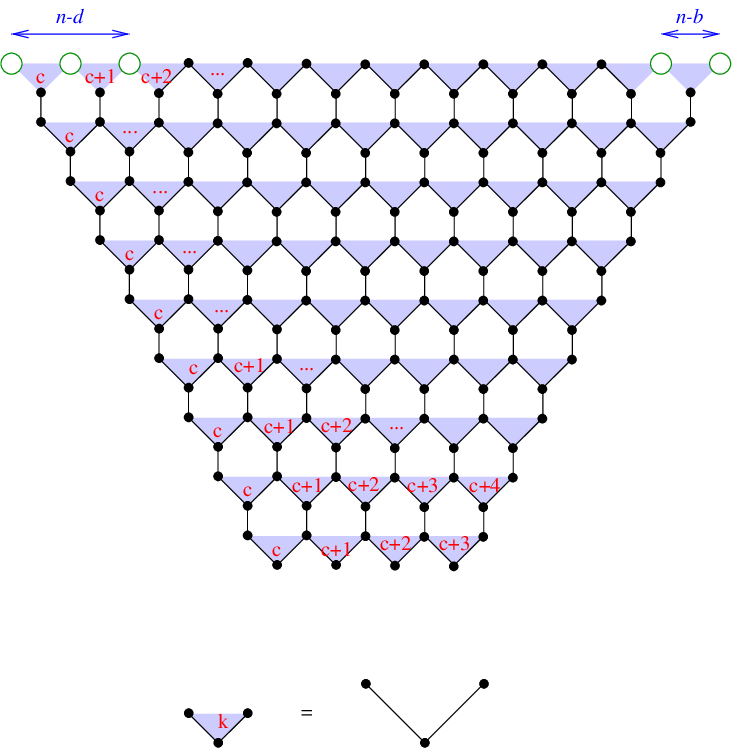}%
\end{picture}%
%
%  Created by WinFIG version 2024.2 
%  METADATA <version>1.0</version> 
%
\setlength{\unitlength}{3947sp}%
\begin{picture}(5853,5979)(2036,-6762)
%  METADATA <id>5570</id> 
\put(4928,-6674){\makebox(0,0)[lb]{\smash{\fontsize{10}{12}\normalfont {\color[rgb]{0,0,1}$hq^k$}%
}}}
%  METADATA <id>5572</id> 
\put(5755,-6674){\makebox(0,0)[lb]{\smash{\fontsize{10}{12}\normalfont {\color[rgb]{0,0,1}$gq^k$}%
}}}
\end{picture}%
}
		\caption{Weight assignment for the bottom part of $F$ (rotated by $180^{\circ}$).}
		\label{threeapp5b}
\end{figure}

Graph $F$ has $m+b+d-n+1$ vertices in the `middle' (see the ones on the dotted line in Figure  \ref{threeapp4}). In each perfect matching $\mu$ of the graph $F$,  there are exactly $m-a$ of these vertices matched upward, and $m-c$ remaining vertices are matched downward. (Note that we have $m+b+d-n+1=(m-a)+(m-c)$ by the balancing condition $m+ n=a+b+c+d+1$.) Thus, the matching $\mu$ of $F$ is partitioned into two disjoint perfect matchings: a perfect matching $\mu_1$ of $SH_1$ with bottom vertices in positions $A'=[n-d] \cup (n-d+A) \cup (m+b+1+ [n-b])$ removed ($n-d+A$ is the position set corresponding to the matched-downward middle vertices) and a perfect matching $\mu_2$ of (the upside-down version of) $SH_2$, with bottom vertices in positions $B'=[n-d] \cup (n-d+B) \cup (m+b+1+ [n-b])$ removed ($n-d+B$ is the position set corresponding to the matched-upward middle vertices). Thus,  \[\wt(\mu)=\wt(\mu_1)\cdot \wt(\mu_2).\]
 It is easy to see that $SH_1$ is exactly the dual graph of the weighted semi-hexagon $\mathcal{SH}_{n+a+1,m-a} (\wt_{eq^{-n-1},f}(q))$ (with dents at the positions in $A'$) in Lemma \ref{semilem1}. Similarly,  $SH_2$ is exactly (an upside-down version of) the dual graph of the weighted semi-hexagon  $\mathcal{SH}_{n+c+1,m-c} (\overline{\wt}_{hq^{c-1},gq^{c-1}}(q))$ (with dents at the positions in $B'$) in Lemma \ref{semilem2}. Applying Lemma  \ref{semilem1}  with the substitutions $a \mapsto n+a+1,b \mapsto m-a, x \mapsto eq^{-n-1}, y \mapsto f$ and  Lemma \ref{semilem2}  with the substitutions $a \mapsto n+c+1,b \mapsto m-c, x \mapsto hq^{c-1}, y \mapsto gq^{c-1}$, we obtain
 \begin{align}
 \wt(\mu)=&(eq^{-n-1})^{(n+a+1)(m-a)} \left(\frac{f}{eq^{-n-1}}\right)^{\square(n-d +B)}s_{\lambda(A')}(q,q^2,\dots,q^{n+a+1}) \notag\\
 &\times (hq^{c-1})^{(n+c+1)(m-c)}\left(\frac{g}{h}\right)^{\square(n-d +A)} q^{(n+c+1)(m-c)(m-n-2c-1)/2+(n+c+1)\square(n-d+A)}\notag\\
 &\times s_{\lambda(B')}(q,q^2,\dots,q^{n+c+1}).
 \end{align}
 %{\color{purple} replace $q^{(n+c+1)(m-c)^2+(n+c+1)\square(A)}$ with $q^{(n+c+1)(m-c)(m-n-2c-1)/2+(n+c+1)\square(n-d+A)}$ ???}
Here, for each set $U=\{u_1<u_2<\cdots<u_k\}$, we define the operation \[\square(U):=\sum_{i=1}^{k} (u_i-i).\]

Now, by Lemma \ref{Schurlem}, we have %({\color{red}simplify})
 \begin{align}
 \wt(\mu)%&=%(eq^{-n-1})^{(n+a+1)(m-a)} \left(\frac{f}{eq^{-n-1}}\right)^{\square(n-d +B)} q^{(n+c+1)(m-c)^2+(n+c+1)\square(n-d +A)}\notag\\
 %&\times  (hq^{c-1})^{(n+c+1)(m-c)}\left(\frac{g}{h}\right)^{\square(n-d +A)}\notag\\
 %& \times K \cdot s_{\lambda(A)}(q,q^2,\dots,q^{m-c})s_{\lambda(B)}(q,q^2,\dots,q^{m-a})\\
 %&=(e ^{(n+a+1)(m-a)}) \cdot (q^{-(n+1)(n+a+1)(m-a)})  (h^{(n+c+1)(m-c)})(q^{(c-1)(n+c+1)(m-c)}) q^{(n+c+1)(m-c)^2}\notag\\
 %&\times \left(\frac{f}{e}\right)^{(n-d)(m-a)} \left(\frac{f}{e}\right)^{\square(B)} \left(\frac{g}{h}\right)^{(n-d)(m-c)} \left(\frac{g}{h}\right)^{\square(A)} \cdot q^{(n+1)(n-d)(m-a)}q^{(n+1)\square(B)+ (n+c+1)\square(A)}\notag \\
 %& \times q^{R(n-d,n-b,m-c,m-a)} \cdot \frac{H(m-c)H(m-a)H(n-b)^2H(2n-b-d)H(2n-b-d+2m-a-c)^2 }{H(2n-b-d+m-c)H(2n-b-d+m-a)H(n-b+2m-a-c)H(n-d+2m-a-c)^2}\\
 %&\times s_{\lambda(A)}(q,q^2,\dots,q^{n+a+1})s_{\lambda(B)}(q,q^2,\dots,q^{n+c+1})\\
 %%&=e ^{(n+a+1)(m-a)} h^{(n+c+1)(m-c)}\left(\frac{f}{e}\right)^{(n-d)(m-a)+\square(B)}\left(\frac{g}{h}\right)^{(n-d)(m-c)+\square(A)}\notag\\
 %%&\times q^{\frac12(n+c+1)(m-c)(m-n-3)-(n+1)(n+a+1)(m-a)}\notag\\
 %%&\times q^{(n+1)((n-d)(m-a)+\square(B))}q^{(n+c+1)((n-d)(m-c)+\square(A))}\notag\\
 %%& \times q^{R(n-d,n-b,m-c,m-a)} \cdot \prod_{i=1}^{2m-a-c} (q^i;q)_{n-d}(q^{2m-a-c-i+1};q)_{n-b}\cdot  \prod_{j=1}^{n-b} (q^{2m-a-c+j};q)_{n-d}^2\notag\\
 %%&\times \frac{\Hf_q(m-c)\Hf_q(m-a)\Hf_q(n-d)^2\Hf_q(n-b)^2}{\Hf_q(m+2n-b-c-d)\Hf_q(m+2n-b-d-a)}\times s_{\lambda(A)}(q,q^2,\dots,q^{m-c})s_{\lambda(B)}(q,q^2,\dots,q^{m-a})\\
 &=e ^{(d+a+1)(m-a)} h^{(d+c+1)(m-c)}g^{(n-d)(m-c)}f^{(n-d)(m-a)}\left(\frac{f}{e}\right)^{\square(B)}\left(\frac{g}{h}\right)^{\square(A)}\notag\\
 &\times q^{\frac12(n+c+1)(m-c)(m-n-3)-(n+1)(n+a+1)(m-a)}q^{(n+1)(n-d)(m-a)}q^{(n+c+1)(n-d)(m-c)}\notag\\
 &\times q^{(n+1)\square(B)}q^{(n+c+1)\square(A)}\notag\\
 & \times q^{(m-a)(m-c)(n-b)+(2m-a-c)\binom{n-b+1}{2}} \cdot \prod_{i=1}^{2m-a-c} (q^i;q)_{n-d}(q^{2m-a-c-i+1};q)_{n-b}\cdot  \prod_{j=1}^{n-b} (q^{2m-a-c+j};q)_{n-d}^2\notag\\
 &\times \frac{\Hf_q(m-c)\Hf_q(m-a)\Hf_q(n-d)^2\Hf_q(n-b)^2}{\Hf_q(n+a+1)\Hf_q(n+c+1)}\times s_{\lambda(A)}(q,q^2,\dots,q^{m-c})s_{\lambda(B)}(q,q^2,\dots,q^{m-a}),
 \end{align}
%where the exponent $R(n-d,n-b,m-c,m-a)$ is defined as in Lemma \ref{Schurlem}.
We also note that $m+n=a+b+c+d+1$ by the balancing condition.
 
We note that $\square(n+S)=n|S|+\square(S)$. %no, it is n|S|, not $n\binom\binom{|S|}{2}}$
We also note that since $A\sqcup B=[2m-a-c]$, we have 
\[\square(A)+\square(B)=\binom{2m-a-c+1}{2}-\binom{m-a+1}{2}-\binom{m-c+1}{2}=(m-a)(m-c).\] 
Here we write $A\sqcup B=S$ when $A$ and $B$ are two disjoint sets whose union is $S$.
From the condition
\[\frac{g}{h}= q^{-c} \cdot \frac{f}{e},\] 
we get %({\color{red}simplify})
\begin{align}
 \wt(\mu)
 %=&e ^{(d+a+1)(m-a)} \cdot f^{(n-d)(m-a)} \cdot g^{(n-d)(m-c)} \cdot  h^{(d+c+1)(m-c)}\cdot q^{-(n+1)(d+a+1)(m-a)+(m-1)(n+c+1)(m-c)} \notag\\
%&\times  \left(\frac{f}{e}\right)^{(m-a)(m-c)}  \cdot q^{(n+1)(m-a)(m-c))}\notag \\
 %& \times q^{R(n-d,n-b,m-c,m-a)} \cdot \frac{H(m-c)H(m-a)H(n-b)^2H(2n-b-d)H(m+n+1)^2 }{H(n+a+1)H(n+c+1)H(m+d+1)H(m+b+1)^2}\\
 %&\times s_{\lambda(A)}(q,q^2,\dots,q^{n+a+1})s_{\lambda(B)}(q,q^2,\dots,q^{n+c+1})\\
 %=&e ^{(n-b)(m-a)} \cdot f^{(m+n-c-d)(m-a)} \cdot g^{(n-d)(m-c)} \cdot  h^{(d+c+1)(m-c)}\cdot {\color{purple}q^{(m-1)(n+c+1)(m-c)-(n+1)(m-a)(n-b)}} \notag\\
 %%&=e^{(n-b)(m-a)}f^{(m+n-c-d)(m-a)}g^{(n-d)(m-c)}h^{(d+c+1)(m-c)}\notag\\
 %%&\times  q^{\frac12(n+c+1)(m-c)(m-n-3)-(n+1)(n+a+1)(m-a)}q^{(n+1)(n-d)(m-a)}q^{(n+c+1)(n-d)(m-c)}\notag\\
 %%&\times {q^{(n+1)(m-a)(m-c)}}\notag\\
 %%& \times q^{R(n-d,n-b,m-c,m-a)} \cdot \prod_{i=1}^{2m-a-c} (q^i;q)_{n-d}(q^{2m-a-c-i+1};q)_{n-b}\cdot  \prod_{j=1}^{n-b} (q^{2m-a-c+j};q)_{n-d}^2\notag\\
 %%&\times \frac{\Hf_q(m-c)\Hf_q(m-a)\Hf_q(n-d)^2\Hf_q(n-b)^2}{\Hf_q(n+a+1)\Hf_q(n+c+1)}s_{\lambda(A)}(q,q^2,\dots,q^{m-c})s_{\lambda(B)}(q,q^2,\dots,q^{m-a})\\
 &=e^{(n-b)(m-a)}f^{(m+n-c-d)(m-a)}g^{(n-d)(m-c)}h^{(d+c+1)(m-c)}\notag\\
 &\times q^{\frac12(n+c+1)(m-c)(n+m-2d-3)-(n+1)(a+d+1)(m-a)+(n+1)(m-a)(m-c)}\notag\\
 & \times q^{(m-a)(m-c)(n-b)+(2m-a-c)\binom{n-b+1}{2}} \notag\\
 &\times \prod_{i=1}^{2m-a-c} (q^i;q)_{n-d}(q^{2m-a-c-i+1};q)_{n-b}\cdot  \prod_{j=1}^{n-b} (q^{2m-a-c+j};q)_{n-d}^2\notag\\
 &\times \frac{\Hf_q(m-c)\Hf_q(m-a)\Hf_q(n-d)^2\Hf_q(n-b)^2}{\Hf_q(n+a+1)\Hf_q(n+c+1)}\notag\\
 &\times s_{\lambda(A)}(q,q^2,\dots,q^{m-c})s_{\lambda(B)}(q,q^2,\dots,q^{m-a}).
 \end{align}
 
Therefore, summing over all perfect matchings of $F$, we obtain
 \begin{align}\label{maineqf}
\M(F)&=e^{(n-b)(m-a)}f^{(m+n-c-d)(m-a)}g^{(n-d)(m-c)}h^{(d+c+1)(m-c)}\notag\\
 &\times  q^{\frac12(n+c+1)(m-c)(n+m-2d-3)-(n+1)(a+d+1)(m-a)+(n+1)(m-a)(m-c)}\notag\\
%&e ^{(n-b)(m-a)} \cdot f^{(m+n-c-d)(m-a)} \cdot g^{(n-d)(m-c)} \cdot  h^{(d+c+1)(m-c)}\cdot {\color{purple}q^{(m-1)(n+c+1)(m-c)-(n+1)(m-a)(n-b)}} \notag\\
 & \times q^{(m-a)(m-c)(n-b)+(2m-a-c)\binom{n-b+1}{2}} \notag\\
 &\times \prod_{i=1}^{2m-a-c} (q^i;q)_{n-d}(q^{2m-a-c-i+1};q)_{n-b}\cdot  \prod_{j=1}^{n-b} (q^{2m-a-c+j};q)_{n-d}^2 \notag\\
&\times \frac{\Hf_q(m-c)\Hf_q(m-a)\Hf_q(n-d)^2\Hf_q(n-b)^2}{\Hf_q(n+a+1)\Hf_q(n+c+1)}\\
&\times\sum_{(A,B)}s_{\lambda(A)}(q,q^2,\dots,q^{m-c})s_{\lambda(B)}(q,q^2,\dots,q^{m-a}),
 \end{align}
 where the sum ranges over all pairs of disjoint sets $(A, B)$ whose union is $[2m-a-c]$, with $|A|=m-c$ and $|B|=m-a$.

\medskip
\textbf{ Step 3.}  \emph{Apply Lemma \ref{Schurlem2}.}
\medskip

Without loss of generality, assume that $c\geq a$. By Lemma \ref{Schurlem2} with the substitutions $m\mapsto m-c$ and $d\mapsto c-a$, we may write the sum on the right-hand side of (\ref{maineqf}) as follows.

If $c-a$ is even, say $c-a=2s$, then
\begin{align}\label{maineqd}
&\sum_{(A,B)}s_{\lambda(A)}(q,q^2,\dots,q^{m-c})s_{\lambda(B)}(q,q^2,\dots,q^{m-a})\notag\\
&=2^{m-c} (1-q)^{s} q^{\frac{(m-c)(m-c+1)(2m-2c+6s+1)}{6}-\frac{s(s-1)(s-2)}{3}}\frac{\Hf_{q^2}(m-c+s)^2}{ \Hf_q(m-c)\Hf_q(m-c+2s)}\notag\\
& \quad \times \prod_{i=1}^{s} \frac{(q^{3};q^2)_{i-1}(q;q^2)_{i-1}(q^{2(i+s)};q^2)_{m-c}(q^2;q^2)_{i-1}}{(q^2;q^2)_{m-c+s-i}}.
\end{align}
where the sum is over all disjoint pairs of sets $A$ and $B$ whose union is $[2m-a-c]$ and whose cardinalities are given by $|A|=m-c$ and $|B|=m-a.$	

If $c-a$ is odd, say $c-a=2s+1$, then
\begin{align}\label{maineqe}
&\sum_{(A,B)}s_{\lambda(A)}(q,q^2,\dots,q^{m-c})s_{\lambda(B)}(q,q^2,\dots,q^{m-a})\notag\\
&=2^{m-c}(1-q)^{2s}q^{\frac{(m-c)(m-c+1)(m-c+3s+2)}{3}-\frac{s(s-1)(s-2)}{3}}\frac{\Hf_{q^2}(m-c+s)H_{q^2}(m-c+s+1)}{ \Hf_q(m-c)\Hf_q(m-c+2s+1)}\notag\\
& \times \prod_{i=1}^{s} \frac{(q^{3};q^2)_{i-1}(q^3;q^2)_{i-1}(q^{2(i+s+1)};q^2)_{m-c}(q^2;q^2)_{i-1}}{(q^2;q^2)_{m-c+s-i}}.
\end{align}

Combining the identities (\ref{maineqc}), (\ref{maineqf}), (\ref{maineqd}), and (\ref{maineqe}), we have, when $|a-c|=2s$, %({\color{red}simplify})
\begin{align}
 \M(C)%%&=\prod_{i=1}^{c} (egq^{i-1}+hf)^{c-i+1}   \prod_{i=1}^{n} (eg+hfq^{i})^{n-i+1}  \prod_{i=1}^{a} (e g  q^{i-1}+hf)^{a-i+1}\notag\\
%%&\times q^{\binom{c}{2}+\binom{c+1}{3}+\binom{a}{2}+\binom{a+1}{3}+(m+b+d)\binom{n+1}{2}+(n+c+1)\binom{a+1}{2}+(c-d)\binom{n+1}{2}} \cdot \M(F)\\
%%&=\prod_{i=1}^{c} (egq^{i-1}+hf)^{c-i+1}   \prod_{i=1}^{n} (eg+hfq^{i})^{n-i+1}  \prod_{i=1}^{a} (e g  q^{i-1}+hf)^{a-i+1}\notag\\
%%&\times q^{\binom{c}{2}+\binom{c+1}{3}+\binom{a}{2}+\binom{a+1}{3}+(m+b+d)\binom{n+1}{2}+(n+c+1)\binom{a+1}{2}+(c-d)\binom{n+1}{2}} \notag\\
%%&\times e^{(n-b)(m-a)}f^{(m+n-c-d)(m-a)}g^{(n-d)(m-c)}h^{(d+c+1)(m-c)}\notag\\
 %%&\times  q^{\frac12(n+c+1)(m-c)(n+m-2d-3)-(n+1)(a+d+1)(m-a)+(n+1)(m-a)(m-c)}\notag\\
 %%& \times q^{(m-a)(m-c)(n-b)+(2m-a-c)\binom{n-b+1}{2}}\frac{\Hf_q(m-c)\Hf_q(m-a)\Hf_q(n-d)^2\Hf_q(n-b)^2}{\Hf_q(n+a+1)\Hf_q(n+c+1)}\notag\\
 %%&\times \prod_{i=1}^{2m-a-c} (q^i;q)_{n-d}(q^{2m-a-c-i+1};q)_{n-b}\cdot  \prod_{j=1}^{n-b} (q^{2m-a-c+j};q)_{n-d}^2 \notag\\
%%&\times \sum_{(A,B)}s_{\lambda(A)}(q,q^2,\dots,q^{m-c})s_{\lambda(B)}(q,q^2,\dots,q^{m-a})\\
&=\prod_{i=1}^{c} (egq^{i-1}+hf)^{c-i+1}   \prod_{i=1}^{n} (eg+hfq^{i})^{n-i+1}  \prod_{i=1}^{a} (e g  q^{i-1}+hf)^{a-i+1}\notag\\
&\times q^{\binom{c}{2}+\binom{c+1}{3}+\binom{a}{2}+\binom{a+1}{3}+(m+b+d)\binom{n+1}{2}+(n+c+1)\binom{a+1}{2}+(c-d)\binom{n+1}{2}} \notag\\
&\times e^{(n-b)(m-a)}f^{(m+n-c-d)(m-a)}g^{(n-d)(m-c)}h^{(d+c+1)(m-c)}\notag\\
 &\times  q^{\frac12(n+c+1)(m-c)(n+m-2d-3)-(n+1)(a+d+1)(m-a)+(n+1)(m-a)(m-c)}\notag\\
 & \times q^{(m-a)(m-c)(n-b)+(2m-a-c)\binom{n-b+1}{2}}\frac{\Hf_q(m-c)\Hf_q(m-a)\Hf_q(n-d)^2\Hf_q(n-b)^2}{\Hf_q(n+a+1)\Hf_q(n+c+1)}\notag\\
 &\times \prod_{i=1}^{2m-a-c} (q^i;q)_{n-d}(q^{2m-a-c-i+1};q)_{n-b}\cdot  \prod_{j=1}^{n-b} (q^{2m-a-c+j};q)_{n-d}^2\notag\\
 &\times 2^{m-c} (1-q)^{s} q^{\frac{(m-c)(m-c+1)(2m-2c+6s+1)}{6}-\frac{s(s-1)(s-2)}{3}}\notag\\
&\times\frac{\Hf_{q^2}(m-c+s)^2}{ \Hf_q(m-c)\Hf_q(m-c+2s)}\notag\\
&\times \prod_{i=1}^{s} \frac{(q^{3};q^2)_{i-1}(q;q^2)_{i-1}(q^{2(i+s)};q^2)_{m-c}(q^2;q^2)_{i-1}}{(q^2;q^2)_{m-c+s-i}}.
\end{align}
When $|c-a|=2s+1$, we have
\begin{align}
 \M(C)&=\prod_{i=1}^{c} (egq^{i-1}+hf)^{c-i+1}   \prod_{i=1}^{n} (eg+hfq^{i})^{n-i+1}  \prod_{i=1}^{a} (e g  q^{i-1}+hf)^{a-i+1}\notag\\
&\times q^{\binom{c}{2}+\binom{c+1}{3}+\binom{a}{2}+\binom{a+1}{3}+(m+b+d)\binom{n+1}{2}+(n+c+1)\binom{a+1}{2}+(c-d)\binom{n+1}{2}} \notag\\
&\times e^{(n-b)(m-a)}f^{(m+n-c-d)(m-a)}g^{(n-d)(m-c)}h^{(d+c+1)(m-c)}\notag\\
 &\times  q^{\frac12(n+c+1)(m-c)(n+m-2d-3)-(n+1)(a+d+1)(m-a)+(n+1)(m-a)(m-c)}\notag\\
 & \times q^{(m-a)(m-c)(n-b)+(2m-a-c)\binom{n-b+1}{2}}\frac{\Hf_q(m-c)\Hf_q(m-a)\Hf_q(n-d)^2\Hf_q(n-b)^2}{\Hf_q(n+a+1)\Hf_q(n+c+1)}\notag\\
 & \times \prod_{i=1}^{2m-a-c} (q^i;q)_{n-d}(q^{2m-a-c-i+1};q)_{n-b}\cdot  \prod_{j=1}^{n-b} (q^{2m-a-c+j};q)_{n-d}^2\notag\\
 &\times 2^{m-c}(1-q)^{2s}q^{\frac{(m-c)(m-c+1)({m-c+3s+2})}{3}-\frac{s(s-1)(s-2)}{3}}\notag\\
&\times\frac{\Hf_{q^2}(m-c+s)H_{q^2}(m-c+s+1)}{ \Hf_q(m-c)\Hf_q(m-c+2s+1)}\notag\\
& \times \prod_{i=1}^{s} \frac{(q^{3};q^2)_{i-1}(q^3;q^2)_{i-1}(q^{2(i+s+1)};q^2)_{m-c}(q^2;q^2)_{i-1}}{(q^2;q^2)_{m-c+s-i}}.
\end{align}
These are exactly the formulas claimed in Theorem \ref{main}.
\end{proof}

%%%%%%%%
\section{Proofs of Several Results in Section \ref{Sec:Prelim}}\label{Sec:Complete}

\subsection{Proofs of the Sandwich Lemmas (Lemmas \ref{lm2} and \ref{lem3})}

\begin{figure}\centering
		\includegraphics[width=16 cm]{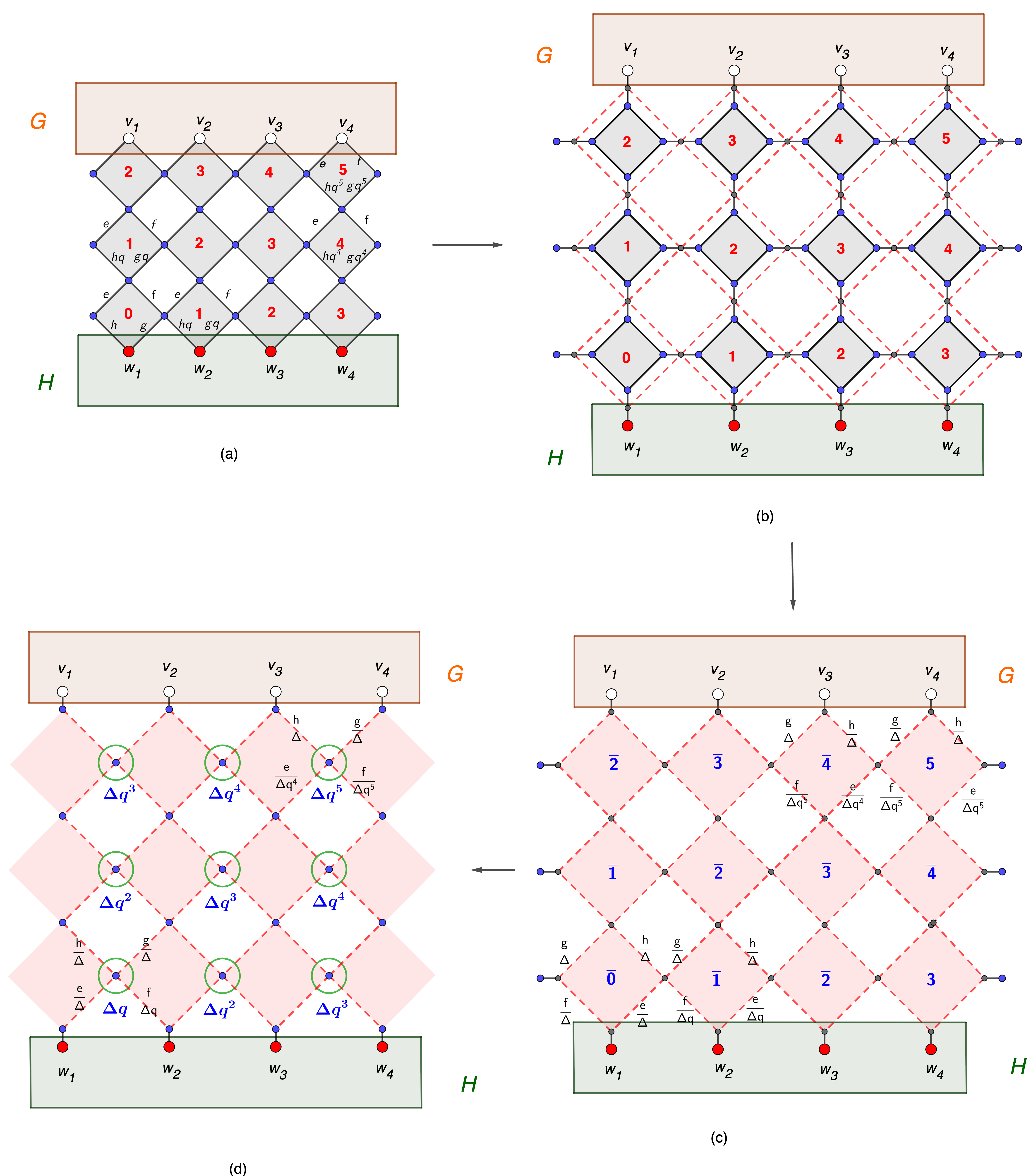}
		\caption{Illustration of the proof of Lemma \ref{lm2}.}
		\label{prooflm2}
\end{figure}

\begin{proof}[Proof of Sandwich Lemma 1 (Lemma \ref{lm2})]
The proof is illustrated in Figure \ref{prooflm2} for the case $m=3$ and $n=4$.  We start with graph $A=G\# AR_{m,n}(\wt_{h,g}^{e,f}(q))\#H$ as in picture (a).
Our Aztec rectangle graph $AR_{m,n}(\wt_{h,g}^{e,f}(q))$ consists of $m=3$ rows, each contains $n=4$ diamonds. We will perform four transformations as follows. In picture (a), the diamonds of label $i$ (bold red number in the middle) have edge weights $e,f,gq^{i},hq^{i}$ in clockwise order, starting from the northwest side. We apply the Vertex-Splitting Lemma (Lemma \ref{VSlem}) at every vertex of the Aztec rectangle graph. We split horizontally at the left and right vertices and vertically at the top and bottom vertices of each diamond, obtaining the graph $B$ in picture (b). This allows us to apply the Spider Lemma  (Lemma \ref{spiderlem}) around each diamond face; each diamond (together with four adjacent edges) will be replaced by the dotted diamond. We get graph $C$ in picture (c); each dotted diamond labeled $\overline{i}$ has edge weights $g/\Delta,h/\Delta,e/\Delta q^i, f/\Delta q^i$ in clockwise order, starting from the northwest side, where $\Delta:=eg+fh$. Next, we note that each horizontal edge in graph $C$ is a forced edge, so we can remove it (together with its vertices) without changing the matching generating function of the graph. Denote by $D$ the resulting graph in picture (d); the leftmost and rightmost diamonds are partial diamonds with their left and edges missing, respectively. Finally, in graph $D$, we apply the Star Lemma (Lemma \ref{starlem}) to each circled vertex with the factor indicated. We obtain exactly the graph $E=G\# _|^|AR_{m,n-1}(\wt_{h,g}^{eq,f}(q))\#H$. 

Next, let us keep track of the multiplicative factor in each transformation above. In the first transformation, we applied the Vertex-Splitting Lemma, and it does not change the matching generating function. We have
\[\M(A)=\M(B).\]
Next, in the second transformation, the application of the Spider Lemma to each diamond of label $k$ gives a factor $(e\cdot gq^k+f\cdot hq^k)=\Delta q^k$. We note that the diamond in the $i$-th row (counted from the bottom) and $j$-th column (counted from the left) is labeled by $i+j-2$. Gathering all such factors, we have
\[\M(B)=\Delta^{mn}q^{\sum_{i=1}^{m}\sum_{j=1}^{n} (i+j-2)} \cdot \M(C).\]
In the third transformation, we removed forced edges of weight $1$ from $C$, and it did not change the matching generating function. We have
\[\M(C)=\M(D).\]
Finally, in the fourth transformation, we applied the Star Lemma at the $m(n-1)$ circled vertices. The vertex in the $i$-th row and $j-$th column contributes a factor $\Delta q^{i+j-1}$. Thus, we have
\[\M(D)=\Delta ^{-m(n-1)} q^{-\sum_{i=1}^{m}\sum_{j=1}^{n-1} (i+j-1)}\M(E).\]
Combining all four transformations, we obtain
\[\M(A)=\Delta^m q^{\binom{m}{2}}\M(E)\] as desired.
\end{proof}

Sandwich Lemma 2 (Lemma \ref{lem3}) can be proved similarly to Sandwich Lemma 1 above, as illustrated in Figure \ref{prooflm3}.

\begin{figure}\centering
		\includegraphics[width=16 cm]{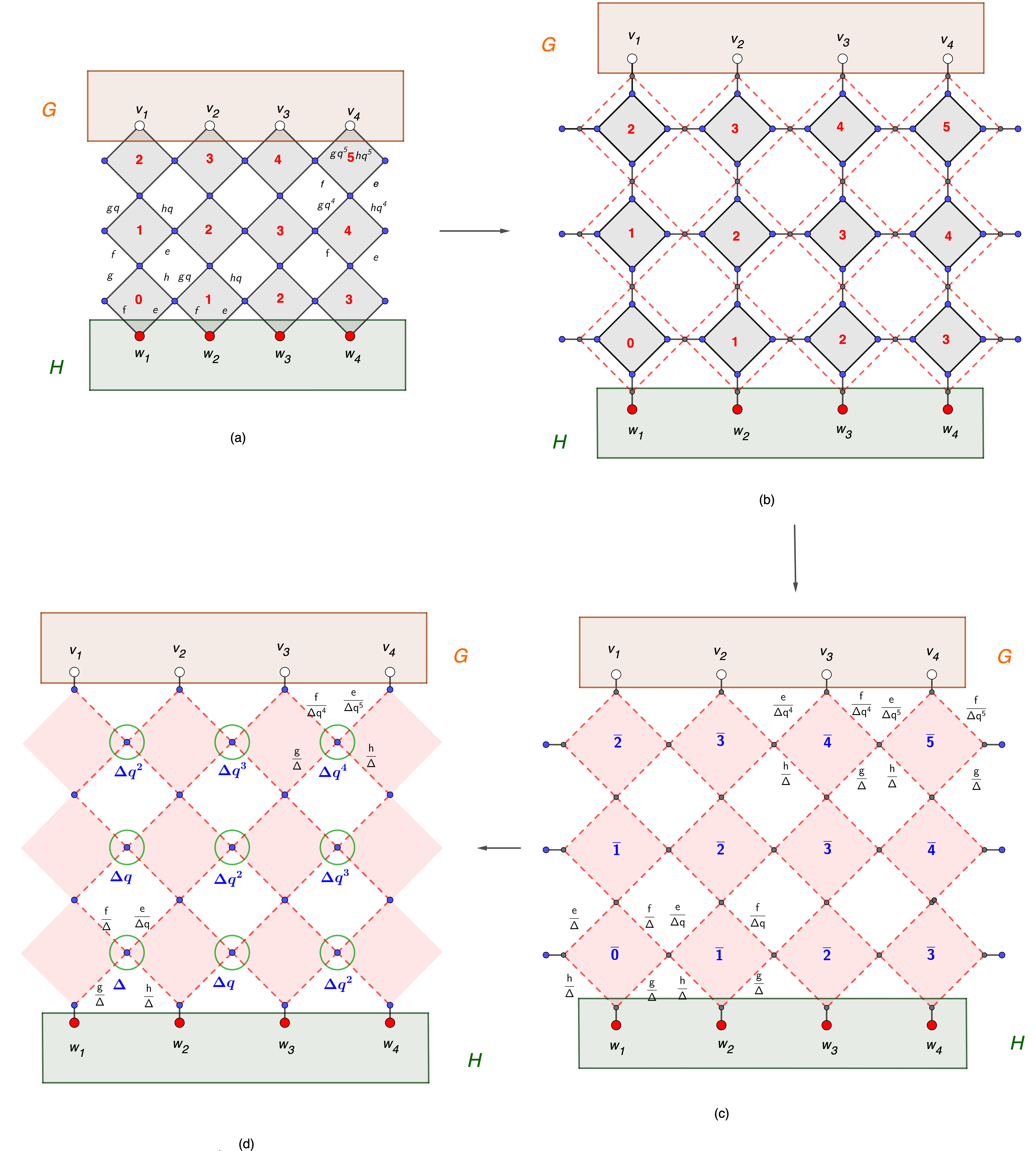}
		\caption{Illustration of the proof of Lemma \ref{lem3}.}
		\label{prooflm3} 		
\end{figure}

\begin{proof}[Proof of Sandwich Lemma 2 (Lemma \ref{lem3})]
The proof is illustrated in Figure \ref{prooflm3} for the case $m=3$ and $n=4$.  We start with graph $A=G\# AR_{m,n}(\overline{wt}_{h,g}^{e,f}(q))\#H$ in picture (a).
We will perform $4$-step transformations as follows. In picture (a), the diamond of label $i$ (bold red number in the middle) has edge weights $gq^{i},hq^{i},e,f$ in clockwise order starting from the northwest side. We apply the Vertex-Splitting Lemma to obtain the graph $B$ in picture (b). We have the relation
\[\M(A)=\M(B).\]
Next, we apply the Spider Lemma around each diamond face of $B$ and get graph $C$ in picture (c); each dotted diamond of label $\overline{i}$ has edge weights $e/\Delta q^i, f/\Delta q^i,g/\Delta,h/\Delta,$ in clockwise order starting from the northwest side, where $\Delta:=eg+fh$. We therefore have
\[\M(B)=\Delta^{mn}q^{\sum_{i=1}^{m}\sum_{j=1}^{n} (i+j-2)} \cdot \M(C).\]
Next, we note that each horizontal edge in graph $C$ is a forced edge of weight $1$. We can remove it (without changing the matching generating function) to get the graph $D$ as in picture (d) and
\[\M(C)=\M(D).\]
Finally, in graph $D$, we apply the Star Lemma to each circled vertex with the indicated factor. We obtain exactly the graph $E=G\# _|^|AR_{m,n-1}(\overline{wt}_{h,g}^{e/q,f}(q))\#H$. We note that the factors in the Star Lemma here are slightly different from those in the proof of Lemma \ref{lm2} above. We have now
\[\M(D)=\Delta ^{-m(n-1)} q^{-\sum_{i=1}^{m}\sum_{j=1}^{n-1} (i+j-2)}\M(E).\]
Combining all four transformations, we get
\[\M(A)=\Delta^m q^{m(n-1)+\binom{m}{2}}\M(E)\] as desired.
\end{proof}

\subsection{Proof of Lemma \ref{semilem2}}

\begin{proof}[Proof of Lemma \ref{semilem2}]
Consider the semi-hexagon $\mathcal{SH}=\mathcal{SH}_{a,b}(s_1,s_2,\dots,s_a)$. Similarly to the proof of Lemma \ref{semilem1}, we extract factors $x$ and $y$ from the weights of the right-tilting and left-tilting lozenges in $\mathcal{SH}$. This yields the simpler weight assignment $\overline{\wt}_{1,1}(q)$. Denote the resulting weighted semi-hexagon by $\mathcal{SH}'$. Then we have
\begin{equation}\label{semieq1}
\T(\mathcal{SH}_{a,b}(\overline{\wt}_{x,y}(q)))=x^{ab} \cdot (y/x)^{\sum_{i=1}^{b}(r_i-i)}\cdot\T(\mathcal{SH}').
\end{equation}
Similarly to the argument in the Star Lemma (Lemma \ref{starlem}), if we multiply the weights of all lozenges containing a fixed unit triangle by a factor $t>0$, then the tiling generating function is multiplied by the same factor $t$. (Since each tiling contains exactly one lozenge covering this triangle.) We multiply the weights of all lozenges covering the down-pointing triangles by the factors as shown in Figure \ref{newqweight}(a). In particular, the factor corresponding to a triangle at distance $k\cdot \frac{\sqrt{3}}{2}$ from the northwest boundary is $q^{-k}$. This way, all left- and right-tilting lozenges have weight $1$, and a vertical lozenge at distance $k\cdot \frac{\sqrt{3}}{2}$ to the northwest boundary now has weight $q^{-k}$ (see Figure \ref{newqweight}(b)).  Denote the resulting weighted semi-hexagon by $ \mathcal{SH}''$  and the new weight assignment by $\wt ''$. Then, we have
\begin{equation}\label{semieq2}
\T(\mathcal{SH}')=q^{\sum_{i=1}^a\binom{a+b-i+1}{2}}\cdot\T(\mathcal{SH}'').
\end{equation}

Next, for any fixed lozenge tiling $L$ of the region,  we will show that the new weight $\wt''(L)$ and the weight $\wt_{1,1}(q^{-1})(L)$ from Lemma \ref{semilem1} (and in Figure \ref{semihex}) differ only by a multiplicative factor independent from $L$. A vertical lozenge of weight $q^{-k}$ is labeled by $-k$ in blue, and the right-tilting lozenge with weight $q^{i}$ is labeled by $i$ in red, as shown in Figure \ref{newqweight}(c). (The left-tilting lozenges all have weight $1$ and are not labeled).
We also record these numbers in two plane partitions: the left one (in red) is the column-strict plane partition $\mu_L$ corresponding to the weight assignment $\wt_{1,1}(q)$ (as shown previously in Figure \ref{semihex}), the right one (in blue) is the staircase-shaped $(a-1,a-2,\dots,1,0)$ plane partition $\nu_L$, recording the negations of the numbers in the vertical lozenges. We now modify these numbers as follows.   We extract a factor $q$ from the weight of each right-tilting lozenge, and extract weight $q^{-k}$ from the $k$-th vertical lozenges (counted from the left) along path $Q_{a-i+1}$. See Figure \ref{newqweight}(d). This is equivalent to subtracting $1$ from each entry of the red plane partition $\mu_L$ in Figure \ref{newqweight}(c) and to subtracting $k$ from the $k$-th rightmost entry in each row of the blue partition $\nu_L$ in Figure \ref{newqweight}(c). Denote by $\overline{\mu}_L$ and $\overline{\nu}_L$ the resulting plane partitions, respectively. The entries in the $i$-th row of $\overline{\mu}_L$ form an integer partition $\pi_i$, and the entries in the $i$-th row of $\overline{\nu}_L$ form the conjugate partition of $\pi_i$. In particular, $|\overline{\mu}_L|=|\overline{\nu}_L|$. This implies that
\[|\nu_L|=|\overline{\nu}_L|+\sum_{i=1}^a\binom{i}{2}=|\overline{\mu}_L|+\sum_{i=1}^a\binom{i}{2}=  |\mu_L| -\sum_{i=1}^a(s_i-i)+\sum_{i=1}^a\binom{i}{2}. \]
We now have
 \begin{align*} \wt''(L)=q^{-|\nu_L|}=&q^{\sum_{i=1}^a(s_i-i)-\sum_{i=1}^a\binom{i}{2}} \cdot q^{-|\mu_L|}.
 \end{align*}
 Summing over all tilings of $\mathcal{SH}''$, we get
\begin{equation}\label{semieq3}
\T(\mathcal{SH}'')=q^{\sum_{i=1}^a(s_i-i)-\sum_{i=1}^a\binom{i}{2}}\cdot s_{\lambda(s_1,s_2,\dots,s_a)}(q^{-1},\dots,q^{-a}).
\end{equation}
Note that, since $\{r_1,\dots,r_b\}=[a+b]\setminus \{s_1,\dots,s_a\}$, we have %({\color{red}simplify})
\begin{align*}
\sum_{i=1}^a(s_i-i)+\sum_{i=1}^{b}(r_i-i)&=\sum_{j\in[a+b]}j-\binom{a+1}{2}-\binom{b+1}{2}\\
%%&=\binom{a+b+1}{2}-\binom{a+1}{2}-\binom{b+1}{2}
&=ab.
\end{align*}
Then, by (\ref{semieq1}), (\ref{semieq2}), and (\ref{semieq3}),  we have
\begin{align}\label{semieq4}
\T(\mathcal{SH}_{a,b}(\overline{\wt}_{x,y}(q)))=x^{ab} \cdot (y/x)^{\sum_{i=1}^{b}(r_i-i)}&\cdot q^{ab(a+b+2)/2-\sum_{i=1}^{b}(r_i-i)}\notag\\
&\times s_{\lambda(s_1,s_2,\dots,s_a)}(q^{-1},q^{-2},\dots,q^{-a}).
\end{align}

Next, from (\ref{Gp}) we obtain %({\color{red}simplify})
\begin{align}\label{semieq5}
s_{\lambda(s_1,s_2,\dots,s_a)}(q^{-1},\dots,q^{-a})%&=q^{-\sum_{i=1}^{a}(s_i-i)}\prod_{1\le i<j \le a} \frac{q^{-s_j}-q^{-s_i}}{q^{-j}-q^{-i}}\notag\\
%%&=q^{-\sum_{i=1}^{a}(s_i-i)}\prod_{1\le i<j \le a} q^{-s_i-s_j+i+j}  \frac{q^{s_j}-q^{s_i}}{q^{j}-q^{i}}\notag\\
%&=q^{-\sum_{i=1}^{a}(s_i-i)} \cdot q^{-\sum_{1\le i < j \le a} (s_i-i+s_j-j)}  \prod_{1\le i<j \le a}  \frac{q^{s_j}-q^{s_i}}{q^{j}-q^{i}}\notag\\
%&=q^{-(a+1)\sum_{i=1}^{a}(s_i-i)}\cdot q^{\sum_{i=1}^{a}(s_i-i)}\prod_{1\leq i<j\leq a}\frac{q^{s_j}-q^{s_i}}{q^{j}-q^{i}}\notag\\
%&=q^{-(a+1)\sum_{i=1}^{a}(s_i-i)} \cdot s_{\lambda(s_1,s_2,\dots,s_a)}(q,q^2,\dots,q^{a})\notag\\
%&=
=q^{-(a+1)\left(ab-\sum_{i=1}^{b}(r_i-i)\right)} \cdot s_{\lambda(s_1,s_2,\dots,s_a)}(q,q^2,\dots,q^{a}).
\end{align}
The lemma now follows from  (\ref{semieq4}) and (\ref{semieq5}).
\end{proof}

\begin{figure}[ht]\centering
\includegraphics[width=13 cm]{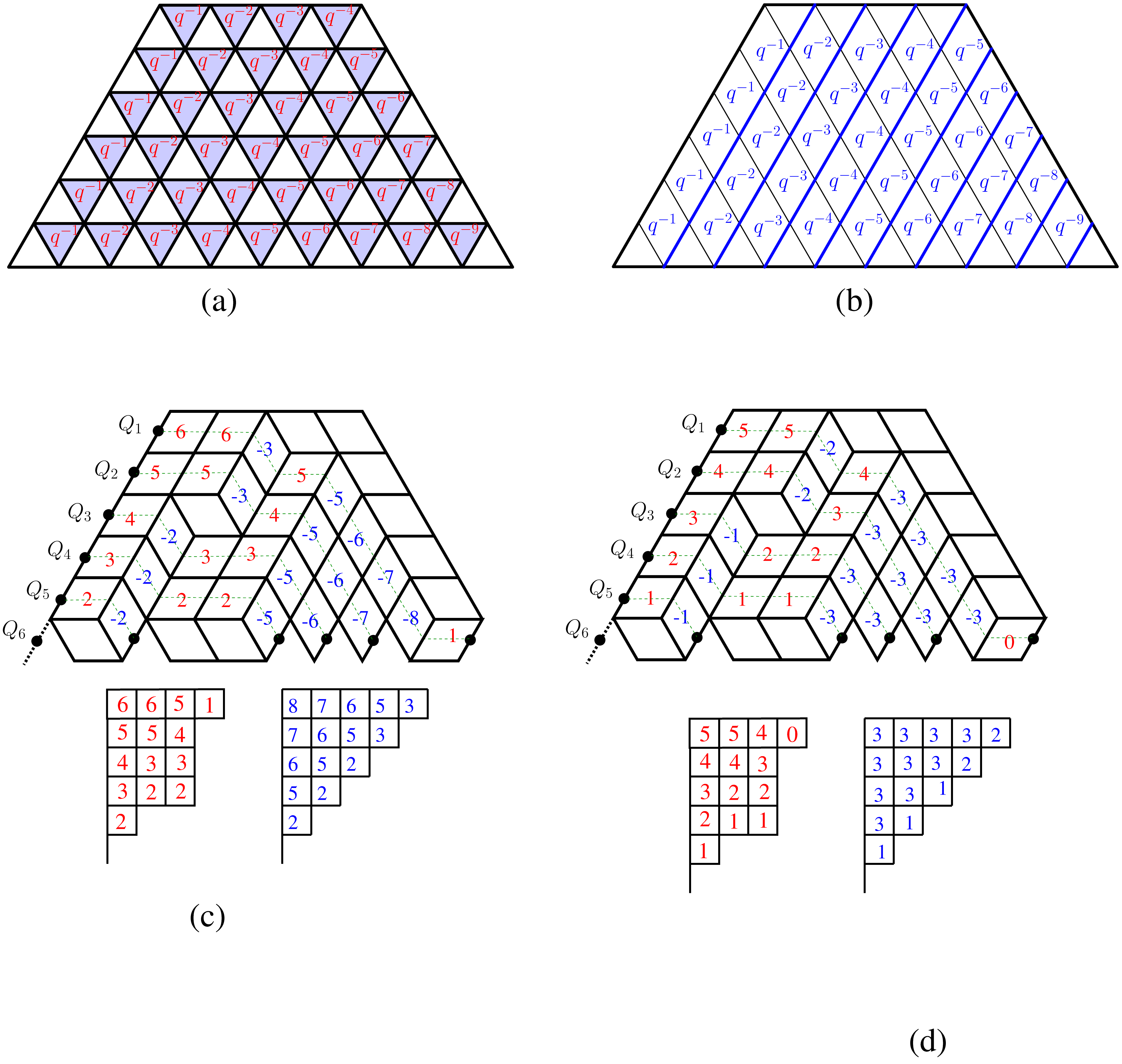}	
	\caption{(a) Modifying the lozenge weights of region $\mathcal{SH}'$. (b) Weights of vertical lozenges in $\mathcal{SH}''$. (c) Recording the two weight assignments for the same tiling. (d) Reducing each red label by $1$, and increasing the blue label of the $k$-th vertical lozenge in each path $Q_i$ by $k$.}
		\label{newqweight}
\end{figure}

\subsection{Proof of Lemma \ref{Schurlem2}}

\begin{proof}[Proof of Lemma \ref{Schurlem2}]
Apply Theorem \ref{thr5} with $T=[2m+d]$ and $t_i=i$. Then the sum involving Schur functions on the left-hand side can be written as follows:
 \begin{align}
&\sum_{\substack{A\sqcup B=[2m+d]\\|A|=m, |B|=m+d}}s_{\lambda(A)}(q,q^2,\dots,q^{m})s_{\lambda(B)}(q,q^2,\dots,q^{m+d})\notag\\
& \quad \quad \quad = 2^{m} \sum_{1\leq k_1<k_2<\cdots k_s\leq m+s} s_{\lambda(\{2,4,6,\dots \} \setminus \{2k_1,2k_2,\dots,2k_s\})}(q,q^2,q^3,\dots {q^m})\notag\\
&\quad\quad\quad \times s_{\lambda(\{1, 3,5,\dots \} \cup \{2k_1,2k_2,\dots,2k_s\})}(q,q^2,q^3,\dots {q^{m+d}}),
 \end{align}
where $s=\lfloor d/2 \rfloor$.

Let $U=\{u_1<u_2<\cdots<u_k\}$. Define the operations \[\blacktriangle(U):=\prod_{1\le i < j \le k}q^{u_j}-q^{u_i}.\]  Then we have \[s_{\lambda(A)}(q,q^2,\dots, q^{m})=q^{\square(A)}\frac{\blacktriangle(A)}{\blacktriangle([m])}.\]
and \[\blacktriangle([m])=\prod_{1\le i < j \le m}(q^{j}-q^{i})=(-1)^{\binom{m}{2}}q^{\binom{m+1}{3}}\Hf_q(m),\] for any positive integer $m$.
Fix an $s$-tuple $(k_1,k_2,\dots, k_s)$, where $1\leq k_1<k_2<\cdots k_s\leq m+s$. Denote by $C:=\{2,4,6,\dots \} \setminus \{2k_1,2k_2,\dots,2k_s\}$ and $D:=\{1,3,5,\dots \} \cup \{2k_1,2k_2,\dots,2k_s\}$. We have
\begin{align}\label{shurlem2eq1}
s_{\lambda(C)}&(q,q^2,\dots,q^{m}) \cdot  s_{\lambda(D)}(q,q^2,\dots,q^{m+d})=\frac{q^{\square(C)+\square(D)}}{\blacktriangle([m])\blacktriangle([m+d])} \cdot \blacktriangle(C)\blacktriangle(D).
\end{align}
We write 
\begin{align}\label{shurlem2eq2}
 \blacktriangle(C)=\dfrac{\displaystyle\prod_{1\le 2i <2j \le 2m+d}(q^{2j}-q^{2i}) \prod_{1\le i < j \le s} (q^{2k_j}-q^{2k_i})}{\displaystyle\prod_{i=1}^s\left(\prod_{1\le 2j<2k_i}(q^{2k_i}-q^{2j})\prod_{2k_i<2j\le2m+d}(q^{2j}-q^{2k_i})\right)},
\end{align}
\begin{align}\label{shurlem2eq3}
\displaystyle \blacktriangle(D)=&\prod_{1\le 2i+1 <2j+1 \le  2m+d}(q^{2j+1}-q^{2i+1}) \prod_{1\le i < j \le s} (q^{2k_j}-q^{2k_i})\notag\\
&\times \prod_{i=1}^s\left( \prod_{1\le 2j+1<2k_i}(q^{2k_i}-q^{2j+1})\prod_{2k_i<2j+1\le2m+d}(q^{2j+1}-q^{2k_i})\right).
\end{align}
Let $E$ and $O$ denote the sets of even and odd numbers in $[2m+d]$, respectively. Multiplying the two identities above, we have %({\color{red}simplify})
\begin{align}\label{shurlem2eq4}
 \blacktriangle(C)\blacktriangle(D)%%=&\prod_{1\le 2i <2j \le 2m+d}(q^{2j}-q^{2i}) \prod_{1\le 2i+1 <2j+1 \le  2m+d}(q^{2j+1}-q^{2i+1})   \prod_{1\le i < j \le s} (q^{2k_j}-q^{2k_i})^2 \notag\\
 %%&\times{\prod_{i=1} ^s}   \dfrac{\displaystyle\prod_{1\le 2j+1<2k_i}(q^{2k_i}-q^{2j+1})\prod_{2k_i<2j+1\le2m+d}(q^{2j+1}-q^{2k_i})}{\displaystyle\prod_{1\le 2j<2k_i}(q^{2k_i}-q^{2j})\prod_{2k_i<2j\le2m+d}(q^{2j}-q^{2k_i})}\notag\\
 =&\blacktriangle(E)\blacktriangle(O)  \prod_{1\le i < j \le s} (q^{2k_j}-q^{2k_i})^2 \notag\\
 &\times \prod_{i=1}^{s} \dfrac{\displaystyle (-1)^{k_i-1} q^{k_i^2}  (q-1) (q^3;q^2)_{k_i-1} \cdot  B}{\displaystyle (-1)^{k_i-1} q^{k_i(k_i-1)}(q^2;q^2)_{k_i-1} (-1)^{m+s-k_i}q^{2k_i(m+s-k_i)}(q^2;q^2)_{m+s-k_i}} 
\end{align}
where
\[
B=\begin{cases}
(-1)^{m+s-k_i+1}q^{2k_i(m+s-k_i)}(q-1) (q^3;q^2)_{m+s-k_i} &\text{ if $d$ is odd;}\\
(-1)^{m+s-k_i}q^{2k_i(m+s-k_i)}(q;q^2)_{m+s-k_i} &\text{ if $d$ is even.}
\end{cases}.
\]
After simplifying, we obtain when $d=2s:$ 
\begin{align}\label{shurlem2eq5}
 \blacktriangle(C)\blacktriangle(D) =&\blacktriangle(E)\blacktriangle(O)    \prod_{1\le i < j \le s} (q^{2k_j}-q^{2k_i})^2 \notag\\
 &\times  \prod_{i=1}^{s} q^{k_i}(q-1) \cdot \dfrac{  (q^{3};q^2)_{k_i-1} (q;q^2)_{m+s-k_i} }{ (q^2;q^2)_{k_i-1}(q^2;q^2)_{m+s-k_i}}, 
 \end{align}
when $d=2s+1:$ 
\begin{align}\label{shurlem2eq6}
 \blacktriangle(C)\blacktriangle(D) =&\blacktriangle(E)\blacktriangle(O)  \prod_{1\le i < j \le s} (q^{2k_j}-q^{2k_i})^2 \notag\\
 &\times \prod_{i=1}^{s} q^{3k_i}(q-1)^2 \cdot \dfrac{  (q^{3};q^2)_{k_i-1} (q^3;q^2)_{m+s-k_i}}{ (q^2;q^2)_{k_i-1}(q^2;q^2)_{m+s-k_i}}, 
\end{align}

Since $C=[2m+d]\setminus D$, we have \[\square(C)+\square(D)=\binom{2m+d+1}{2}-\binom{m+1}{2}-\binom{m+d+1}{2}=m(m+d).\] %Thus
%\[\frac{q^{\square(C)+\square(D)}}{\blacktriangle([m])\blacktriangle([m+d])} \blacktriangle(E)\blacktriangle(O) \]
%does not depend on the choice of the $s$-tuple $(k_1,k_2,\dots, k_s)$. Summing up over all these $s$-tuples, we get
From   (\ref{shurlem2eq1}),  (\ref{shurlem2eq5}), and  (\ref{shurlem2eq6}), we obtain the following when $d$  is even: %({\color{red}simplify})
\begin{align}\label{shurlem2eq7}
&\sum_{1\leq k_1<k_2<\cdots k_s\leq m+s} s_{\lambda(C)}(q,q^2,q^3\dots,q^m) \cdot  s_{\lambda(D)}(q,q^2,q^3\dots,q^{m+d})\\
%%&= \frac{(q-1)^s\cdot q^{m(m+d)}}{\blacktriangle([m])\blacktriangle([m+d])}\blacktriangle(E)\blacktriangle(O) \notag \\
%%&\times \sum_{1\leq k_1<k_2<\cdots k_s\leq m+s}  \prod_{1\le i < j \le s} (q^{2k_j}-q^{2k_i})^2 \prod_{i=1}^{s} q^{k_i} \dfrac{  (q^{3};q^2)_{k_i-1} (q;q^2)_{m+s-k_i} }{ (q^2;q^2)_{k_i-1}(q^2;q^2)_{m+s-k_i}}\notag\\
&= \frac{(q-1)^s\cdot q^{2s(s-1)+s}\cdot q^{m(m+d)}}{\blacktriangle([m])\blacktriangle([m+d])} \blacktriangle(E)\blacktriangle(O)\notag\\
& \times \sum_{0\leq k'_1<k'_2<\cdots k'_s\leq m+s-1}   q^{\sum_{i=1}^{s} k'_i} \cdot  \prod_{1\le i < j \le s} (q^{2k'_j}-q^{2k'_i})^2 \prod_{i=1}^{s}  \dfrac{  (q^{\bf 3};q^2)_{k'_i} (q;q^2)_{(m+s-1)-k'_i} }{ (q^2;q^2)_{k'_i}(q^2;q^2)_{(m+s-1)-k'_i}},\notag
\end{align}
where $k_i':=k_i-1$.
Similarly, when $d$ is odd, we have: %({\color{red}simplify})
\begin{align}\label{shurlem2eq8}
&\sum_{1\leq k_1<k_2<\cdots k_s\leq m+s} s_{\lambda(C)}(q,q^2,q^3\dots,q^m) \cdot  s_{\lambda(D)}(q,q^2,q^3\dots,q^{m+d})\\
%%&= \frac{(q-1)^{2s}\cdot q^{m(m+d)}}{\blacktriangle([m])\blacktriangle([m+d])} \blacktriangle(E)\blacktriangle(O)\notag \\
%%&\times \sum_{1\leq k_1<k_2<\cdots k_s\leq m+s}  \prod_{1\le i < j \le s} (q^{2k_j}-q^{2k_i})^2 \prod_{i=1}^{s} q^{3k_i} \dfrac{  (q^{3};q^2)_{k_i-1} (q^3;q^2)_{m+s-k_i} }{ (q^2;q^2)_{k_i-1}(q^2;q^2)_{m+s-k_i}}\notag\\
&= \frac{(q-1)^{2s}\cdot q^{2s(s-1)+3s}\cdot q^{m(m+d)}}{\blacktriangle([m])\blacktriangle([m+d])}\blacktriangle(E)\blacktriangle(O)\notag\\
&\times \sum_{0\leq k'_1<k'_2<\cdots k'_s\leq m+s-1} (q^3)^{\sum_{i=1}^s k'_i}  \prod_{1\le i < j \le s} (q^{2k'_j}-q^{2k'_i})^2 \prod_{i=1}^{s}  \dfrac{  (q^{3};q^2)_{k'_i} (q^3;q^2)_{(m+s-1)-k'_i} }{ (q^2;q^2)_{k'_i}(q^2;q^2)_{(m+s-1)-k'_i}}.
%&{\color{purple}\text{ replace }\times \sum_{0\leq k'_1<k'_2<\cdots k'_s\leq m+s-1}{(q^3)}^{\sum_{i=1}^s k'_i}  \prod_{1\le i < j \le s} (q^{2k'_j}-q^{2k'_i})^2 \prod_{i=1}^{s}  \dfrac{  (q^{\bf\color{purple}3};q^2)_{k'_i} (q^3;q^2)_{(m+s-1)-k'_i} }{ (q^2;q^2)_{k'_i}(q^2;q^2)_{(m+s-1)-k'_i}}.}
\end{align}

Next, we apply Theorem \ref{thr6} and get the following two identities.
If $d$ is even, say $d=2s$, then
\begin{align}\label{shurlem2eq9}
&\sum_{1\leq k_1<k_2<\cdots k_s\leq m+s} s_{\lambda(C)}(q,q^2,q^3\dots,q^m) \cdot  s_{\lambda(D)}(q,q^2,q^3\dots,q^{m+d})\\
&= \frac{(q-1)^s\cdot q^{2s(s-1)+s}\cdot q^{m(m+d)}}{\blacktriangle([m])\blacktriangle([m+d])} \blacktriangle(E)\blacktriangle(O)\notag\\
&\times q^{2\binom{s}{3}}q^{\binom{s}{2}}\prod_{i=1}^{s} \frac{(q^{3};q^2)_{i-1}(q;q^2)_{i-1}(q^{2(i+s)};q^2)_{m}(q^2;q^2)_{i-1}}{(q^2;q^2)_{m+s-i}}.\notag
\end{align}
If $d$ is odd, say $d=2s+1$, then
\begin{align}\label{shurlem2eq10}
&\sum_{1\leq k_1<k_2<\cdots k_s\leq m+s} s_{\lambda(C)}(q,q^2,q^3\dots,q^m) \cdot  s_{\lambda(D)}(q,q^2,q^3\dots,q^{m+d})\\
&= \frac{(q-1)^{2s}\cdot q^{2s(s-1)+3s}\cdot q^{m(m+d)}}{\blacktriangle([m])\blacktriangle([m+d])}\blacktriangle(E)\blacktriangle(O)\notag\\
&\times q^{2\binom{s}{3}}q^{3\binom{s}{2}}\prod_{i=1}^{s} \frac{(q^{3};q^2)_{i-1}(q^3;q^2)_{i-1}(q^{2(i+s+1)};q^2)_{m}(q^2;q^2)_{i-1}}{(q^2;q^2)_{m+s-i}}.\notag
\end{align}
We note that if $d=2s$, then
\begin{align}\label{shurlem2eq11}
\blacktriangle(E)\blacktriangle(O)&=\prod_{1\leq i<j\leq m+s}(q^{2j}-q^{2i}) \cdot  \prod_{1\leq i<j\leq m+s}(q^{2j-1}-q^{2i-1})\\
&=\prod_{1\leq i<j\leq m+s}(q^{2j}-q^{2i}) \cdot  \prod_{1\leq i<j\leq m+s}q^{-1}(q^{2j}-q^{2i})\notag\\
&=q^{-\binom{m+s}{2}} \prod_{1\leq i<j\leq m+s}(q^{2j}-q^{2i})^2\notag\\
&=q^{-\binom{m+s}{2}}\cdot  q^{4\binom{m+s+1}{3}} \Hf_{q^2}(m+s)^2. \notag
\end{align}
If $d=2s+1$, then 
\begin{align}\label{shurlem2eq12}
\blacktriangle(E)\blacktriangle(O)&=\prod_{1\leq i<j\leq m+s}(q^{2j}-q^{2i}) \cdot  \prod_{1\leq i<j\leq m+s+1}(q^{2j-1}-q^{2i-1})\\
&=q^{-\binom{m+s}{2}} \prod_{1\leq i<j\leq m+s}(q^{2j}-q^{2i})^2 \cdot \prod_{i=1}^{m+s} (q^{2m+2s+1}-q^{2i-1})\notag\\
&=q^{-\binom{m+s}{2}}\cdot \left( q^{2\binom{m+s+1}{3}} \Hf_{q^2}(m+s)\right)^2 \cdot (-1)^{m+s} q^{(m+s)^2} (q^2;q^2)_{m+s}\notag\\
&=(-1)^{m+s}q^{\binom{m+s+1}{2}} q^{4\binom{m+s+1}{3}} \Hf_{q^2}(m+s) \Hf_{q^2}(m+s+1).\notag
\end{align}
The lemma now follows from (\ref{shurlem2eq9}) and (\ref{shurlem2eq11}) for the case when $d=2s$, and from (\ref{shurlem2eq10}) and (\ref{shurlem2eq12}) for the case when $d=2s+1$.
\end{proof}

\subsection{Proof of Lemma \ref{Schurlem}}

\begin{proof}[Proof of Lemma \ref{Schurlem}]
We have
\begin{align}\label{Schureq0}
&\frac{s_{\lambda(A')}(q,q^2,\dots, q^{m+n+a})s_{\lambda(B')}(q,q^2,\dots, q^{m+n+b})}{s_{\lambda(A)}(q,q^2,\dots, q^{a})s_{\lambda(B)}(q,q^2,\dots, q^{b})}\notag\\
&=q^{\square(A')+\square(B')-\square(A)-\square(B)}\frac{\blacktriangle(A')\blacktriangle(B')}{\blacktriangle(A)\blacktriangle(B)}\frac{\blacktriangle([a])\blacktriangle([b])}{\blacktriangle([m+n+a])\blacktriangle([m+n+b])}.
\end{align}

Now consider the exponent of $q$ on the right-hand side. We note that \[\square(U)=\sum_{u\in U} u-\binom{|U|+1}{2}.\] Since $A=[a+b]\setminus B$, we have %({\color{red}simplify})
\begin{align}\label{Schureq1}
\square&(A')+\square(B')-\square(A)-\square(B)\notag\\
=&\sum_{i\in A'} i+\sum_{i\in B'}i-\sum_{i\in A}i-\sum_{i\in B}i\notag\\
&\quad-\binom{m+n+a+1}{2}-\binom{m+n+b+1}{2}+\binom{a+1}{2}+\binom{b+1}{2}\notag\\
%%=&\left(\sum_{i\in A' \cup B'} i+ \sum_{i\in A' \cap B'} i\right)-\sum_{i\in A\cup B} i\notag\\
%%&\quad-\binom{m+n+a+1}{2}-\binom{m+n+b+1}{2}+\binom{a+1}{2}+\binom{b+1}{2}\notag\\
%%=&\left(\binom{m+n+a+b+1}{2}+\binom{m+1}{2}+\binom{m+n+a+b+1}{2}-\binom{m+a+b+1}{2}\right)\notag\\&\quad -\binom{{a+b}+1}{2}
%%-\binom{m+n+a+1}{2}-\binom{m+n+b+1}{2}+\binom{a+1}{2}+\binom{b+1}{2} \notag\\
=&n(a+b).
\end{align}
Next, we recall that \[\blacktriangle([m])=\prod_{1\le i < j \le m}(q^{j}-q^{i})=(-1)^{\binom{m}{2}}q^{\binom{m+1}{3}}\Hf_q(m),\] for any positive integer $m$.
Thus
\begin{align} \label{Schureq2}
\frac{\blacktriangle([a])\blacktriangle([b])}{\blacktriangle([m+n+a])\blacktriangle([m+n+b])}&=(-1)^{(a+b)(m+n)} q^{\binom{a+1}{3}+\binom{b+1}{3}-\binom{m+n+a+1}{3}-\binom{m+n+b+1}{3}}\notag\\
&\times \frac{\Hf_q(a) \Hf_q(b)}{\Hf_q(m+n+a) \Hf_q(m+n+b)}.
\end{align}
%{\color{purple}in 4.21, replace $(-1)^{(a+b)(m+n)}$ with $(-1)^{-(m+n)(m+n+a+b-1)}$:Because the exponents have the same parity.}\\
Finally, we simplify
$\frac{\blacktriangle(A')\blacktriangle(B')}{\blacktriangle(A)\blacktriangle(B)}.$ We note that $\blacktriangle(k+U)=q^{k\cdot \binom{|U|}{2} }\blacktriangle(U)$. %{\color{purple} replace with $\blacktriangle(k+{\color{purple}U})=q^{k\binom{|U|}{2} }\blacktriangle(U)$. Good catch!}.
Assume that $A=\{x_1<x_2<\cdots<x_a\}$ and $B=\{y_1<y_2<\cdots<y_b\}$. We write 
\begin{align}
\blacktriangle(A')&=\blacktriangle([m])\cdot \blacktriangle(m+A)\cdot \blacktriangle (m+a+b+[n])\notag\\
&\times \prod_{\substack{1\le i \le m\\ 1\leq j \leq a}} (q^{m+x_j}-q^i)\cdot \prod_{\substack{1\le i \le a\\ 1\leq j \leq n}} (q^{m+a+b+j}-q^{m+x_i}) \cdot \prod_{\substack{1\le i \le m\\ 1\leq j \leq n}}(q^{m+a+b+j}-q^{i})\notag\\
&=q^{m\binom{a}{2}+(m+a+b)\binom{n}{2}}\blacktriangle([m]) \cdot \blacktriangle(A) \cdot\blacktriangle([n])\notag\\
&\times\prod_{\substack{1\le i \le m\\ 1\leq j \leq a}} (q^{m+x_j}-q^i)\cdot \prod_{\substack{1\le i \le a\\ 1\leq j \leq n}} (q^{m+a+b+j}-q^{m+x_i}) \cdot \prod_{\substack{1\le i \le m\\ 1\leq j \leq n}}(q^{m+a+b+j}-q^{i}),
\end{align}
%{\color{purple} in 4.22, replace $q^{ma+(m+d)n}$ with $q^{m\binom{a}{2}+(m+a+b)\binom{n}{2}}$}\\
doing similarly for $B',$ 
\begin{align}
	\blacktriangle(B')	&=q^{m\binom{b}{2}+(m+a+b)\binom{n}{2}}\blacktriangle([m]) \cdot \blacktriangle(B) \cdot\blacktriangle([n])\notag\\
	&\times\prod_{\substack{1\le i \le m\\ 1\leq j \leq b}} (q^{m+y_j}-q^i)\cdot \prod_{\substack{1\le i \le b\\ 1\leq j \leq n}} (q^{m+a+b+j}-q^{m+y_i}) \cdot \prod_{\substack{1\le i \le m\\ 1\leq j \leq n}}(q^{m+a+b+j}-q^{i}).
	\end{align} 
Observe that %({\color{red}simplify})
\begin{align*}
\prod_{\substack{1\le i \le m\\ 1\leq j \leq a+b}} (q^{m+j}-q^i)&=\prod_{j=1}^{a+b}\prod_{i=1}^{m}q^{i}(q^{m-i+j}-1)\\
%%&=\prod_{j=1}^{a+b} q^{\binom{m+1}{2}} \prod_{i=1}^{m}(q^{m-i+j}-1)\\
%%&=q^{(a+b)\binom{m+1}{2}}\prod_{j=1}^{a+b} (q^{m+j-1}-1)(q^{m+j-2}-1)\cdots (q^{j}-1)\\
%%&=q^{(a+b)\binom{m+1}{2}}\prod_{j=1}^{a+b} (-1)^{m} (q^j;q)_m\\
&=q^{(a+b)\binom{m+1}{2}}(-1)^{(a+b)m}\prod_{j=1}^{a+b}(q^j;q)_m.
\end{align*}
Similarly, we have %({\color{red}simplify})
\begin{align*}
\prod_{\substack{1\le i \le a+b\\ 1\leq j \leq n}} (q^{m+a+b+j}-q^{m+i})&=\prod_{i=1}^{a+b} \prod_{j=1}^{n}q^{m+i}(q^{a+b-i+j}-1)\\
%%&=\prod_{i=1}^{a+b} q^{(m+i)n} \prod_{j=1}^{n}(q^{a+b-i+j}-1)\\
%%&=q^{mn(a+b)+n\binom{a+b+1}{2}}\prod_{i=1}^{a+b} (q^{a+b-i+1}-1)(q^{a+b-i+2}-1)\cdots (q^{a+b-i+n}-1)\\
%%&=q^{mn(a+b)+n\binom{a+b+1}{2}}\prod_{i=1}^{a+b} (-1)^{n} (q^{a+b-i+1};q)_n\\
&=q^{mn(a+b)+n\binom{a+b+1}{2}}(-1)^{(a+b)n}\prod_{i=1}^{a+b} (q^{a+b-i+1};q)_n
\end{align*}
and %({\color{red}simplify})
\begin{align*}
\prod_{\substack{1\le i \le m\\ 1\leq j \leq n}} (q^{m+a+b+j}-q^{i})&=\prod_{j=1}^{n} \prod_{i=1}^{m}q^{i}(q^{m+a+b-i+j}-1)\\
%%&=\prod_{j=1}^{n} q^{\binom{m+1}{2}} \prod_{i=1}^{m}(q^{m+a+b-i+j}-1)\\
%%&=q^{n\binom{m+1}{2}}\prod_{j=1}^{n} (q^{m+a+b+j-1}-1)(q^{m+a+b+j-2}-1)\cdots (q^{a+b+j}-1)\\
%%&=q^{n\binom{m+1}{2}}\prod_{j=1}^{n} (-1)^{m} (q^{a+b+j};q)_m\\
&=q^{n\binom{m+1}{2}}(-1)^{mn}\prod_{j=1}^{n} (q^{a+b+j};q)_m.
\end{align*}   
Since $A=[a+b]\setminus B$, we now may write $\frac{\blacktriangle(A')\blacktriangle(B')}{\blacktriangle(A)\blacktriangle(B)}$ as follows. %({\color{red}simplify})
\begin{align} \label{Schureq3}
\frac{\blacktriangle(A')\blacktriangle(B')}{\blacktriangle(A)\blacktriangle(B)}
%=q^{m\binom{a}{2}+m\binom{b}{2}+2(m+a+b)\binom{n}{2}}\blacktriangle([m])^2\blacktriangle([n])^2\notag\\
%&\quad \quad \times \left( \prod_{\substack{1\le i \le m\\ 1\leq j \leq a}} (q^{m+{x}_j}-q^i)  \prod_{\substack{1\le i \le m\\ 1\leq j \leq b}} (q^{m+{y}_j}-q^i) \right)\notag\\
%&\quad \quad \times  \left( \prod_{\substack{1\le i \le a\\ 1\leq j \leq n}}  (q^{m+a+b+j}-q^{m+{x}_i}) \prod_{\substack{1\le i \le b\\ 1\leq j \leq n}}  (q^{m+a+b+j}-q^{m+{y}_i}) \right)\notag\\
%&\quad \quad \times \left(\prod_{\substack{1\le i \le m\\ 1\leq j \leq n}} (q^{m+a+b+j}-q^{i})\right)^2\notag\\
&\quad =q^{m\binom{a}{2}+m\binom{b}{2}+2(m+a+b)\binom{n}{2}}\blacktriangle([m])^2\blacktriangle([n])^2 \cdot \prod_{\substack{1\le i \le m\\ 1\leq j \leq a+b}} (q^{m+j}-q^i)\notag \\
&\quad \quad \times   \prod_{\substack{1\le i \le a+b\\ 1\leq j \leq n}} (q^{m+a+b+j}-q^{m+i}) \cdot  \prod_{\substack{1\le i \le m\\ 1\leq j \leq n}} (q^{m+a+b+j}-q^{i})^2\notag\\
%%&{=q^{m\binom{a}{2}+m\binom{b}{2}+2(m+a+b)\binom{n}{2}}q^{2\binom{m+1}{3}+2\binom{n+1}{3}}\Hf_q(m)^2\Hf_q(n)^2 } \notag \\
%%	&\quad \quad \times   \prod_{\substack{1\le i \le m\\ 1\leq j \leq a+b}} (q^{m+j}-q^i)\cdot \prod_{\substack{1\le i \le a+b\\ 1\leq j \leq n}} (q^{m+a+b+j}-q^{m+i}) \cdot  \prod_{\substack{1\le i \le m\\ 1\leq j \leq n}} (q^{m+a+b+j}-q^{i})^2 \notag\\
 &=q^{m\binom{a}{2}+m\binom{b}{2}+2(m+a+b)\binom{n}{2}}q^{2\binom{m+1}{3}+2\binom{n+1}{3}} \Hf_q(m)^2\Hf_q(n)^2 \notag\\
&\times q^{(a+b)\binom{m+1}{2}}(-1)^{(a+b)m}\prod_{j=1}^{a+b}(q^j;q)_m \notag\\
&\times q^{mn(a+b)+n\binom{a+b+1}{2}}(-1)^{(a+b)n}\prod_{i=1}^{a+b} (q^{a+b-i+1};q)_n\cdot  q^{m(m+1)n} \prod_{j=1}^{n} (q^{a+b+j};q)_m^2.
\end{align}
Combining (\ref{Schureq0}), (\ref{Schureq1}), (\ref{Schureq2}), and (\ref{Schureq3}), we obtain the desired identity after a routine simplification.

\end{proof}
\bibliographystyle{plain}

\bibliography{TriRevision.bib}

\end{document}